\documentclass[a4paper,10pt]{amsart}
\usepackage{tikz}
\usetikzlibrary{positioning,arrows,calc,patterns,fadings,cd}
\usepackage[final,ocgcolorlinks,pdfpagelabels]{hyperref}
\definecolor{darkred}{rgb}{.6,0,0}
\definecolor{darkgreen}{rgb}{0,.5,0}
\hypersetup{
    linkcolor = darkred,
    citecolor = darkgreen
}

\usepackage{amssymb,amsthm,mathtools}
\usepackage{eucal}
\usepackage{subcaption}
\usepackage[final]{microtype}
\usepackage[shortlabels]{enumitem}

\usepackage[nameinlink,noabbrev,capitalize]{cleveref}

\numberwithin{equation}{section}
\newtheorem{thm}[equation]{Theorem}
\newtheorem{lemma}[equation]{Lemma}
\newtheorem{cor}[equation]{Corollary}
\newtheorem{prop}[equation]{Proposition}
\theoremstyle{definition}
\newtheorem{ex}[equation]{Example}
\newtheorem{defn}[equation]{Definition}
\theoremstyle{remark}
\newtheorem{remark}[equation]{Remark}
\newtheorem{construction}[equation]{Construction}

\Crefname{thm}{Theorem}{Theorems}
\Crefname{prop}{Proposition}{Propositions}
\Crefname{ex}{Example}{Examples}
\Crefname{cor}{Corollary}{Corollaries}

\usepackage{enumitem}

\DeclareMathOperator{\coker}{coker}
\DeclareMathOperator{\im}{im}
\DeclareMathOperator{\id}{id}

\DeclareMathOperator{\Hom}{Hom}
\DeclareMathOperator{\End}{End}
\DeclareMathOperator{\Ext}{Ext}
\DeclareMathOperator{\modcat}{mod} \renewcommand{\mod}{\modcat}
 
\DeclareMathOperator{\setcat}{set}
\DeclareMathOperator{\Ab}{Ab}
\DeclareMathOperator{\add}{add}
\DeclareMathOperator{\rep}{rep}
\DeclareMathOperator{\Rep}{Rep}

\DeclareMathOperator{\inj}{inj}
\DeclareMathOperator{\Inj}{Inj}
\DeclareMathOperator{\proj}{proj}
\DeclareMathOperator{\Proj}{Proj}

\DeclareMathOperator{\pdim}{pdim}
\DeclareMathOperator{\idim}{idim}
\DeclareMathOperator{\Ob}{Ob}

\DeclareMathOperator{\Fac}{Fac}
\DeclareMathOperator{\Sub}{Sub}

\DeclareMathOperator{\modni}{mod_{non-inj}}

\def\hopen#1,#2{[#1,#2)}

\def\sfI{\mathsf{I}}
\def\sfP{\mathsf{P}}

\def\sfRI{\mathsf{RI}}

\def\mcA{\mathcal{A}}
\def\mcB{\mathcal{B}}
\def\mcC{\mathcal{C}}
\def\mcD{\mathcal{D}}
\def\mcE{\mathcal{E}}
\def\mcF{\mathcal{F}}
\def\mcS{\mathcal{S}}
\def\mcT{\mathcal{T}}
\def\mcX{\mathcal{X}}

\def\mbC{{\mathbb C}}

\def\mbR{{\mathbb R}}

\def\mbT{{\mathbb T}}
\def\mbX{{\mathbb X}}

\newcommand{\ctt}{\mathsf{c}}
\newcommand{\ctf}{\mathsf{d}}
\newcommand{\tors}{\mathsf{t}}
\newcommand{\tfree}{\mathsf{f}}
\newcommand{\e}{\mathsf{e}}
\newcommand{\m}{\mathsf{m}}

\newcommand{\tensor}{\otimes}
\newcommand{\isom}{\cong}

\newcommand{\op}{\mathsf{op}}

\renewcommand{\to}{\longrightarrow}
\newcommand{\shortto}{\rightarrow}
\newcommand{\into}{\lhook\joinrel\to}
\newcommand{\onto}{\to\mathrel{\mspace{-25mu}}\to}

\newcommand{\field}{\mathbf{k}}

\definecolor{lightergray}{gray}{0.90}

\newcommand\tikzpo[2]
{\begin{scope}[shift=($(#1)!.8!(#2)$)] \draw +(0,.25) -- +(-.25,.25)  -- +(-.25,0); \end{scope}}
\newcommand\tikzpbadhoc[2]
{\begin{scope}[shift=($(#1)!.2!(#2)$)] \draw +(0,.25) -- +(.25,.25)  -- +(.25,0); \end{scope}}

\title
[Representation theory of filtered hierarchical clustering]
{Cotorsion torsion triples and the representation theory of filtered hierarchical clustering}

\author[U.~Bauer]{Ulrich Bauer}
\address{Department of Mathematics, Technical University of Munich, Germany}
\email{mail@ulrich-bauer.org}

\author[M.B.~Botnan]{Magnus B. Botnan}
\address{Department of Mathematics, VU University Amsterdam, The Netherlands}
\email{m.b.botnan@vu.nl}

\author[S.~Oppermann]{Steffen Oppermann}
\address{Department of Mathematical Sciences, NTNU, Trondheim, Norway}
\email{steffen.oppermann@ntnu.no}

\author[J.~Steen]{Johan Steen}
\address{Mathematics Department, UC Santa Cruz, USA}
\email{jsteen1@ucsc.edu}



\begin{document}

\begin{abstract}
    We give a full classification of representation types of the subcategories of representations of an $m \times n$ rectangular grid with monomorphisms (dually, epimorphisms) in one or both directions, which appear naturally in the context of clustering as two-parameter persistent homology in degree zero.
    We show that these subcategories are equivalent to the category of all representations of a smaller grid, modulo a finite number of indecomposables.
    This equivalence is constructed from a certain cotorsion torsion triple, which is obtained from a tilting subcategory generated by said indecomposables.
\end{abstract}

\maketitle


\section{Introduction}

Clustering analysis encompasses a wide range of statistical methods for inferring structure in data. Loosely speaking, a clustering method aims to partition the data into clusters such that data of different clusters are significantly more different than data belonging to the same cluster. Such methods play an important role in unsupervised data analysis.

In order to provide context for our results, we first discuss why clustering methods are commonly based on the choice of two parameters, explaining the relevance for studying two parameter grid representations, and also explaining why the special case of epimorphisms in one parameter direction arises from this scenario.

Define a \emph{clustering method} to be a map $\mathcal{C}$ which associates to every finite metric space $(M,d)$ a surjective set map $\mathcal{C}(M,d)\colon M \onto C(M,d)$ from the points of $M$ to a set of \emph{clusters} $C(M,d)$. We say that two elements $m$ and $m'$ in $M$ are \emph{clustered} (with respect to $\mathcal{C}$) if $\mathcal{C}(M,d)(m) = \mathcal{C}(M,d)(m')$.  We write $\mathcal{C}(M,d) \cong \mathcal{C}(M,d')$ if \[\mathcal{C}(M,d)(m)  =\mathcal{C}(M,d)(m') \iff \mathcal{C}(M,d')(m) = \mathcal{C}(M,d')(m').\]
\begin{ex}
    Fix an $\varepsilon\geq 0$ and define the \emph{geometric graph at scale $\varepsilon$}, denoted by $\mathcal{G}_\varepsilon(M,d)$, to be the graph on $M$ with an edge connecting $m$ and $m'$ if and only if $d(m,m')\leq \varepsilon$. The canonical surjection $\mathcal{S}_\varepsilon(M,d) \colon M \onto \pi_0\big(\mathcal{G}_\varepsilon(M,d)\big)$ to the connected components of the geometric graph defines a clustering method.
    If $M$ is a subspace of some Euclidean space (or, more generally, of a length metric space), then an equivalent clustering method is given by considering the connected components of a union of closed balls, $M\to \pi_0\big(\bigcup_{x \in M}B_{\varepsilon(x)/2}\big)$, as illustrated in \cref{fig:zoomin}.
    \label{ex1}
\end{ex}

Given the abundance of different clustering techniques, it is natural to ask what kind of properties a clustering method may satisfy. Consider the following two desirable properties of a clustering method:

\begin{itemize}[leftmargin=*]
    \item \textbf{Scale invariance}: For all $\alpha>0$, $\mathcal{C}(M,d) \cong \mathcal{C}(M,\alpha\cdot d)$.
    \item \textbf{Consistency}: For any two metrics $d$ and $d'$ on $M$ satisfying
    \begin{itemize}
        \item $d'(x,y) \geq d(x,y)$ if $\mathcal{C}(M,d)(x)\neq \mathcal{C}(M,d)(y)$, and
        \item $d'(x,y) \leq d(x,y)$ if $\mathcal{C}(M,d)(x)=\mathcal{C}(M,d)(y)$,
    \end{itemize}
    we have $\mathcal{C}(M,d) \cong \mathcal{C}(M,d')$.
\end{itemize}
It is not hard to see that the clustering method $\mathcal{S}_\varepsilon$ is consistent but not scale invariant, whereas a normalized version of $\mathcal{S}_\varepsilon$ would be scale invariant but not consistent. An immediate consequence of a theorem by Kleinberg \cite[Theorem 3.1]{kleinberg2003impossibility} is that an isometry invariant clustering method simultaneously satisfying scale invariance and consistency must be trivial, either always returning a single cluster, or always returning a separate cluster for each point.

An implication of this is that (non-trivial) clustering methods are inherently unstable; a small perturbation of the input data may produce a vastly different graph. Furthermore, there might be no unique correct scale at which the data should be considered. Indeed, \cref{fig:zoomin} illustrates that what appears to be well-defined clusters at one scale, may reveal a finer structure upon inspection at a smaller scale.

\begin{figure}
    \centering
    \includegraphics[scale=0.8]{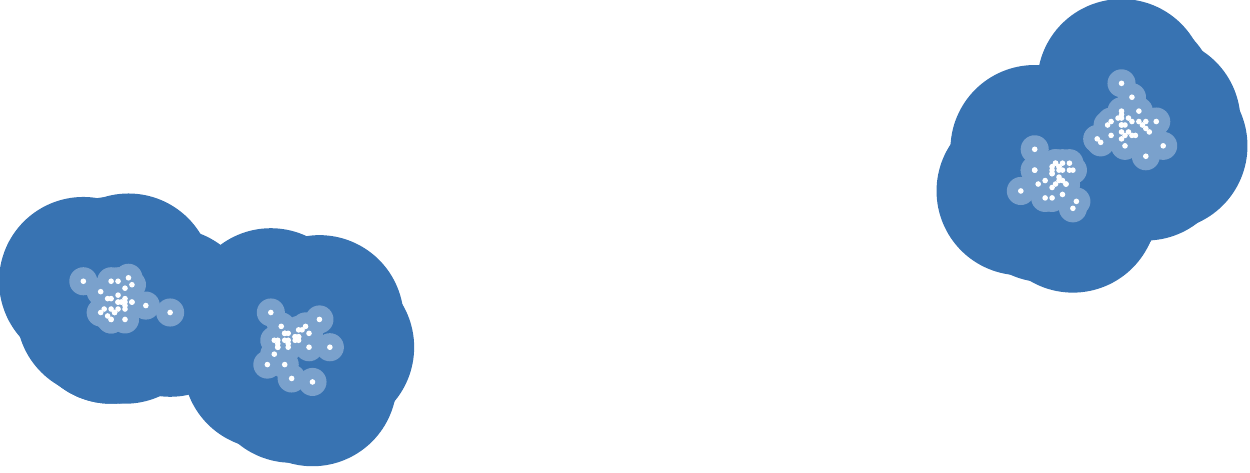}
    \caption{Example illustrating the multiscale nature of the clustering problem.  For single linkage clustering,  the clusters correspond to the connected components of a unions of balls.  At a coarse scale, two clusters (dark shading) are apparent. At a finer scale, each of these two clusters appear to decompose further into two subclusters (light shading).}
    \label{fig:zoomin}
\end{figure}

One may attempt to rectify these issues by considering \emph{hierarchical clustering methods}. Such methods do not assign a single set map to the input metric space, but rather a one-parameter family of surjective set maps $\{\mathcal{C}_\varepsilon(M,d)\colon M \onto C_\varepsilon(M,d)\}_{\varepsilon\geq0}$,
where the codomain $C_\varepsilon(M,d)$ is the set of clusters at scale $\varepsilon$,
such that
\[
    \mathcal{C}_\varepsilon(M,d)(m_1) = \mathcal{C}_\varepsilon(M,d)(m_2) \implies \mathcal{C}_{\varepsilon'}(M,d)(m_1) = \mathcal{C}_{\varepsilon'}(M,d)(m_2)
\]
for all $\varepsilon\leq \varepsilon'$.

The output of a hierarchical clustering method applied to a metric space can be visualized by means of a rooted tree called a \emph{dendrogram}, see \cref{fig:dendrogram} for an example.

\begin{figure}
    \centering
    \includegraphics[scale=0.4]{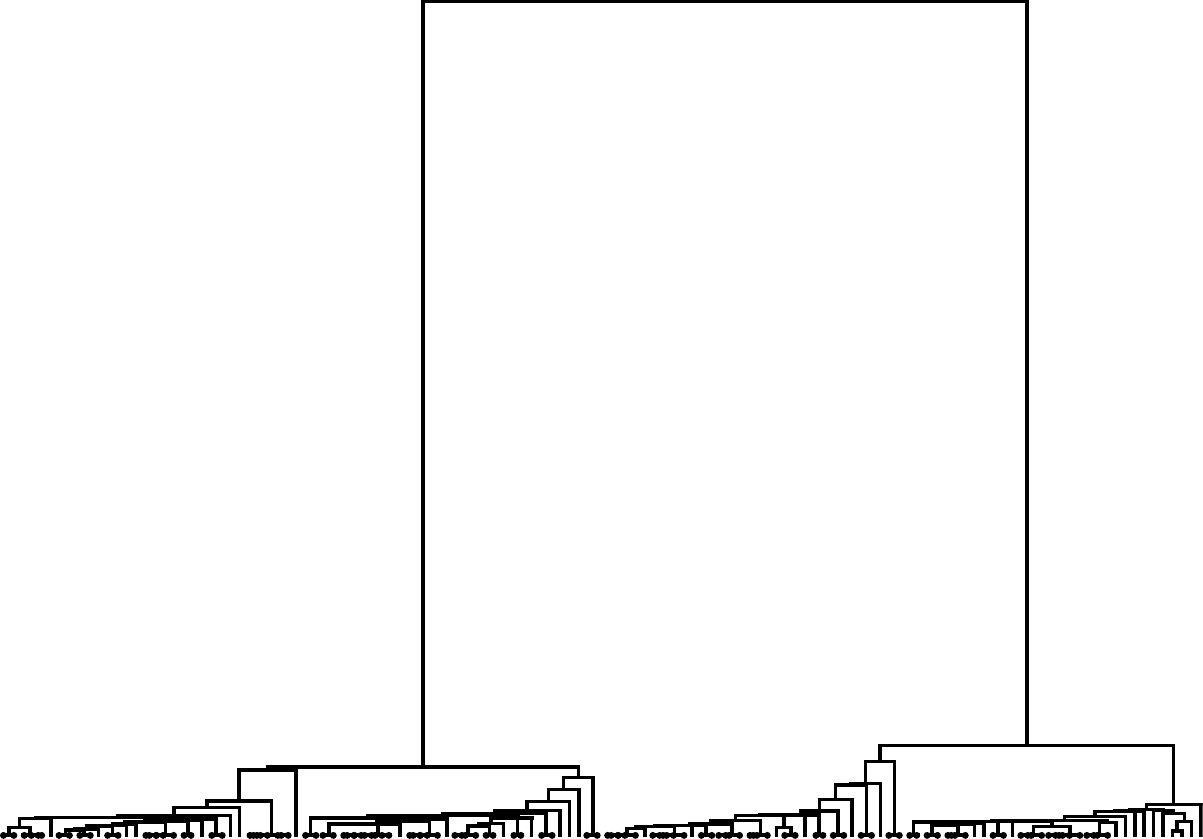}
    \caption{Dendrogram for the point set shown in \cref{fig:zoomin}.}
    \label{fig:dendrogram}
\end{figure}

\begin{ex}
    \label{ex:singlelinkage}
    The clustering method from \cref{ex1} based on geometric graphs $\mathcal{G}_\varepsilon(M,d)$ extends naturally to a hierarchical clustering method. Indeed, observe that all $\varepsilon\leq \varepsilon'$, the inclusion of graphs $\mathcal{G}_\varepsilon(M,d) \into \mathcal{G}_{\varepsilon'}(M,d)$ induces a surjection $\pi_0\circ\mathcal{G}_\varepsilon(M,d) \onto \pi_0\circ\mathcal{G}_{\varepsilon'}(M,d)$ of connected components.
    Thus, assigning to a metric space $(M,d)$ the family of canonical epimorphisms $\{\mathcal{S}_{\varepsilon}(M,d) \colon M \to \pi_0\circ \mathcal{G}_{\varepsilon}(M,d)\}_{\varepsilon \geq 0}$ yields a hierarchical clustering method, commonly referred to as \emph{single linkage hierarchical clustering}. For other hierarchical methods such as \emph{average} and \emph{complete linkage hierarchical clustering}, see, e.g., \cite{Rokach2005}.
\end{ex}

Carlsson and Memoli \cite{MR3032681} give examples showing that average and complete linkage clustering are unstable with respect to perturbation of the input data. Intuitively, this means that a small distortion to the input data may produce a vastly different rooted tree. In contrast, they show that single linkage clustering is stable with respect to the Gromov--Hausdorff distance. Furthermore, they go on to classify single linkage clustering as the \emph{unique functorial} hiearchical clustering method (under mild additional assumptions; see \cite[Theorem 7.1]{MR3032681} for the precise statement). In practice, however, other hierarchical methods are preferred over the single linkage, as it suffers from the so-called \emph{chaining effect}. Intuitively, this means that single linkage clustering may fail to detect two dense regions connected by regions of low density.

One may attempt to rectify this by considering the hierarchical clustering at a specific density threshold. However, similar to fixing a geometric scale above, such choices would lead to instability in the method. Therefore we will consider clustering methods which offer a multi-scale view of both density and scale. This approach to clustering was first considered in \cite{MR2722123}.

\begin{figure}
    \centering
    {
        \renewcommand{\arraystretch}{2}
        \begin{tabular}{cc}
            \includegraphics[scale=0.4]{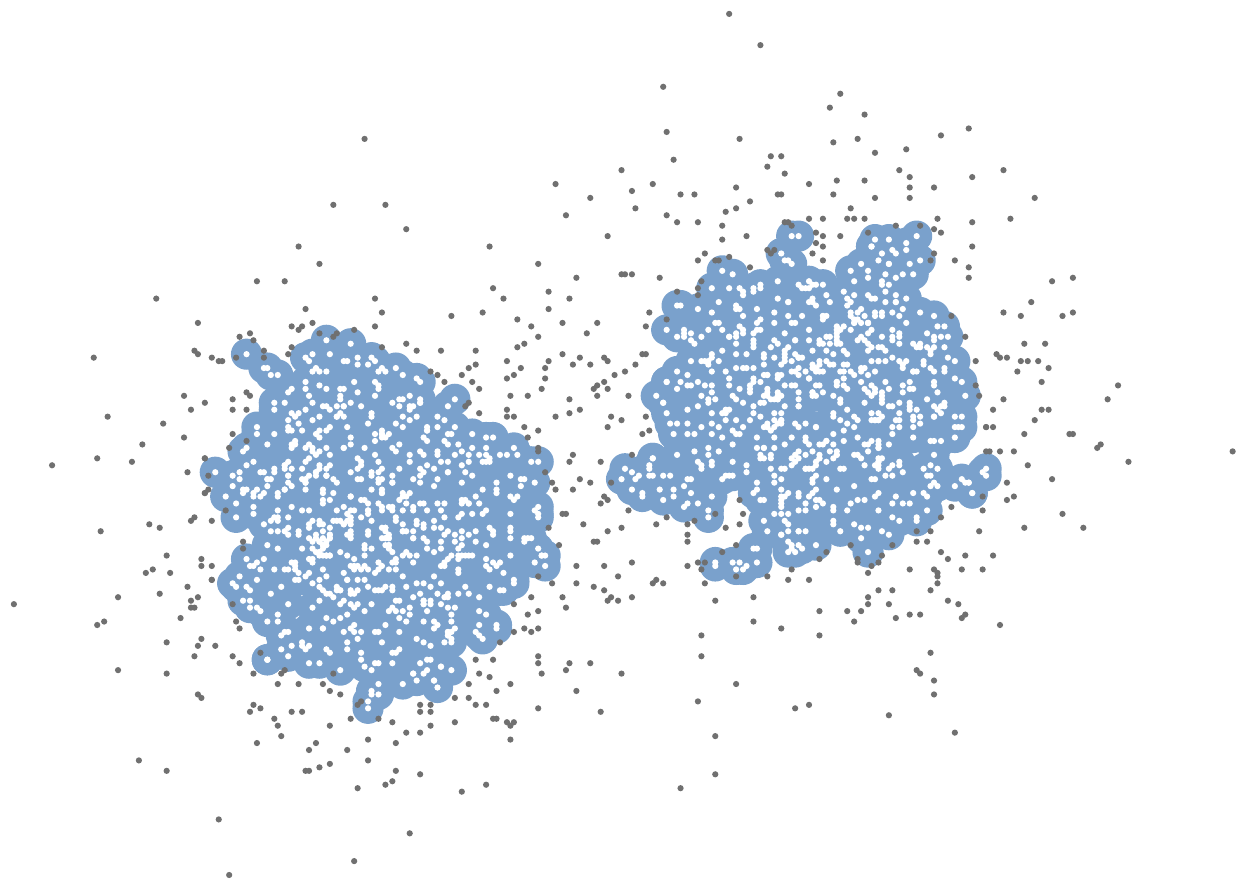}
            &
            \includegraphics[scale=0.4]{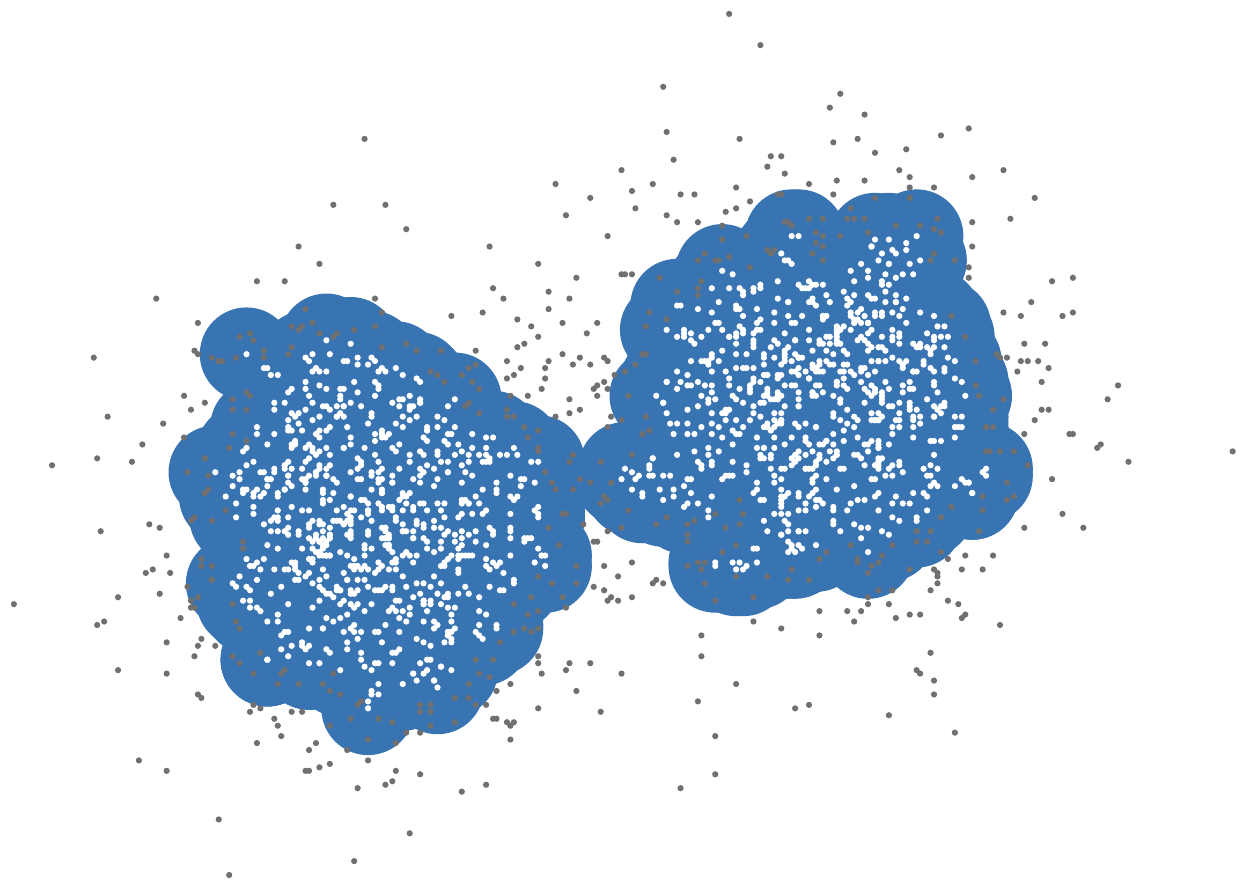}
            \\
            \includegraphics[scale=0.4]{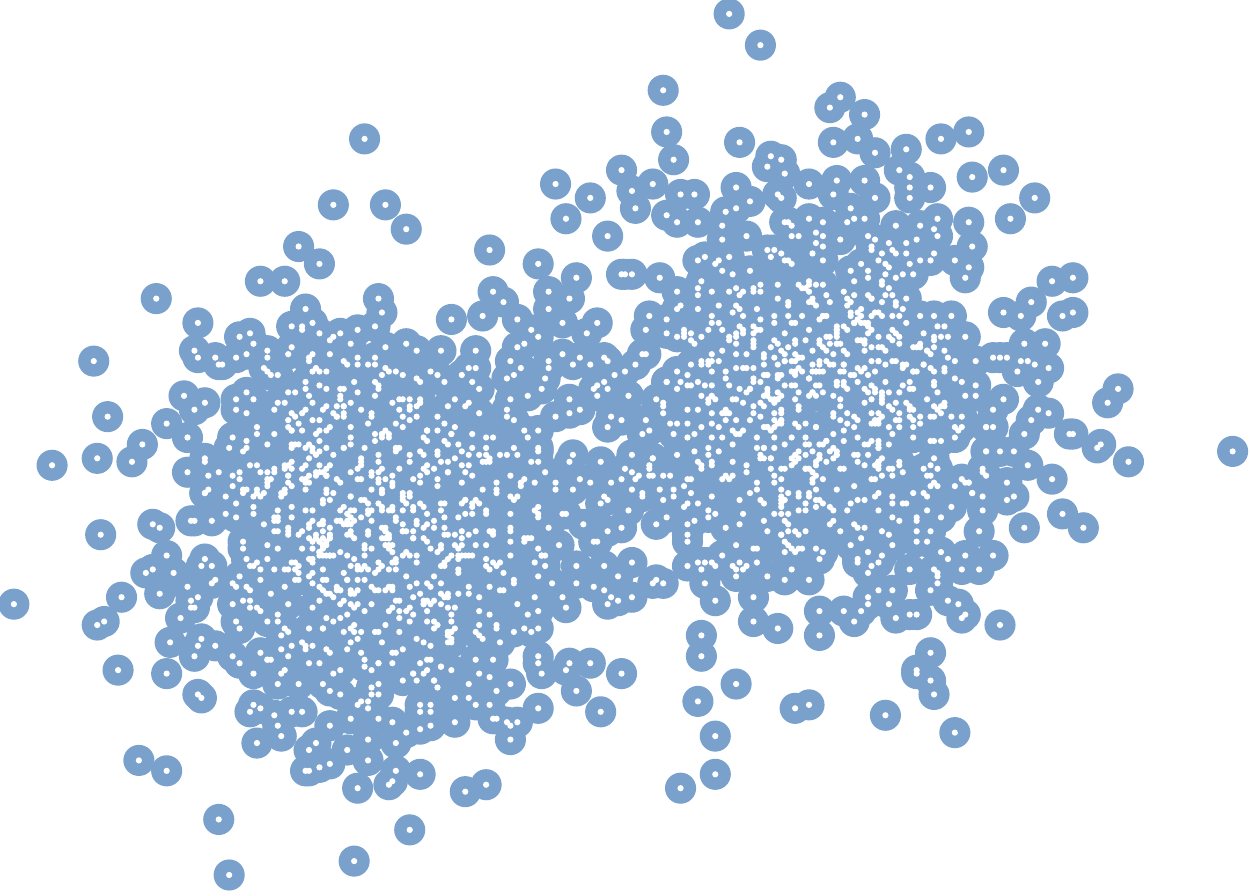}
            &
            \includegraphics[scale=0.4]{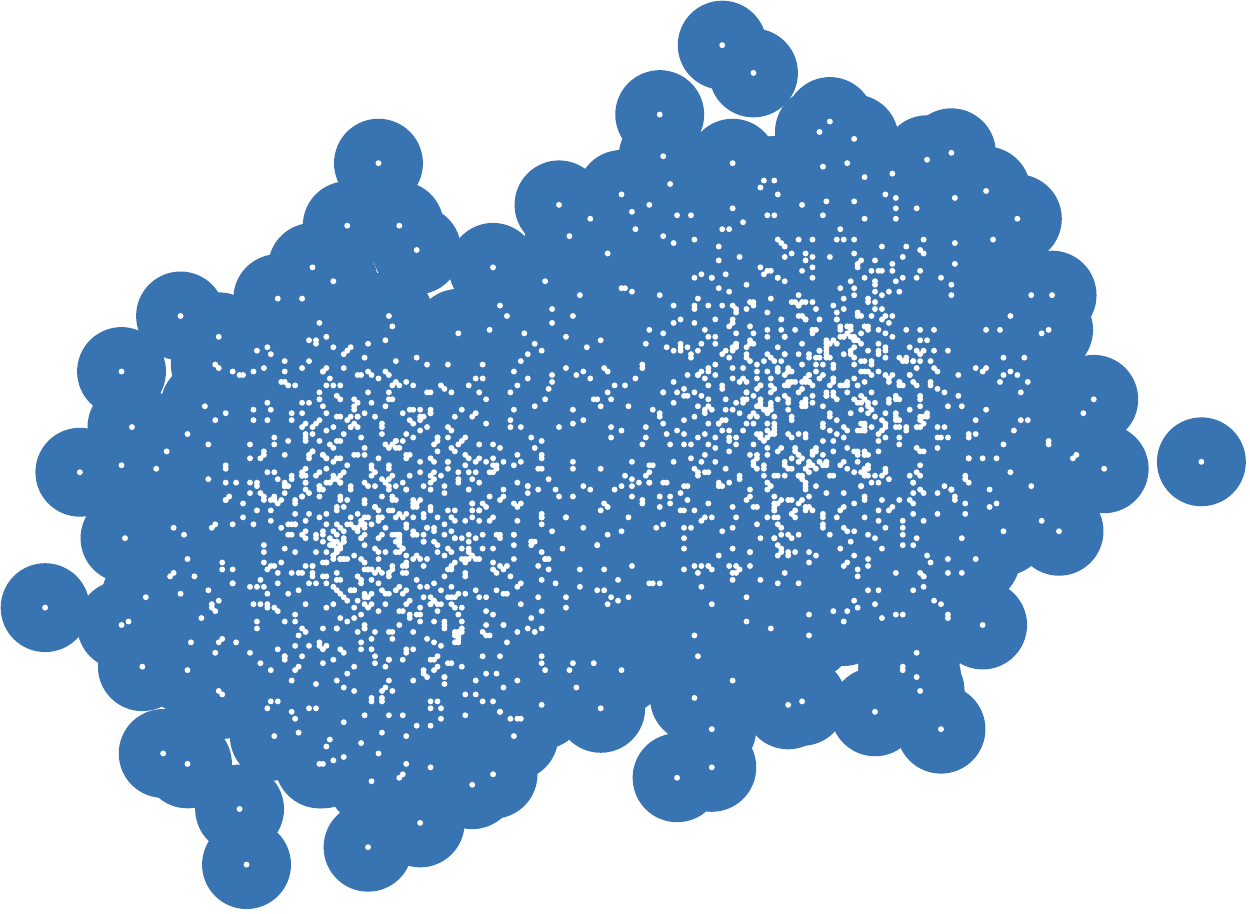}
        \end{tabular}
    }
    \caption{Example illustrating hierarchical clustering method for filtered spaces.  Points are drawn from a mixture of two Gaussians, and filtered using a kernel density estimator.  The horizontal arrangement corresponds to geometric scale, while the vertical arrangement corresponds to the density threshold.}
    \label{fig:densityclustering}
\end{figure}

Define a \emph{filtration} of a metric space $(M,d)$ to be a family of subspaces $\{(M_\delta, d)\}_{\delta\geq0}$ such that $M_\delta\subseteq M_{\delta'}$ for $\delta\leq \delta'$.
For example, let $\sigma\colon M\to [0, \infty)$ be any function (e.g., a density estimate) and define a filtration of $M$ by letting $M_\delta = \sigma^{-1}[\delta, \infty)$.

A \emph{filtered hierarchical clustering method} is a map which associates to every $M_\delta$ in a filtration of $M$ a one-parameter family of surjective set maps $\{\mathcal{C}_\varepsilon(M_\delta,d)\colon M_\delta \onto C_{\varepsilon,\delta}\}_{\varepsilon\geq0}$ such that
\[
    \mathcal{C}_\varepsilon(M_\delta,d)(m_1) = \mathcal{C}_\varepsilon(M_{\delta},d)(m_2) \implies \mathcal{C}_{\varepsilon'}(M_{\delta'},d)(m_1) = \mathcal{C}_{\varepsilon'}(M_{\delta'},d)(m_2)
\]
for all $\delta\leq \delta'$ and $\varepsilon\leq \varepsilon'$.
As an example,
the clustering method of \cref{ex1,ex:singlelinkage} naturally extends to a hierarchical clustering method for filtered spaces by considering the canonical maps
$\mathcal{S}_{\varepsilon,\delta}(M,d) \colon M_\delta \to \pi_0\circ \mathcal{G}_{\varepsilon}(M_\delta,d)$.
Typically, the hierarchical clustering in the $\varepsilon$ parameter depends only on the subspace $M_\delta$ but not on the parameter $\delta$ itself, i.e.,
\[
    M_\delta = M_{\delta'}
    \implies
    \mathcal{C}_\varepsilon(M_{\delta},d) = \mathcal{C}_\varepsilon(M_{\delta'},d)
\]
for all $\delta\leq \delta'$.
In this case, by finiteness of the metric space $M$, there is only a finite number of distinct subspaces $M_\delta$, and for each $M_\delta$ there is only a finite number of distinct clusterings $\mathcal{C}_\varepsilon(M_\delta,d)$.

For a filtration of $M$ defined by a  density estimate $\sigma \colon M\to [0, \infty)$, the collection of maps $\{\mathcal{S}_{\varepsilon, \delta}(M)\}_{\varepsilon, \delta\geq0}$ contains plentiful information. For a fixed $\delta$, the family $\{\mathcal{S}_{\varepsilon, \delta}(M)\}_{\varepsilon\geq0}$ simply yields single linkage clustering of all points in $m\in M$ with $\sigma(m)\geq \delta$. Similarly, fixing a scale $\varepsilon$ and by considering the subset $\{\mathcal{S}_{\varepsilon,\delta}(M)\}_{\delta\geq0}$ gives a density-based clustering method at a fixed geometric scale, akin to the
\emph{Topological Mode Analysis}
method proposed in \cite{MR3144911}.
Unfortunately, whereas hierarchical clustering methods enjoy simple graphical representations, it is not clear how to visualize hierarchical clustering methods for filtered spaces in any reasonable way.

\subsection{From sets to vector spaces}
We will now rephrase the above constructions in terms of functors. Given a finite metric space $(M,d)$ and a hierarchical clustering method~$\mathcal{C}$,
for all $\varepsilon\leq \varepsilon'$
there is a unique surjective set map $\mu_{\varepsilon,\varepsilon'}: C_\varepsilon \to C_{\varepsilon'}$ such that $\mathcal{C}_{\varepsilon'}(M,d) = \mu_{\varepsilon,\varepsilon'} \circ \mathcal{C}_\varepsilon(M,d)$. This defines a functor $\mathcal{C}(M,d)\colon [0, \infty)\to \setcat$, and all the internal morphisms of this functor are surjective.
For any field $\field$, post-composing the functor $\mathcal{C}(M,d)$ with the free functor $F_\field: \setcat \to \modcat \field$ to the category of finite dimensional $\field$-vector spaces, we get a functor $F_\field\circ \mathcal{C}(M,d)\colon [0, \infty) \to \modcat \field$.
Furthermore, it is easy to see see that all the internal morphisms are epimorphisms.  Note that for hierarchical clustering methods defined using the path-component functor $\pi_0$, such as single linkage clustering $\pi_0\circ \mathcal{G}_{-}(M,d)$, this can also be interpreted as applying the 0-th homology functor $H_0(-;\field)$ with coefficients in $\field$, since $H_0(-;\field) \simeq F_\field \circ \pi_0$.  In this case, the resulting functor $H_0\big(\mathcal{G}_-{(M,d)} ;\field\big) \colon [0, \infty) \to \modcat \field$ is the \emph{persistent homology} in degree zero of the geometric graphs.

In general, functors of the type $[0, \infty) \to \modcat \field$ (not necessarily assuming epimorphisms) are commonly referred to as \emph{persistence modules}, and they are the main objects of study in the field of topological data analysis. What makes such functors so useful is that they decompose as a direct sum of persistence modules $\field_I$ (see the beginning \cref{sec:bg} for the precise definition), where $I \subseteq [0,\infty)$ is an interval \cite{MR3323327}. This collection of intervals, typically referred to as a \emph{persistence barcode}, is then used to extract topological information from the data at hand; a ``long''
interval corresponds to a topological feature that persists over a significant range.

The process of linearizing through post-composition with the 0-th homology functor allows for the decomposition of the rooted tree associated to the functor $\mathcal{C}(M,d)$ into a collection of intervals describing the evolution of the clusters as we increase the scale parameter. The reduction in complexity comes at the expense that we no longer have precise information of which points belong to each cluster. Even though having this precise information of the clusters may be helpful for visualization, the persistence barcode is better suited for statistical analysis and machine learning.

A natural question to ask is whether filtered hierarchical clustering methods can be linearized in a similar fashion such that they decompose into simple components which can be used to interpret the evolution of connected components across multiple scales.

Following the arguments above, we see that a filtered hierarchical clustering method transforms a  filtered finite metric space $\mathcal M = (M_\delta)_\delta$ into a functor
\[
    \mathcal{C}(\mathcal M,d)\colon [0, \infty)^2\to \setcat,
    \quad
    \text{where}
    \quad
    \mathcal{C}(\mathcal M,d)_{(\varepsilon, \delta)} := \mathcal{C}_\varepsilon(M_\delta, d).
\]
Similarly to above, $\mathcal{C}(\mathcal M,d)_{(\varepsilon, \delta)} \onto \mathcal{C}(\mathcal M,d)_{(\varepsilon', \delta)}$ is a surjection for all $\delta \geq0$ and all $\varepsilon\leq \varepsilon'$, and post-composing with $F_\field$ yields a functor $F_\field\circ \mathcal{C}(\mathcal M,d) \colon [0, \infty)^2\to \modcat\field$ where the morphisms $F_\field\circ \mathcal{C}(\mathcal M,d)_{(\varepsilon, \delta)} \onto F_\field\circ \mathcal{C}(\mathcal M,d)_{(\varepsilon', \delta)}$ are all epimorphisms.

It is well known (see \cref{thm.types_for_grids} below) that the representation theory of functors $[0, \infty)^2 \to \modcat \field$ is very complicated. But what about the subcategory of functors carrying the additional property that all morphisms are epimorphic along one or both parameters? Or the subcategory of functors whose morphisms are injective along one parameter and surjective along the other? In what follows we shall precisely capture the representation type of such functors under the assumption that they can be re-indexed over a finite regular grid, which is typically the case for filtered hierarchical clustering methods, as discussed before.

To make this precise, denote by $\vec A_n$ the quiver $1 \to 2 \to \cdots \to n$. Alternatively, the reader may think of $\vec A_n$ as the poset $\{ 1< 2 < \cdots < n \}$, or its poset category. We denote by $\vec A_m \otimes \vec A_n$ the quiver with relations given by the fully commutative grid
\[
    \begin{tikzcd}[row sep={9ex,between origins},column sep={10ex,between origins}]
        (1,1) \ar[r] \ar[d] & (2,1) \ar[r] \ar[d] & (3,1) \ar[r,dotted] \ar[d] & (m,1) \ar[d] \\
        (1,2) \ar[r] \ar[d,dotted] & (2,2) \ar[r] \ar[d,dotted] & (3,2) \ar[r,dotted] \ar[d,dotted]
        & (m, 2) \ar[d,dotted] \\
        (1,n) \ar[r] & (2,n) \ar[r] & (3,n) \ar[r,dotted] & (m,n)
    \end{tikzcd}
\]
Equivalently, we may think of $\vec A_m \otimes \vec A_n$ as the product of $\vec A_m$ and $\vec A_n$ in the category of posets or in the category of small categories.
We denote by $\rep_\field(\vec A_m \otimes \vec A_n)$ the category of fully commutative grids of finite dimensional $\field$-vector spaces. There is a complete classification of when these categories are finite or tame, but unfortunately these are very few cases beyond the one-parameter situation:

\begin{thm}[{\cite[Theorem 5]{MR1808681}, \cite[Theorem 2.5]{MR1273693}}]
    \label{thm.types_for_grids}
    The category of representations $\rep_\field ( \vec A_m \otimes  \vec A_n )$ contains finitely many indecomposables precisely in the cases
    \begin{itemize}
        \item $m = 1$ or $n = 1$,
        \item $(m,n) \in \{ (2,2), (2,3), (2,4), (3,2), (4,2)\}$.
    \end{itemize}
    It is of tame representation type precisely in the cases
    \begin{itemize}
        \item $(m,n) \in \{ (2,5), (3,3), (5,2)\}$.
    \end{itemize}
    In all other cases it is of wild representation type.
\end{thm}

It is not hard to see that the remaining cases are of wild representation type. As an example, the following diagrams show that any representation of a particular quiver of wild representation type lifts to a representation of $\vec A_4 \otimes \vec A_3$:
\[
   \begin{tikzpicture}[xscale=.65,yscale=.6]
        \node (V1) at (0,0) {$V_1$};
        \node (V2) at (2,0) {$V_2$};
        \node (V3) at (2,-2) {$V_3$};
        \node (V4) at (2,2) {$V_4$};
        \node (V5) at (4,0) {$V_5$};
        \node (V6) at (6,0) {$V_6$};
        \node(W1) at (9,0) {$V_1$};
        \node (W2) at (11,0) {$V_2$};
        \node (W3) at (11,-2) {$V_3$};
        \node (W4) at (11,2) {$V_4$};
        \node (W5) at (13,0) {$V_5$};
        \node (W6) at (15,0) {$V_6$};
        \node (X1) at (9,-2) {$V_1$};
        \node (O1) at (9,2) {$0$};
        \node (X3) at (13,2) {$V_4$};
        \node (X4) at (15,2) {$V_4$};
        \node (O3) at (13,-2) {$0$};
        \node (O4) at (15,-2) {$0$};
        \path[->,font=\scriptsize]
            (V1) edge  (V2)
            (V4) edge (V2)
            (V2) edge(V3)
            (V2) edge (V5)
            (V5) edge (V6)
            (W1) edge (W2)
            (W2) edge (W5)
            (W5) edge (W6)
            (W2) edge (W3)
            (W4) edge (W2)
            (O1) edge (W4)
            (W4) edge node[auto] {$\id$} (X3)
            (X3) edge node[auto] {$\id$} (X4)
            (O1) edge (W1)
            (W3) edge (O3)
            (W5) edge (O3)
            (W6) edge (O4)
            (W1) edge node[auto] {$\id$} (X1)
            (X1) edge (W3)
            (X3) edge (W5)
            (X4) edge (W6)
            (O3) edge (O4);
   \end{tikzpicture}
\]

As discussed above, the representations coming from hierarchical clustering on a filtered space will contain epimorphisms on all horizontal arrows. We denote by $\rep_\field^{\e, *}(\vec A_m \otimes \vec A_n)$ and $\rep_\field^{\e,\e}(\vec A_m \otimes \vec A_n)$ the full subcategories of $\rep_{\field}(\vec A_m \otimes \vec A_n)$ consisting of all grid-shaped diagrams of vector spaces such that the horizontal morphisms, respectively all morphisms, are epimorphisms. Given that these are proper subcategories, one may hope that their indecomposables are classifiable for a wider range of values for $m$ and $n$.

One starting point for this project was the observation that the subcategory of representations $\rep_\field^{\e, *}(\vec A_m \otimes \vec A_n)$ can be studied using Auslander--Reiten-theory.
To begin with, one observes that this subcategory is closed under extensions and quotients. By \cite[Corollaries~3.7 and 3.8]{MR617088}, it has almost split sequences, and these are induced by almost split sequences in  $\rep_\field(\vec A_m \otimes \vec A_n)$. In the cases that $\rep_\field^{\e, *}(\vec A_m \otimes \vec A_n)$ contains only finitely many indecomposables, these can be obtained by constructing the Auslander--Reiten quiver, starting from the injectives.

One then observes that these finite Auslander--Reiten quivers look very similar to Auslander--Reiten quivers for categories $\rep_\field (\vec A_m \otimes \vec A_{n-1})$ --- see \cref{fig:compare-ar-quivers}. Only the thin modules with rectanglular support ``in the lower left corner'', modules which we will denote by $\field_{\{1, \ldots, i \} \times \{ j, \ldots, m \}}$ --- see \cref{subfig.C_e_star} --- do not seem to appear in this correspondence.

\begin{figure}
    \centering
    \begin{tikzpicture}[scale=-1.4]
        \node (0) at (1,0) {$\begin{smallmatrix} 1 & 0 \\ 0 & 0 \\ 0 & 0 \end{smallmatrix}$};
        \node (1) at (2,-1) {$\begin{smallmatrix} 1 & 0 \\ 1 & 0 \\ 0 & 0 \end{smallmatrix}$};
        \node (2) at (2,1) {$\begin{smallmatrix} 1 & 1 \\ 0 & 0 \\ 0 & 0 \end{smallmatrix}$};
        \node (3) at (3,0) {$\begin{smallmatrix} 1 & 1 \\ 1 & 0 \\ 0 & 0 \end{smallmatrix}$};
        \node[draw,rectangle,inner sep=2pt,outer sep=2pt] (4) at (3,-2) {$\begin{smallmatrix} 1 & 0 \\ 1 & 0 \\ 1 & 0 \end{smallmatrix}$};
        \node (5) at (4,-1) {$\begin{smallmatrix} 1 & 1 \\ 1 & 0 \\ 1 & 0 \end{smallmatrix}$};
        \node (6) at (4,1) {$\begin{smallmatrix} 0 & 0 \\ 1 & 0 \\ 0 & 0 \end{smallmatrix}$};
        \node (7) at (4,0) {$\begin{smallmatrix} 1 & 1 \\ 1 & 1 \\ 0 & 0 \end{smallmatrix}$};
        \node (8) at (5,0) {$\begin{smallmatrix} 1 & 1 \\ 2 & 1 \\ 1 & 0 \end{smallmatrix}$};
        \node (9) at (6,-1) {$\begin{smallmatrix} 0 & 0 \\ 1 & 1 \\ 0 & 0 \end{smallmatrix}$};
        \node (10) at (6,1) {$\begin{smallmatrix} 1 & 1 \\ 1 & 1 \\ 1 & 0 \end{smallmatrix}$};
        \node[draw,rectangle,inner sep=2pt,outer sep=2pt] (11) at (6,0) {$\begin{smallmatrix} 0 & 0 \\ 1 & 0 \\ 1 & 0 \end{smallmatrix}$};
        \node (12) at (7,0) {$\begin{smallmatrix} 0 & 0 \\ 1 & 1 \\ 1 & 0 \end{smallmatrix}$};
        \node[draw,rectangle,inner sep=2pt,outer sep=2pt] (13) at (8,-1) {$\begin{smallmatrix} 0 & 0 \\ 0 & 0 \\ 1 & 0 \end{smallmatrix}$};
        \node[draw,rectangle,inner sep=2pt,outer sep=2pt] (14) at (7,2) {$\begin{smallmatrix} 1 & 1 \\ 1 & 1 \\ 1 & 1 \end{smallmatrix}$};
        \node[draw,rectangle,inner sep=2pt,outer sep=2pt] (15) at (8,1) {$\begin{smallmatrix} 0 & 0 \\ 1 & 1 \\ 1 & 1 \end{smallmatrix}$};
        \node[draw,rectangle,inner sep=2pt,outer sep=2pt] (16) at (9,0) {$\begin{smallmatrix} 0 & 0 \\ 0 & 0 \\ 1 & 1 \end{smallmatrix}$};
        \path[<-]
            (0) edge node {} (1)
            (0) edge node {} (2)
            (1) edge node {} (3)
            (1) edge node {} (4)
            (2) edge node {} (3)
            (3) edge node {} (5)
            (3) edge node {} (6)
            (3) edge node {} (7)
            (4) edge node {} (5)
            (5) edge node {} (8)
            (6) edge node {} (8)
            (7) edge node {} (8)
            (8) edge node {} (9)
            (8) edge node {} (10)
            (8) edge node {} (11)
            (9) edge node {} (12)
            (10) edge node {} (12)
            (10) edge node {} (14)
            (11) edge node {} (12)
            (12) edge node {} (13)
            (12) edge node {} (15)
            (13) edge node {} (16)
            (14) edge node {} (15)
            (15) edge node {} (16);

        \begin{scope}[shift={(8,4)},scale=-1]
            \node (0) at (1,0) {$\begin{smallmatrix} 0 & 0 \\ 0 & 1 \end{smallmatrix}$};
            \node (1) at (2,-1) {$\begin{smallmatrix} 0 & 1 \\ 0 & 1 \end{smallmatrix}$};
            \node (2) at (2,1) {$\begin{smallmatrix} 0 & 0 \\ 1 & 1 \end{smallmatrix}$};
            \node (3) at (3,0) {$\begin{smallmatrix} 0 & 1 \\ 1 & 1 \end{smallmatrix}$};
            \node (4) at (4,-1) {$\begin{smallmatrix} 0 & 0 \\ 1 & 0 \end{smallmatrix}$};
            \node (5) at (4,1) {$\begin{smallmatrix} 0 & 1 \\ 0 & 0 \end{smallmatrix}$};
            \node (6) at (4,0) {$\begin{smallmatrix} 1 & 1 \\ 1 & 1 \end{smallmatrix}$};
            \node (7) at (5,0) {$\begin{smallmatrix} 1 & 1 \\ 1 & 0 \end{smallmatrix}$};
            \node (8) at (6,-1) {$\begin{smallmatrix} 1 & 1 \\ 0 & 0 \end{smallmatrix}$};
            \node (9) at (6,1) {$\begin{smallmatrix} 1 & 0 \\ 1 & 0 \end{smallmatrix}$};
            \node (10) at (7,0) {$\begin{smallmatrix} 1 & 0 \\ 0 & 0 \end{smallmatrix}$};
            \path[->]
                (0) edge node {} (1)
                (0) edge node {} (2)
                (1) edge node {} (3)
                (2) edge node {} (3)
                (3) edge node {} (4)
                (3) edge node {} (5)
                (3) edge node {} (6)
                (4) edge node {} (7)
                (5) edge node {} (7)
                (6) edge node {} (7)
                (7) edge node {} (8)
                (7) edge node {} (9)
                (8) edge node {} (10)
                (9) edge node {} (10);
        \end{scope}
    \end{tikzpicture}
    \caption{The AR-quivers of $\rep^{\e,*}_\field \vec A_2 \tensor \vec A_3$ (top, with
        generators of
        $\rep^{\e,\m}_\field \vec A_2 \tensor \vec A_3$ highlighted) and of $\rep_\field \vec A_2 \tensor \vec A_2$ (bottom).}
    \label{fig:compare-ar-quivers}
\end{figure}

This observation led us to suspect that the following natural construction, transforming the horizontal epimorphsims into general morphisms using kernels, might be useful in this context --- a hope that was ultimately justified by the theorem below.

\begin{construction}
    \label{const:intro}
    Let $X = (X_{i,j}) \in \rep_{\field}^{\e, *}(\vec A_m \otimes \vec A_n)$. We denote by $\tors X$ the object in $\rep_{\field}(\vec A_m \otimes \vec A_n )$ given by
    \[ (\tors X)_{i,j} = \ker [ X_{i,j} \to X_{i,n} ]. \]
    One easily observes that $\tors$ defines a functor $\rep_{\field}^{\e, *}(\vec A_m \otimes \vec A_n) \to \rep_{\field}(\vec A_m \otimes \vec A_{n-1})$, and that the objects sent to $0$ by $\tors$ are precisely those where all vertical maps are monomorphisms, that is the objects in $\rep_{\field}^{\e, \m}(\vec A_m \otimes \vec A_n )$.
\end{construction}

\begin{thm}[See \cref{cor:epireps}]
    \label{thm:concrete}
    The functor $\tors$ induces an equivalence
    \[
        \frac{\rep_{\field}^{\e, *}(\vec A_m \otimes \vec A_n)}{\rep_{\field}^{\e, \m}(\vec A_m \otimes \vec A_n)} \to \rep_{\field}(\vec A_m \otimes \vec A_{n-1}).
    \]
    Moreover, $\rep_{\field}^{\e, \m}(\vec A_m \otimes \vec A_n)$ consists precisely of all finite direct sums of thin modules of the form $\field_{\{1, \ldots, i \} \times \{ j, \ldots, m \}}$, as depicted in \cref{subfig.C_e_star}.
\end{thm}

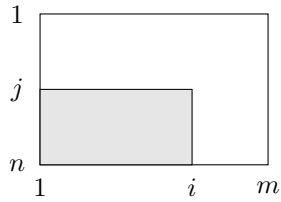
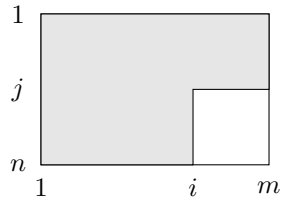
\begin{figure}
    \begin{subfigure}{0.4\textwidth}
        \[
            \begin{tikzpicture}
                \node at (-.3, 2) {1};
                \node at (-.3, 1) {$j$};
                \node at (-.3, 0) {$n$};
                \node at (0, -.3) {1};
                \node at (2, -.3) {$i$};
                \node at (3,-.3) {$m$};
                \draw (0,0) rectangle (3,2);
                \draw [fill=lightergray] (0,0) rectangle (2,1);
            \end{tikzpicture}
        \]
        \subcaption{Diagrammatic depiction of the thin module $\field_{\{1, \ldots, i \} \times \{ j, \ldots, n \}}$ } \label{subfig.C_e_star}
    \end{subfigure}
    \qquad
    \begin{subfigure}{0.4\textwidth}
        \[
            \begin{tikzpicture}
                \node at (-.3, 2) {1};
                \node at (-.3, 1) {$j$};
                \node at (-.3, 0) {$n$};
                \node at (0, -.3) {1};
                \node at (2, -.3) {$i$};
                \node at (3,-.3) {$m$};
                \draw (0,0) rectangle (3,2);
                \draw [fill=lightergray] (0,0) -- (2,0) -- (2,1) -- (3,1) -- (3,2) -- (0,2) -- cycle;
            \end{tikzpicture}
        \]
        \subcaption{Diagramatic depiction of the module $\field_{\{(x, y) \mid x < i \text{ or } y < j\}}$ in $\mbC_{\e,\e}$} \label{subfig.C_e_e}
    \end{subfigure}
    \caption{Certain special modules}
\end{figure}

Spelled out, this means that the representations in $\rep_\field(\vec A_m \otimes \vec A_n)$ that are surjective in one direction correspond (up to a finitely classified, completely explicit list of direct summands) to the representations in $\rep_\field(\vec A_m \otimes \vec A_{n-1})$.

We will also see, in \cref{cor.commutes_with_tau}, that this correspondence preserves the Auslander--Reiten structure.

As an immediate application we can combine \cref{thm:concrete} with \cref{thm.types_for_grids} and obtain the following classification.

\begin{cor}
    The category $\rep_{\field}^{\e, *} (\vec A_m \otimes \vec A_n)$ contains finitely many indecomposables precisely in the cases
    \begin{itemize}
        \item $n \leq 2$ or $m = 1$;
        \item $(m,n) \in \{ (2,3), (2,4), (2,5), (3,3), (4,3)\}$.
    \end{itemize}
    It is of tame representation type precisely in the cases
    \begin{itemize}
        \item $(m,n) \in \{ (2,6), (3,4), (5,3)\}$.
    \end{itemize}
    In all other cases it is of wild representation type.
\end{cor}

In fact, for $n \leq 2$, all the indecomposables are constant modules on connected subsets. In the context of clustering above, this means that  $F_\field\circ \mathcal{C}(M,d)$ can be understood from a collection of simple regions if and only if the filtration of $M$ is trivial or essentially a two-step filtration $M_{\delta_1} \subseteq M_{\delta_2}$.

\medskip

For the categories of the form $\rep^{\e,\e}(\vec A_m \otimes \vec A_n)$ we obtain the following variant of \cref{thm:concrete}.

\begin{thm}[Dual of \cref{cor.mono_mono}]
    There is an equivalence
    \[
        \frac{\rep_\field^{\e,\e}(\vec A_m \otimes \vec A_n)}{\mbC_{\e,\e}} \simeq \rep_{\field} (\vec A_{m-1} \otimes \vec A_{n-1}),
    \]
    where $\mbC_{\e,\e}$ is the subcategory whose indecomposables are of the form  $\field_{\{(x, y) \mid x < i \text{ or } y < j\}}$ as depicted in \cref{subfig.C_e_e}.
\end{thm}

Although the representation types of both $\rep^{\e,*}(\vec A_m \otimes \vec A_n)$ and $\rep^{\e,\e}(\vec A_m \otimes \vec A_n)$ are wild already for relatively small values of $m$ and $n$, it is not clear what indecomposables appear as summands of, say,  $F_\field\circ\mathcal{S}(M,d)$ for some metric space $(M,d)$ whose filtration is given by a ``density'' function $\sigma\colon M\to [0, \infty)$. Perhaps more interestingly, which indecomposables do we see in relevant data sets? These are questions which hopefully will be answered in future work.

\subsection{Torsion and cotorsion pairs}

While a direct approach to the above theorem is possible, we chose here to prove it by considering the concepts of torsion and cotorsion pairs. The advantage of this approach is two-fold: Firstly, it is easier to get a feeling for why the proof should work, rather than just a technical verification. Secondly, the result in this language is actually much more general than the above application, and is therefore of independent interest both from a purely representation theoretic point of view and with respect to applicability to other specific instances.

For a formal definition of torsion and cotorsion pairs we refer to \cref{sec:cotorsiontorsion}. For this introduction, the main point is that both torsion and cotorsion pairs are a way of ``orthogonally'' decomposing an abelian category into two parts, where ``orthogonal'' refers to $\Hom$ vanishing for torsion and to $\Ext^1$-vanishing for cotorsion pairs. If a subcategory appears in both a torsion and a cotorsion pair, it is natural to wonder if its two complements are related. The most abstract version of the main result gives a positive answer to this suspicion:

\begin{thm}[See \cref{thm:tct}]
    \label{thm.triple}
    Let $\mcT, \mcF, \mcD$ be three subcategories of an abelian category, such that $( \mcT, \mcF)$ is a torsion pair and $( \mcF, \mcD )$ is a cotorsion pair. Then
    \[
        \mcT \simeq \frac{\mcD}{\mcF \cap \mcD}.
    \]
\end{thm}

From this general result, we will deduce \cref{thm:concrete} by establishing that the triple of subcategories of $\rep_{\field}(\vec A_m \otimes \vec A_n)$ given by
\[
    \big(\rep_{\field}(\vec A_m \otimes \vec A_{n-1}), \rep_{\field}^{*, \m}(\vec A_m \otimes \vec A_n), \rep_{\field}^{\e, *}(\vec A_m \otimes \vec A_n)\big)
\]
satisfies the conditions of \cref{thm.triple}.

\begin{remark}
    \label{rem:related_results}
    While the dual statements \cref{thm.ctt,thm:tct}, to the best of our knowledge, do no appear in the current literature in this form, such equivalences obtained from torsion and cotorsion pairs have been known to exist in various settings.
    In particular, \cite[Proposition V.5.2]{MR2327478} seems to be a special case of our theorem, with a number of additional technical conditions.
    Moreover, after we made a first preprint of the present paper publicly available, Apostolos Beligiannis informed us about an extensive manuscript \cite{Beligiannis2010Tilting}, which he had started in 2004 and announced at several conferences in the subsequent years but not yet made publicly available, and which we were not aware of at the time of writing.
    The principal aim of that manuscript was to develop a fully general tilting theory in an arbitrary abelian category.
    In that manuscript he obtained statements analogous to our \cref{thm.ctt,thm.ctt=tilting}, together with the dual versions \cref{thm:tct,thm.tct_from_tilting}, in a slightly more general setting (avoiding the assumption on enough projective or injective objects).
    Our results in \cref{sec:applicationquivers} demonstrate the usefulness of such a general tilting theory.
\end{remark}

\section{Preliminary results on (co)torsion pairs and tilting theory}
\label{sec:bg}
In this section we establish the general theory that will later be applied to examples coming from multi-parameter persistent homology.  This includes recalling known results from tilting theory and the theory of (co)torsion pairs while illustrating the results using an example from ordinary (one-parameter) persistent homology.

Let $(X,\leq)$ be a poset and $\field$ a field.  By considering $(X,\leq)$ as a category, we form the abelian category of pointwise finite dimensional (covariant) representations, $\Rep_\field^\mathrm{pfd} X$, consisting of functors
\[
    M\colon X \to \modcat\field
\]
into the category of finite dimensional vector spaces and natural transformations.  A subset $C$ of $X$ is \emph{convex} if for every pair $x,y\in C$, $x\leq z\leq y$ implies $z\in C$.  For a convex subset $C\subseteq X$, the \emph{constant representation} $\field_C$ is defined as
\[
    \field_C(x) = \begin{cases} \field,& x \in C \\ 0,& x \notin C \end{cases}
\]
and $\field_C(x\leq y) = \id_\field$ whenever $x,y\in C$.  Putting $P_x = \field_{\{ y \in X \mid y \geq x \}}$, we observe the following standard fact: For $x\in X$ and $M \in \Rep_\field^\mathrm{pfd} X$ there is an isomorphism
\[
    \Hom_{\Rep X}(P_x, M) \isom M(x),
\]
natural in both $x$ and $M$.  This is an application of the Yoneda lemma; a natural transformation $\eta\colon P_x \to M$ is mapped to $\eta_x(1) \in M(x)$, and an element $m \in M(x)$ is mapped to the uniquely determined natural transformation sending $1\in P_x(x) = \field$ to $m\in M(x)$.  In particular, $\Hom_{\Rep X}(P_x, -)$ is right exact, so $P_x$ is a projective object of $\Rep_\field^\mathrm{pfd} X$.

Based on this, a representation $M \in \Rep_\field^\mathrm{pfd} X$ is \emph{finitely generated} if there exists a finite indexing set $I$, a set of elements $(x_i\in X)_{i\in I}$ and an epimorphism of functors $\bigoplus_{i\in I} P_{x_i} \onto M$.  It is \emph{finitely presented} if, additionally, the kernel of every such epimorphism is again finitely generated.  The full subcategory of finitely presented representations is denoted $\Rep_\field^\mathrm{fp} X \subseteq \Rep_\field^\mathrm{pfd} X$.

The following will be our running example.
\begin{ex}
    \label{ex:repsofR}
    Fix a field $\field$ and let $\mbR_{\geq0}$ be the totally ordered set of non-negative real numbers considered as a category.  For each $x\geq 0$, the projective $P_x$ in $\Rep_\field^\mathrm{pfd} \mbR_{\geq0}$ is concretely given as $\field_{\hopen x,\infty}$.

    The indecomposable objects in this category are completely classified by the constant representations $\field_I$, where $I$ is an (open, half-open or closed) interval; see \cite[Theorem~1.1]{MR3323327}.

    The abelian subcategory $\Rep_\field^\mathrm{fp} \mbR_{\geq0}$ of finitely presented representations therefore has as indecomposable objects the $\field_{\hopen x,y}$, where we allow $y = \infty$, as they have presentations
    \begin{equation} \label{eq:presentation}
        P_y \into P_x \onto \field_{\hopen x,y}.
    \end{equation}
    Conversely, a constant representation $\field_I$ is finitely presented only if the interval $I$ is half-open of the form $\hopen x,y$.
\end{ex}

\subsection{Cotorsion torsion triples} \label{sec:cotorsiontorsion}

\begin{defn}
    \label{def:tp}
    Let $\mcA$ be an abelian category. A \emph{torsion pair} is a pair $(\mcT, \mcF)$ of full subcategories of $\mcA$, the \emph{torsion} and \emph{torsion-free} parts, respectively, satisfying
    \begin{enumerate}
        \item $\Hom_{\mcA}(\mcT, \mcF) = 0$, and
        \item for any object $A \in \mcA$ there is a short exact sequence $\tors A \into A \onto \tfree A$ with $\tors A \in \mcT$ and $\tfree A \in \mcF$.
    \end{enumerate}
\end{defn}

\begin{remark}
    \label{rmk:tp}
    The subcategories in a torsion pair determine each other, namely
    \begin{align*}
        \mcF = \mcT^\perp &:= \{ A \in \mcA \mid \Hom_\mcA(\mcT,A) = 0 \} \\
        \mcT = {}^\perp\mcF &:= \{ A \in \mcA \mid \Hom_\mcA(A,\mcF) = 0 \}.
    \end{align*}
    This is seen by considering the short exact sequence associated to $A$: If $A \in \mcT^\perp$, then $A \onto \tfree A$ necessarily becomes an isomorphism.

    The subcategory $\mcT$ is extension-closed, as can be seen by applying the left-exact functor $\Hom_\mcA(-,F)$ with $F \in \mcF$ to any short exact sequence $T \into X \onto T'$ with $T,T' \in \mcT$.  Similarly, $\mcF$ is extension-closed as well.  Moreover, left exactness of the $\Hom$-functors also implies that $\mcT$ is closed under factors, while $\mcF$ is closed under subobjects.

    The short exact sequence in the definition of torsion pair is functorial.  More precisely, for each object of $\mcA$, fix a short exact sequence as in the definition.  For any morphism $g\colon A\to A'$ the composite $\tors A \into A \stackrel g\to A' \onto \tfree A'$ is zero, since $\Hom_\mcA(\mcT,\mcF) = 0$.  By universality of kernels and cokernels, we can complete the following diagram uniquely.
    \[
        \begin{tikzpicture}[scale=.75]
            \node (1) at (0,2) {$\tors A$};
            \node (2) at (3,2) {$A$};
            \node (3) at (6,2) {$\tfree A$};
            \node (4) at (0,0) {$\tors A'$};
            \node (5) at (3,0) {$A'$};
            \node (6) at (6,0) {$\tfree A'$};
            \path[->,font=\scriptsize]
                (1) edge[right hook->] node[auto] {} (2)
                (2) edge[->>] node[auto] {} (3)
                (4) edge[right hook->] node[auto] {} (5)
                (5) edge[->>] node[auto] {} (6)
                (2) edge node[auto] {$g$} (5)
                (1) edge[densely dashed] node[auto] {$\tors g$} (4)
                (3) edge[densely dashed] node[auto] {$\tfree g$} (6);
        \end{tikzpicture}
    \]
    This defines the functors $\tors\colon \mcA \to \mcT$ and $\tfree\colon \mcA \to \mcF$.

    Applying $\Hom_\mcA(-,F)$, where $F\in\mcF$, to $\tors A \into A \onto \tfree A$ yields an isomorphism $\Hom_\mcA(A,F) \cong \Hom_\mcA(\tfree A,F)$, using left-exactness and $\Hom_\mcA(\tors A,F)=0$.  Thus $\tfree$ is left adjoint to the inclusion $\mcF \into \mcA$, and similarly $\tors$ is right adjoint to $\mcT \into \mcA$.
\end{remark}

\begin{ex}
    \label{ex:first_tp}
    Consider $\mcA = \Rep_\field^{\mathrm{fp}} \mbR_{\geq 0}$, and put
    \begin{align*}
        \mcT &= \add \big(\{ \field_{\hopen 0,y} \mid 0 < y \leq \infty \} \cup \{ \field_{\hopen x,y} \mid 1 \leq x < y \leq \infty \}\big) \\
        \mcF &= \add \big(\{ \field_{\hopen x,y} \mid 0 < x < y \leq 1 \}\big).
    \end{align*}
    Then $(\mcT,\mcF)$ forms a torsion pair, as can be verified directly.  In \cref{ex.conc_torsion_from_tilting} we arrive at this conclusion by constructing the pair from a \emph{tilting} subcategory.

    The ``interesting'' indecomposables, i.e., the ones not lying in either $\mcT$ or $\mcF$ already, are the $\field_{\hopen x,y}$ with $0 < x < 1 < y \leq \infty$. For these, the functorial short exact sequence is given as
    \[
        \underbrace{\field_{\hopen 1,y}}_{\in \mcT} \into \field_{\hopen x,y} \onto \underbrace{\field_{\hopen x,1}}_{\in \mcF}.
    \]
\end{ex}

For $\mcX \subseteq \mcA$ define full subcategories by
\begin{align*}
    \mcX^{\perp_1} := \{ A \in \mcA \mid \Ext_\mcA^1(\mcX,A) = 0 \} \\
    {}^{\perp_1}\mcX := \{ A \in \mcA \mid \Ext_\mcA^1(A,\mcX) = 0 \}.
\end{align*}

\begin{defn}[{\cite{MR565595}, see also~\cite[\S V.3.3]{MR2327478}}]
    \label{def.cotorsion}
    A \emph{cotorsion pair} is a pair $(\mcC, \mcD)$ of full subcategories of an abelian category $\mcA$, the \emph{cotorsion} and \emph{cotorsion-free} parts, respectively, such that
    \begin{enumerate}
        \item $\mcC = {}^{\perp_1} \mcD \text{ and } \mcD = \mcC^{\perp_1}$,
    and
        \item for any object $A \in \mcA$ there are two short exact sequences
    \[
        \ctf A \into \ctt A \onto A \quad \text{and} \quad A \into \widetilde{\ctf} A \onto \widetilde{\ctt} A
    \]
    with $\ctt A$ and $\widetilde{\ctt} A$ in $\mcC$ and $\ctf A$ and $\widetilde{\ctf} A$ in $\mcD$.
    \end{enumerate}
\end{defn}

\begin{remark}
    Other sources define cotorsion pairs as pairs only satisfying the first condition, and rather call pairs additionally providing the two exact sequences \emph{complete}.  The existence of these short exact sequences will be crucial for our results, so we have chosen to follow Beligiannis--Reiten's convention.

    In contrast to the case with torsion pairs, the constructions $\ctt,\widetilde\ctt,\ctf,\widetilde\ctf$ are \emph{not} functorial.  They do, however, define functors to certain additive quotients; see \cref{lemma:cotors_functors_cd}.

    In the presence of condition (2), we may replace condition (1) by the following
    \begin{enumerate}
        \item[(1$'$)] $\Ext^1(\mcC, \mcD) = 0$ and $\mcC$ and $\mcD$ are closed under direct summands.
    \end{enumerate}
    To see this, consider a short exact sequence $\ctf A \into \ctt A \onto A$, where $A \in {}^{\perp_1}\mcD$.  Then this sequence splits and therefore $A$ is a summand in $\ctt A$.  Since $\mcC$ is closed under summands, this implies $A \in \mcC$.  Analogously, we also obtain the second inclusion $\mcC^{\perp_1} \subseteq \mcD$.

    Finally, we note that the condition $\mcC = {}^{\perp_1}\mcD$ readily implies that $\mcC$ is closed under taking extensions, by an argument similar to that of \cref{rmk:tp}.  Furthermore, $\mcC$ must contain every projective object in $\mcA$.  Likewise, $\mcD$ is extension-closed and contains every injective in $\mcA$.
\end{remark}

\begin{ex}
    \label{ex.first_ctp}
    Consider $\mcA = \Rep_\field^{\mathrm{fp}} \mbR_{\geq 0}$, and put
    \begin{align*}
        \mcC &= \add \big(\{ \field_{\hopen x,\infty} \mid 0\leq x < \infty \} \cup \{ \field_{\hopen x,y} \mid 0 \leq x < y < 1 \}\big) \\
        \mcD &= \add \big(\{ \field_{\hopen 0,y} \mid 0 < y \leq \infty \} \cup \{ \field_{\hopen x,y} \mid 1 \leq x < y \leq \infty \}\big).
    \end{align*}
    Now let $\field_{\hopen a,b}$ be an indecomposable in $\mcA$.  The first short exact sequence is given by
    \[
        \underbrace{\field_{\hopen b,\infty}}_{\in \mcD} \into \underbrace{\field_{\hopen a,\infty}}_{\in \mcC} \onto \field_{\hopen a,b}
    \]
    for $b \geq 1$. (If $b < 1$ then the indecomposable is already in $\mcC$.) The second short exact sequence is non-trivial for $\field_{\hopen a,b}$ with $a < 1$, and in that case it is given as
    \[
        \field_{\hopen a,b} \into \underbrace{\field_{\hopen 0,b}}_{\in \mcD} \onto \underbrace{\field_{\hopen 0,a}}_{\in \mcC}.
    \]
\end{ex}

The following definition introduces the main object of study in this section.
\begin{defn}
    A \emph{cotorsion torsion triple} is a triple of subcategories $(\mcC, \mcT, \mcF)$ in an abelian category such that $(\mcC, \mcT)$ is a cotorsion pair and $(\mcT, \mcF)$ is a torsion pair.
\end{defn}

\begin{ex}
    \label{ex:first_cttt}
    Comparing \cref{ex:first_tp,ex.first_ctp}, we thus obtain an example of a cotorsion torsion triple in $\mcA = \Rep_\field^\mathrm{fp} \mbR_{\geq0}$:
    \begin{align*}
        \mcC &= \add \big(\{ \field_{\hopen x,\infty} \mid 0\leq x < \infty \} \cup \{ \field_{\hopen x,y} \mid 0 \leq x < y < 1 \}\big) \\
        \mcT &= \add \big(\{ \field_{\hopen 0,y} \mid 0 < y \leq \infty \} \cup \{ \field_{\hopen x,y} \mid 1 \leq x < y \leq \infty \}\big) \\
        \mcF &= \add \big(\{ \field_{\hopen x,y} \mid 0 < x < y \leq 1 \}\big).
    \end{align*}
\end{ex}

In preparation for the next lemma, we recall the construction of additive quotients.  Let $\mcE$ be an additive category, and $\mcX \subseteq \mcE$ a full subcategory which is closed under finite direct sums.  Define the category $\mcE/\mcX$ as having the same objects as $\mcE$, and with
\[
    \Hom_{\mcE/\mcX}(E,E') := \Hom_\mcE(E,E')/\sim,
\]
where $f \sim g$ if and only if $f-g$ factors through an object of $\mcX$.  In particular, all objects of $\mcX$ become zero in $\mcE/\mcX$, and the canonical quotient functor $\mcE \to \mcE/\mcX$ enjoys the following universal property:  If $F\colon \mcE \to \mcE'$ is an additive functor such that $F(\mcX) = 0$, then there exists a unique additive functor $\bar F\colon \mcE/\mcX \to \mcE'$ such that
\[
    \begin{tikzpicture}[scale=.75]
        \node (1) at (0,2) {$\mcE$};
        \node (2) at (3,2) {$\mcE'$};
        \node (3) at (0,0) {$\mcE/\mcX$};
        \path[->,font=\scriptsize]
            (1) edge node[auto] {$F$} (2)
            (1) edge node[auto] {} (3)
            (3) edge[densely dashed] node[auto,swap] {$\exists! \bar F$} (2);
    \end{tikzpicture}
\]
commutes.

\begin{lemma}
    \label{lemma:cotors_functors_cd}
    Let $(\mcC, \mcD)$ be a cotorsion pair in $\mcA$.  Then any map from $\mcC$ to $\mcD$ factors through some object in $\mcC \cap \mcD$.

    Moreover, the constructions $\ctf, \widetilde\ctf$ and $\ctt, \widetilde\ctt$ define functors from $\mcA$ to $\frac{\mcD}{\mcC \cap \mcD}$ and $\frac{\mcC}{\mcC \cap \mcD}$, respectively.
\end{lemma}
\begin{proof}
    For the first claim, let $f \colon C \to D$, with $C \in \mcC$ and $D \in \mcD$. Consider the short exact sequence $\ctf D \into \ctt D \onto D$. Since $\Ext_\mcA^1(C, \ctf D) = 0$, the map $\Hom_\mcA(C,\ctt D) \to \Hom_\mcA(C,D)$ is an epimorphism, so $f$ factors through $\ctt D$. But $\ctt D$ is in $\mcC$ by construction and in $\mcD$ since $\mcD$ is closed under extensions.

    For the second part, we only show that $\ctt$ and $\ctf$ are functors, the statement for $\widetilde\ctt$ and $\widetilde\ctf$ being dual.  Note that any map $g\colon A \to B$ in $\mcA$ may be lifted to a map of short exact sequences
    \[
        \begin{tikzpicture}[scale=.75]
            \node (dA) at (0,0) {$\ctf A$};
            \node (cA) at (3,0) {$\ctt A$};
            \node (A) at (6,0) {$A$};
            \node (dB) at (0,-2) {$\ctf B$};
            \node (cB) at (3,-2) {$\ctt B$};
            \node (B) at (6,-2) {$B$,};
            \path[->,font=\scriptsize]
                (dA) edge [right hook->] (cA)
                (cA) edge[->>] (A)
                (dB) edge [right hook->] (cB)
                (cB) edge [->>] (B)
                (dA) edge [densely dashed] node[auto] {$\hat g$} (dB)
                (cA) edge [densely dashed] node[auto] {$\bar g$} (cB)
                (A) edge node[auto] {$g$} (B);
        \end{tikzpicture}
    \]
    first to $\bar g\colon \ctt A \to \ctt B$ since $\Ext^1_\mcA(\ctt A,\ctf B) = 0$, and then to $\hat g\colon \ctf A \to \ctf B$ by the universal property of kernels.

    We need to check that the choice of lifts to $\ctt A \to \ctt B$ and $\ctf A \to \ctf B$ are unique in $\frac{\mcC}{\mcC \cap \mcD}$ and $\frac{\mcD}{\mcC \cap \mcD}$.  Equivalently, we show that any lifts of the zero map $A\to B$ are zero in $\frac{\mcC}{\mcC \cap \mcD}$ and $\frac{\mcD}{\mcC \cap \mcD}$.  Indeed, any choice of lift $\bar g$ of the zero map $A\to B$ factors through some $h\colon \ctt A\to \ctf B$; and $\hat g$ factors through this $h$ as well.  By the first part of this lemma, $h$, and thus $\bar g$ and $\hat g$, factor through an object of $\mcC \cap \mcD$.  Thus $\bar g$ is zero in $\frac{\mcC}{\mcC \cap \mcD}$ and $\hat g$ is zero in $\frac{\mcD}{\mcC \cap \mcD}$.
\end{proof}

The cotorsion pairs we study will usually come from a cotorsion torsion triple.  The following lemma guarantees that in this case, the cotorsion part consists of objects whose projective dimension is at most one.

\begin{lemma} \label{lem.pd=fac}
Let $(\mcC, \mcD)$ be a cotorsion pair in an abelian category $\mcA$. Then $\mcD$ is closed under factor modules if and only if all objects in $\mcC$ have projective dimension at most one.

In particular, if $(\mcC,\mcT,\mcF)$ is a cotorsion torsion triple, then all objects in $\mcC$ have projective dimension at most one.
\end{lemma}

\begin{proof}
    Assume first that $\mcD$ is closed under factor modules. Let $C$ in $\mcC$, and consider any $2$-extension from $C$ to any object $A \in \mcA$, as in the upper line of the following diagram. By the second property of cotorsion pairs there is a monomorphism $E \into \widetilde \ctf E$ with $\widetilde \ctf E \in \mcD$. We construct the lower half of the diagram as the pushout along this map.
    \[
        \begin{tikzpicture}[scale=.6]
            \node (1) at (0,2) {$A$};
            \node (2) at (3,2) {$E$};
            \node (3) at (6,2) {$F$};
            \node (4) at (9,2) {$C$};
            \node (5) at (0,0) {$A$};
            \node (6) at (3,0) {$\widetilde\ctf E$};
            \node (7) at (6,0) {$F'$};
            \node (8) at (9,0) {$C$};
            \tikzpo27
            \path[->,font=\scriptsize]
                (1) edge[right hook->] node[auto] {} (2)
                (2) edge node[auto] {} (3)
                (3) edge[->>] node[auto] {} (4)
                (5) edge[right hook->] node[auto] {} (6)
                (6) edge node[auto] {} (7)
                (7) edge[->>] node[auto] {} (8)
                (1) edge[-, double distance=.5mm] node[auto] {} (5)
                (2) edge[right hook->] node[auto] {} (6)
                (3) edge[right hook->] node[auto] {} (7)
                (4) edge[-, double distance=.5mm] node[auto] {} (8);
        \end{tikzpicture}
    \]
    The image of the map $\widetilde \ctf E \to F'$ lies in $\mcD$, since $\mcD$ is closed under factor modules.
    Thus, by the first property of cotorsion pairs, the $1$-extension from $C$ to this image splits at $F'$, whence our original $2$-extension is trivial.
    Thus we have shown that $\Ext_{\mcA}^2(C, -) = 0$, that is, the projective dimension of $C$ is at most one.

    Assume now that all objects $C$ in $\mcC$ have projective dimension at most one. It follows that $\Ext_{\mcA}^1(C, -)$ is right exact, and thus that $\mcD$ is closed under quotients.

    Finally, assume $(\mcC,\mcT,\mcF)$ is a cotorsion torsion triple.  Since $(\mcT,\mcF)$ is a torsion pair, $\mcT$ is automatically closed under factors.  It follows from the first part that all objects of $\mcC$ has projective dimension at most one.
\end{proof}

The following result (also obtained in \cite{Beligiannis2010Tilting}, see \cref{rem:related_results}) will be of central significance in the sequel, providing the foundation for the main equivalences established in our work.
\begin{thm}
    \label{thm.ctt}
    Let $(\mcC, \mcT, \mcF)$ be a cotorsion torsion triple in an abelian category~$\mcA$. Then $\ctt\colon \mcA \to \frac{\mcC}{\mcC\cap\mcT}$ and $\tfree\colon \mcA\to\mcF$ induce mutually inverse equivalences
    \[
        \mcF \simeq \frac{\mcC}{\mcC \cap \mcT}.
    \]
\end{thm}
\begin{proof}
    The functor $\tfree\colon \mcA \to \mcF$ satisfies $\tfree(\mcT) = 0$, and so its restriction to $\mcC$ induces a unique functor
    \[
        \overline{\tfree|_\mcC}\colon \frac{\mcC}{\mcC\cap\mcT} \to \mcF
    \]
    making the top triangle in the following diagram commute.
    \[
        \begin{tikzpicture}[scale=.75]
            \node (1) at (0,2.5) {$\mcC$};
            \node (3) at (4,2.5) {$\mcF$};
            \node (4) at (0,0) {$\frac{\mcC}{\mcC \cap \mcT}$};
            \node (5) at (4,0) {$\mcA$};
            \path[->,font=\scriptsize]
                (1) edge node[auto] {$\tfree|_\mcC$} (3)
                (5) edge node[auto] {$\ctt$} (4)
                (1) edge node[auto] {} (4)
                (3) edge[right hook->] node[auto] {} (5)
                (3) edge[bend left=6] node[auto] {$\ctt|_\mcF$} (4)
                (4) edge[bend left=6, densely dashed] node[auto] {$\overline{\tfree|_\mcC}$} (3);
        \end{tikzpicture}
    \]
    We claim that the restriction of $\ctt$ in \cref{lemma:cotors_functors_cd} to $\mcF$,
    \[
        \ctt|_\mcF\colon \mcF \to \frac{\mcC}{\mcC\cap\mcT},
    \]
    is a quasi-inverse.

    By the proof of the previous lemma, for any $g\colon A \to B$ in $\mcF$, we obtain a commutative diagram in $\mcA$
    \[
        \begin{tikzpicture}[scale=.75]
            \node (dA) at (0,0) {$\ctf A$};
            \node (cA) at (3,0) {$\ctt A$};
            \node (A) at (6,0) {$A$};
            \node (dB) at (0,-2) {$\ctf B$};
            \node (cB) at (3,-2) {$\ctt B$};
            \node (B) at (6,-2) {$B$,};
            \path[->,font=\scriptsize]
                (dA) edge [right hook->] (cA)
                (cA) edge[->>] (A)
                (dB) edge [right hook->] (cB)
                (cB) edge [->>] (B)
                (dA) edge [densely dashed] node[auto] {$\hat g$} (dB)
                (cA) edge [densely dashed] node[auto] {$\bar g$} (cB)
                (A) edge node[auto] {$g$} (B);
        \end{tikzpicture}
    \]
    where the lifting $\ctt(g) = \bar g\colon \ctt A\to \ctt B$ is unique in $\frac{\mcC}{\mcC\cap\mcT}$.  Since $\hat g\colon \ctf A \to \ctf B$ is in $\mcT$ and $g\colon A \to B$ is in $\mcF$, the functoriality of $\tfree$, pointed out in \cref{rmk:tp}, implies that $\overline{\tfree|_\mcC}\circ\ctt|_\mcF \isom \id_\mcF$.

    For the other composition, let $h\colon C\to C'$ be an arbitrary morphism in $\frac{\mcC}{\mcC \cap \mcT}$, and choose a preimage $k$ in $\mcC$.  In $\mcA$ we obtain the commutative diagram
    \[
        \begin{tikzpicture}[scale=.75]
            \node (1) at (0,2) {$\tors C$};
            \node (2) at (3,2) {$C$};
            \node (3) at (6,2) {$\tfree C$};
            \node (4) at (0,0) {$\tors C'$};
            \node (5) at (3,0) {$C'$};
            \node (6) at (6,0) {$\tfree C'$};
            \path[->,font=\scriptsize]
                (1) edge [right hook->] (2)
                (2) edge[->>] (3)
                (4) edge [right hook->] (5)
                (5) edge [->>] (6)
                (1) edge node[auto] {$\tors k$} (4)
                (2) edge node[auto] {$k$} (5)
                (3) edge node[auto] {$\tfree k$} (6);
        \end{tikzpicture}
    \]
    from the torsion pair $(\mcT,\mcF)$.  Since $C,C'\in \mcC$, this tells us that $k$ is a lift of $\tfree k$.  Because this lift is unique in $\frac{\mcC}{\mcC \cap \mcT}$, we conclude that
    \[
        h = \ctt\circ\tfree(k) = \ctt\circ\overline{\tfree|_\mcC}(h),
    \]
    and so $\ctt|_\mcF\circ\overline{\tfree|_\mcC} = \id_{\frac{\mcC}{\mcC \cap \mcT}}$.
\end{proof}

\begin{ex}
    Let $\mcA = \Rep_\field^\mathrm{fp} \mbR_{\geq 0}$, and let $(\mcC,\mcT,\mcF)$ be the cotorsion torsion triple of \cref{ex:first_cttt}.  Note that
    \[
        \mcC \cap \mcT = \add \big(\{ \field_{\hopen x,\infty} \mid x \geq 1 \} \cup \{ \field_{\hopen 0,y} \mid y < 1 \}  \cup \{ \field_{\hopen 0,\infty} \}\big).
    \]
    We obtain the indecomposable objects in $\frac{\mcC}{\mcC\cap\mcT}$ by removing those in $\mcC\cap\mcT$, so
    \[
        \Ob \Big(\frac{\mcC}{\mcC\cap\mcT}\Big) = \add \{ \field_{\hopen x,\infty} \mid 0 < x < 1 \} \cup \{ \field_{\hopen x,y} \mid 0 < x < y < 1 \}.
    \]
    Recall that
    \[
        \mcF = \add \{ \field_{\hopen x,y} \mid 0 < x < y \leq 1 \}.
    \]
    One may verify that the bijection on objects given by the equivalence $\frac{\mcC}{\mcC\cap\mcT} \simeq \mcF$ of \cref{thm.ctt} is
    \begin{align*}
        \Ob\Big(\frac{\mcC}{\mcC\cap\mcT}\Big) & \longleftrightarrow \Ob(\mcF) \\
        \field_{\hopen x,y} & \longleftrightarrow \field_{\hopen x,y} \\
        \field_{\hopen x,\infty} & \longleftrightarrow \field_{\hopen x,1}
    \end{align*}
    for all $0<x<y<1$.
\end{ex}

\subsection{Torsion and cotorsion pairs and tilting}

In this section we discuss how to produce (co)torsion pairs in a class of very general abelian categories.  To this end we will introduce the notion of (weak) tilting subcategories.  When applying this to our running example, we notice, in \cref{ex.tilting}, that only one out of two very similar subcategories give rise to a good tilting theory.  The deciding factor turns out to be whether or not the subcategory approximates the abelian category well enough.  The following definition makes this condition precise.

\begin{defn}
    Let $\mcE$ be an additive category, with $\mcX\subseteq\mcE$ a full subcategory and $E$ an object of $\mcE$.  A morphism $\varphi\colon X\to E$, with $X\in\mcX$, is said to be a \emph{right $\mcX$-approximation of $E$} if for every $X'\to E$, with $X'\in\mcX$, there exists a morphism $X'\to X$ making the following diagram commute
    \[
        \begin{tikzpicture}[scale=.75]
            \node (1) at (3,2) {$X$};
            \node (2) at (3,0) {$E$,};
            \node (3) at (0,0) {$X'$};
            \path[->,font=\scriptsize]
                (1) edge node[auto] {$\varphi$} (2)
                (3) edge[densely dashed] node[auto] {} (1)
                (3) edge node[auto] {} (2);
        \end{tikzpicture}
    \]
    or, in equivalent terms, $\varphi$ induces an epimorphism of functors
    \[
        \Hom(-,X)|_\mcX \stackrel{\varphi_*}\onto \Hom(-,E)|_\mcX.
    \]
    If every $E\in\mcE$ admits a right $\mcX$-approximation, then $\mcX \subseteq \mcE$ is said to be \emph{contravariantly finite}.

    Dually, the morphism $\psi\colon E \to X$ is a \emph{left $\mcX$-approximation of $E$} if $\psi$ induces an epimorphism of functors
    \[
        \Hom(X,-)|_\mcX \stackrel{\psi^*}\onto \Hom(E,-)|_\mcX,
    \]
    and if every $E$ admits such a left $\mcX$-approximation, $\mcX \subseteq \mcE$ is \emph{covariantly finite}.

    The subcategory $\mcX\subseteq\mcE$ is simply called \emph{functorially finite}, provided it is both covariantly and contravariantly finite.
\end{defn}

The use of the word \emph{finite} stems from the fact that the existence of a right approximation of an object $E$ implies that the functor $\Hom(-,E)|_\mcX \colon \mcX^\op \to \Ab$ is finitely generated, in the following sense.
\begin{defn}
    Let $\mcX$ be an additive category.  An additive functor $F\colon \mcX^\op \to \Ab$ is \emph{finitely generated} provided there is an object $X\in\mcX$ and an epimorphism of functors $\Hom_\mcX(-,X) \onto F$.
\end{defn}

\begin{defn}[{Compare \cite[Proposition~4.3]{MR2786598}}]
    \label{def.tilting}
    Let $\mcA$ be an abelian category with enough projectives. An additively closed full subcategory $\mbT$ of $\mcA$, i.e., one which is closed under taking finite direct sums and summands, is a \emph{weak tilting} subcategory if
    \begin{enumerate}
        \item $\Ext_{\mcA}^1(T_1, T_2) = 0$ for all $T_1, T_2 \in \mbT$.
        \item Any object $T \in \mbT$ has projective dimension at most $1$, that is, it appears in a short exact sequence
            \[
                P_1 \into P_0 \onto T
            \]
            with $P_i$ projective in $\mcA$.
        \item For any $P$ projective in $\mcA$, there is a short exact sequence
            \[
                P \into T^0 \onto T^1
            \]
            with $T^i \in \mbT$.
    \end{enumerate}
    A weak tilting subcategory $\mbT\subseteq \mcA$ is \emph{tilting} provided $\mbT$ is contravariantly finite in $\mcA$.
\end{defn}

\begin{remark} \label{rem.cores_on_sums_enough}
    Clearly we do not need to verify (2) for all objects in $\mbT$: It suffices to do so for a subcollection $\mbT' \subseteq \mbT$ such that $\add \mbT' = \mbT$.

    Similarly we do not need to check (3) for all projectives, it suffices to check it for a subcategory $\mathbb{P} \subseteq \Proj \mcA$ such that $\add \mathbb{P} = \Proj \mcA$. Indeed, the collection of objects having two-term $\mbT$ coresolutions is clearly closed under direct sums. Thus it remains to show that it is also closed under summands. So assume $P \oplus Q$ has a two-term $\mbT$ coresolutions as in the top row of the following diagram.
    \[
        \begin{tikzcd}[ampersand replacement=\&]
            P \oplus Q \ar[r,right hook->,"{(f \; g)}"] \ar[d,equal] \& T^0 \ar[r,->>] \& T^1 \\
            P \oplus Q \ar[r,right hook->,"{\left( \begin{smallmatrix} f & 0 \\ 0 & g \end{smallmatrix} \right)}"] \& (T^0)^2 \ar[r,->>] \ar[u,->>,"{(1\;1)}" swap] \& \coker f \oplus \coker g \ar[u,->>,densely dashed] \\
            \& T_0 \ar[r,equal] \ar[u,right hook->,"{\left( \begin{smallmatrix} \phantom-1 \\ -1 \end{smallmatrix} \right)}" swap] \& T_0 \ar[u,right hook->,densely dashed]
        \end{tikzcd}
    \]
    We can clearly draw the remaining solid part of this diagram, and it follows that there is the dashed short exact sequence in the right column. Since $\Ext_{\mcA}^1(T^1, T_0) \subseteq \Ext_{\mcA}^1(\mbT, \mbT) = 0$ this sequence splits, and thus both $\coker f$ and $\coker g$ are in $\mbT$.
\end{remark}

\begin{remark}
    The existence of enough projectives is not essential for tilting theory.  However we keep this assumption here since it simplifies the formulation of the definition, and it will be satisfied in the examples we are interested in.
\end{remark}

Conditions (1) to (3) are the standard conditions for tilting, while assuming contravariant finiteness is less commonly made explicit.  In fact, we proceed to show that this is automatically satisfied if the abelian category $\mcA$ is \emph{noetherian}, that is, any ascending chain of subobjects of a given object eventually becomes stationary.  We will see below, however, that it is not automatically satisfied in general, even for fairly reasonable categories.

\begin{prop}
    \label{prop.noetherian_tilting}
    Let $\mcA$ be a noetherian abelian category with enough projectives and assume that $\mbT$ is a weak tilting subcategory of $\mcA$.  Then $\mbT$ is automatically contravariantly finite, and therefore already tilting.
\end{prop}

A related result can be found in the proof of \cite[Theorem~2.5(ii)]{MR1458808}.

\begin{proof}
    We need to show that $\mbT \subseteq \mcA$ is contravariantly finite, i.e., for any given object $M \in \mcA$, the functor
    \[
        \Hom_\mcA(-, M)|_{\mbT}\colon \mbT^\op \to \Ab
    \]
    is finitely generated.  To this end, first observe that by noetherianity, $M$ has a (unique) maximal subobject $\tors M$ that is an epimorphic image of an object $T \in \mbT$.  This means that for $X\in \mbT$ and $f\colon X\to M$ the factorization through $\tors M$ in
    \[
        \begin{tikzpicture}[scale=.75]
            \node (1) at (0,2) {$X$};
            \node (2) at (2,1) {$\im f$};
            \node (3) at (4,2) {$M$};
            \node (4) at (4,0) {$\tors M$};
            \path[->,font=\scriptsize]
                (1) edge[->>] node[auto] {} (2)
                (2) edge[right hook->] node[auto] {} (3)
                (1) edge[bend left=15] node[auto] {$f$} (3)
                (2) edge[right hook->, densely dashed] node[auto] {} (4)
                (4) edge[right hook->] node[auto] {} (3);
        \end{tikzpicture}
    \]
    exists and is unique.  From this follows that $\Hom_\mcA(-, M)|_{\mbT} \isom \Hom_\mcA(-, \tors M)|_{\mbT}$: Any morphism $f\colon X \to M$ is sent to the composition $X\onto \im f\into \tors M$.  The other direction is then clearly given by sending $X\to \tors M$ to $X\to \tors M\into M$.

    Denote by $K$ the kernel of the epimorphism $T \stackrel\varphi\onto \tors M$.  Then we obtain an exact sequence
    \[
        \Hom_\mcA(-, T)|_{\mbT} \stackrel{\varphi_*}\to \Hom_\mcA(-, \tors M)|_{\mbT} \to \Ext_\mcA^1(-, K)|_{\mbT} \to \Ext_\mcA^1(-, T)|_{\mbT}.
    \]
    Since $\Hom_\mcA(-, M)|_{\mbT} \isom \Hom_\mcA(-, \tors M)|_{\mbT}$ and the last term vanishes by (1), we obtain a short exact sequence
    \[
        \im \varphi_* \into \Hom_\mcA(-, M)|_{\mbT} \onto \Ext_\mcA^1(-, K)|_{\mbT},
    \]
    where the image of $\varphi_*$ is finitely generated via $\Hom_\mcA(-, T)|_{\mbT} \onto \im \varphi_*$.  Thus by the horseshoe lemma, $\Hom_\mcA(-, M)|_{\mbT}$ is finitely generated provided $\Ext_\mcA^1(-, K)|_{\mbT}$ is.

    Let $P \onto K$ be an epimorphism from a projective in $\mcA$.  Since the projective dimension of $\mbT$ is at most $1$ by (2), we obtain an epimorphism $\Ext_\mcA^1(-,P)|_{\mbT} \onto \Ext_\mcA^1(-, K)|_{\mbT}$.  Thus it suffices to check that $\Ext_\mcA^1(-, P)|_{\mbT}$ is finitely generated for $P$ projective.  So let $P \into T^0 \onto T^1$ be the sequence from (3).  This sequence induces a long exact sequence, whose relevant part is
    \[
        \Hom_\mcA(-, T^1)|_{\mbT} \to \Ext_\mcA^1(-, P)|_{\mbT} \to \Ext_\mcA^1(-,T^0)|_{\mbT}.
    \]
    Since the last term is $0$ by (1), $\Ext_\mcA^1(-,P)|_{\mbT}$ is finitely generated and the claim follows.
\end{proof}

\begin{ex}
    \label{ex.tilting}
    Consider the abelian category $\Rep_\field^\mathrm{fp} \mbR_{\geq 0}$ of finitely presented (covariant) representations of $\mbR_{\geq 0}$ described in \cref{ex:repsofR}.  Recall that the functors $\field_{\hopen x,\infty}$ are projective, while it can be shown that the functors $\field_{\hopen 0,x}$ are injective, for any $x\in\mbR_{\geq 0}$.

    Define two subcategories by
    \begin{align*}
        \mbT_{>1} & = \add \big(\{ \field_{\hopen x,\infty} \mid x > 1 \} \cup \{ \field_{\hopen 0,y} \mid y \leq 1 \} \cup \{ \field_{\hopen 0,\infty} \}\big), \quad \text{and} \\
        \mbT_{\geq 1} & = \add \big(\{ \field_{\hopen x,\infty} \mid x \geq 1 \} \cup \{ \field_{\hopen 0,y} \mid y < 1 \}  \cup \{ \field_{\hopen 0,\infty} \}\big).
    \end{align*}
    Thus both subcategories consist of a projective object $\field_{\hopen x,\infty}$ for each $x>1$, an injective object $\field_{\hopen 0,y}$ for each $y<1$ and the projective-injective $\field_{\hopen 0,\infty}$.  The only difference is that we in addition include $\field_{\hopen 0,1}$ in the former subcategory and $\field_{\hopen 1,\infty}$ in the latter.

    We claim that $\mbT_{>1}$ is weakly tilting: Every indecomposable object of $\mbT_{>1}$ is either projective or injective, so that $\Ext^1$ vanishes, and each indecomposable object has projective dimension at most $1$ by \eqref{eq:presentation}.  Lastly, every indecomposable projective in $\Rep_\field^\mathrm{fp} \mbR_{\geq 0}$ is of form $\field_{\hopen a,\infty}$, which appear in short exact sequences
    \[
    \begin{cases}
        \field_{\hopen a,\infty} \stackrel\sim\to \field_{\hopen a,\infty} \to 0, & a>1 ,\\
        \field_{\hopen a,\infty} \into \field_{\hopen 0,\infty} \onto \field_{\hopen 0,a}, & a\leq1,
    \end{cases}
    \]
    showing that (3) also holds.  Similarly, $\mbT_{\geq 1}$ is weakly tilting.

    Moreover, $\mbT_{\geq1}$ is contravariantly finite, i.e., is tilting.  To see this, we first note that $\Hom(\field_{\hopen c,d},\field_{\hopen a,b}) = \field$ if and only if $a\leq c<b\leq d$, and zero otherwise.  With this in mind, it is straightforward to check that any indecomposable representation $\field_{\hopen a,b}$ has a right $\mbT_{\geq1}$-approximation, namely
    \[
    \begin{cases}
        \hfill \field_{\hopen 0,b} \to \field_{\hopen 0,b},&  a=0<b<1 \\
        \hfill \field_{\hopen 0,\infty} \to \field_{\hopen 0,b},&  a=0, ~ 1\leq b\leq\infty \\
        \hfill 0 \to \field_{\hopen a,b}, & 0<a<b<1 \\
        \hfill \field_{\hopen 1,\infty} \to \field_{\hopen a,b},&  0<a<1\leq b\leq\infty \\
        \hfill \field_{\hopen a,\infty} \to \field_{\hopen a,b},&  1\leq a<b\leq\infty.
    \end{cases}
    \]

    On the other hand, $\mbT_{>1}$ is not contravariantly finite, and therefore not tilting:  Indeed one observes that $\field_{\hopen 1,\infty}$ has no right $\mbT_{>1}$-approximation.  A right approximation would have to contain summands of the form $\field_{\hopen x,\infty} \to \field_{\hopen 1,\infty}$, for some $x>1$, but then any $\field_{\hopen y,\infty}$ with $x>y>1$ would violate the lifting property.
\end{ex}

Recall that the \emph{factor category} $\Fac \mcX \subseteq \mcA$ is the full subcategory consisting of factor objects of finite direct sums of objects in $\mcX$.
The following result is well known in classical tilting theory; see for instance \cite[Lemma~4.5]{MR1327209}.  We include a proof for later reference.
\begin{prop} \label{prop:torsionpair}
    Let $\mcA$ be an abelian category with enough projectives.  If $\mbT\subseteq\mcA$ is weakly tilting, then $\Fac \mbT = \mbT^{\perp_1}.$  If $\mbT$ is additionally contravariantly finite, i.e., tilting, then \[( \Fac \mbT, \mbT^{\perp} )\] is a torsion pair.
\end{prop}
\begin{proof}
    Let $X \in \Fac\mbT$, so there is an epimorphism $T\onto X$, where $T$ is an object of $\mbT$.  Since each object of $\mbT$ has projective dimension at most $1$ by (2), $\Ext^1_\mcA(\mbT,-)$ is a right exact functor.  Thus $\Ext^1_\mcA(\mbT,T) \onto \Ext^1_\mcA(\mbT,X)$ is an epimorphism, so $X \in \mbT^{\perp_1}$ by (1).  For the converse, let $P \onto X$ be an epimorphism from a projective object in $\mcA$.  By (3) we can form the following diagram, with $T^i \in \mbT$, by taking a pushout:
    \[
        \begin{tikzpicture}[scale=.75]
            \node (1) at (0,2) {$P$};
            \node (2) at (3,2) {$T^0$};
            \node (3) at (6,2) {$T^1$};
            \node (4) at (0,0) {$X$};
            \node (5) at (3,0) {$Y$};
            \node (6) at (6,0) {$T^1$};
            \tikzpo15
            \path[->,font=\scriptsize]
                (1) edge[right hook->] node[auto] {} (2)
                (2) edge[->>] node[auto] {} (3)
                (1) edge[->>] node[auto] {} (4)
                (2) edge[->>] node[auto] {} (5)
                (3) edge[-, double distance=.5mm] node[auto] {} (6)
                (4) edge[right hook->] node[auto] {} (5)
                (5) edge[->>] node[auto] {} (6);
        \end{tikzpicture}
    \]
    Thus $Y \in\Fac \mbT$, so since every extension of $T^1$ by $X$ splits, $X$ is a summand of $Y$, and thus in $\Fac \mbT$.  This completes the proof of $\Fac \mbT = \mbT^{\perp_1}$.

    Now we show that $( \Fac \mbT, \mbT^{\perp} )$ is a torsion pair.  Let $X \in \Fac\mbT$ and choose an epimorphsim $T \onto X$ with $T \in \mbT$.  This gives rise to a monomorphism $\Hom_\mcA(X,Y) \into \Hom_\mcA(T,Y)$, showing that $\Hom_\mcA(X,Y) = 0$ whenever $Y \in \mbT^\perp$.

    For an arbitrary object $A \in \mcA$, contravariant finiteness guarantees the existence of a right $\mbT$-approximation $\varphi\colon T \to A$.  The short exact sequence
    \[
        \im\varphi \stackrel i\into A \onto \coker\varphi,
    \]
    where $\im\varphi \in \Fac\mbT$, yields a long exact sequence
    \[
        \begin{split}
            0 \to \Hom_\mcA(-,\im\varphi)|_{\mbT} \stackrel{i_*} \to \Hom_\mcA(-,A)|_{\mbT} \to \Hom_\mcA(-,\coker\varphi)|_{\mbT} \phantom.
            \\ \to \Ext^1_\mcA(-,\im\varphi)|_{\mbT}.
        \end{split}
    \]
    Now $i_*$ is an epimorphism since $\varphi$ is an approximation, and the rightmost term is $0$ since $\im\varphi \in \Fac\mbT = \mbT^{\perp_1}$.  Thus $\coker\varphi \in \mbT^\perp$.
\end{proof}

\begin{ex} \label{ex.conc_torsion_from_tilting}
    Consider $\mcA = \Rep_\field^\mathrm{fp} \mbR_{\geq 0}$, and
    \[
        \mbT_{\geq 1} = \add \big(\{ \field_{\hopen x,\infty} \mid x \geq 1 \} \cup \{ \field_{\hopen 0,x} \mid x < 1 \}  \cup \{ \field_{\hopen 0,\infty} \}\big)
    \]
    as in \cref{ex.tilting}. Then the associated torsion pair is given by
    \begin{align*}
        \mcT & = \Fac \mbT_{\geq 1} = \add \big(\{ \field_{\hopen 0,x} \mid 0 < x \leq \infty \} \cup \{ \field_{\hopen x,y} \mid 1 \leq x < y \leq \infty \}\big) \\
        \mcF & = \mbT_{\geq 1}^{\perp} = \add \big(\{ \field_{\hopen x,y} \mid 0 < x < y \leq 1 \}\big).
    \end{align*}
    In other words, we recover the torsion pair in \cref{ex:first_tp}.
\end{ex}

\begin{remark}
    We used contravariant finiteness of $\mbT$ to show existence of the short exact sequence $\tors A \into A \onto \tfree A$.  The following example shows that a weak tilting subcategory is, in general, not sufficient to produce a torsion pair.
\end{remark}

\begin{ex}
    \label{ex.not_torsion_from_not_tilting}
    If one were to try the same construction with the subcategory
    \[
        \mbT_{>1} = \add \big(\{ \field_{\hopen x,\infty} \mid x > 1 \} \cup \{ \field_{\hopen 0,x} \mid x \leq 1 \} \cup \{ \field_{\hopen 0,\infty} \}\big),
    \]
    also from \cref{ex.tilting}, but which is not contravariantly finite, we observe that
    \begin{align*}
        \Fac \mbT_{>1} & = \add \big(\{ \field_{\hopen 0,x} \mid 0 < x \leq \infty \} \cup \{ \field_{\hopen x,y} \mid 1 < x < y \leq \infty \}\big) \\
        \mbT_{>1}^{\perp} & = \add \big(\{ \field_{\hopen x,y} \mid 0 < x < y \leq 1 \}\big).
    \end{align*}
    This pair of subcategories does not form a torsion pair: Either abstractly, because $\mbT_{\geq1}^\perp = \mbT_{>1}^\perp$ and two different torsion pairs cannot have the same torsion-free class, or concretely, because $\field_{\hopen 1,\infty}$ does not appear as an extension between these two subcategories.
\end{ex}

To contrast the construction of torsion pairs from tilting subcategories as in \cref{prop:torsionpair},
we now show that \emph{weak} tilting subcategories are already sufficient to produce cotorsion pairs; cf. \cite[Theorem~5.5]{MR1097029}
\begin{prop} \label{prop.cotorsion_from_tilting}
    Let $\mbT$ be weakly tilting in an abelian category $\mcA$ with enough projectives. Then
    \[
        \big({}^{\perp_1}( \Fac \mbT ), \Fac \mbT\big)
    \]
    is a cotorsion pair, and moreover
    \[ {}^{\perp_1}( \Fac \mbT ) = \{ X \in {}^{\perp_1}\mbT \mid \pdim X \leq 1 \}. \]

    If $\mbT$ is additionally contravariantly finite, i.e., tilting, then the intersection of the cotorsion and cotorsion-free part is
    \[
        {}^{\perp_1}( \Fac \mbT ) \cap \Fac \mbT = \mbT.
    \]
\end{prop}

\begin{proof}
    Showing that the first condition of a cotorsion pair holds amounts to showing
    \[
        \Fac \mbT = \big( {}^{\perp_1} (\Fac \mbT) \big)^{\perp_1}.
    \]
    The first part of the proof of \cref{prop:torsionpair} shows that $\mbT \subseteq {}^{\perp_1}(\Fac\mbT)$.  Thus $\mbT^{\perp_1} \supseteq \big({}^{\perp_1}(\Fac\mbT)\big)^{\perp_1}$.
    Clearly, for any subcategory $\mcS$, we have $\mcS \subseteq ({}^{\perp_1}\mcS)^{\perp_1}$, so that in particular
    \[
        \Fac\mbT \subseteq \big({}^{\perp_1}(\Fac\mbT)\big)^{\perp_1} \subseteq \mbT^{\perp_1}.
    \]
    Then the maximal $\Ext^1$-orthogonality follows from $\Fac \mbT = \mbT^{\perp_1}$, by \cref{prop:torsionpair}.

    To conclude that this is a cotorsion pair, it thus remains to construct the two short exact sequences, as above, for a given object $A \in \mcA$. First consider a projective presentation
    \[
        P_1 \to P_0 \onto A
    \]
    together with the short exact sequence
    \[
        P_1 \into T^0 \onto T^1
    \]
    from the definition of tilting. Forming a pushout we obtain a diagram
    \[
        \begin{tikzpicture}[scale=.75]
            \node (1) at (0,4) {$P_1$};
            \node (2) at (3,4) {$T^0$};
            \node (3) at (6,4) {$T^1$};
            \node (4) at (0,2) {$P_0$};
            \node (5) at (3,2) {$V$};
            \node (6) at (6,2) {$T^1$};
            \tikzpo15
            \path[->,font=\scriptsize]
                (1) edge[right hook->] (2)
                (2) edge[->>] (3)
                (4) edge[right hook->] (5)
                (5) edge[->>] (6)
                (1) edge (4)
                (2) edge[auto] node {$f$} (5)
                (3) edge[-, double distance=.5mm] (6);
        \end{tikzpicture}
    \]
    Since $P_0$ and $T^1$ are in ${}^{\perp_1}(\Fac \mbT)$, and this subcategory is extension closed, the pushout $V$ also lies in ${}^{\perp_1}(\Fac \mbT)$.  The image of $f\colon T^0 \to V$ is a quotient of $T^0$, hence in $\Fac \mbT$.  Since $\coker f \isom A$, we thus get the first desired short exact sequence as
    \[
        \im f \into V \onto A.
    \]

    For the second short exact sequence, we consider $P_0 \into \bar T^0 \onto \bar T^1$, again from the definition of tilting, and the pushout diagram
    \[
        \begin{tikzpicture}[scale=.75]
            \node (1) at (0,4) {$P_0$};
            \node (2) at (3,4) {$\bar T^0$};
            \node (3) at (6,4) {$\bar T^1$};
            \node (4) at (0,2) {$A$};
            \node (5) at (3,2) {$\bar V$};
            \node (6) at (6,2) {$\bar T^1$};
            \tikzpo15
            \path[->,font=\scriptsize]
                (1) edge[right hook->] (2)
                (2) edge[->>] (3)
                (4) edge[right hook->] (5)
                (5) edge[->>] (6)
                (1) edge[->>] (4)
                (2) edge[->>] (5)
                (3) edge[-, double distance=.5mm] (6);
        \end{tikzpicture}
    \]
    Here the second horizontal sequence is the desired one, since the pushout is in $\Fac \mbT$ and $\bar T^1 \in \mbT \subseteq {}^{\perp_1}(\Fac \mbT)$.

    \medskip
    For the equality ${}^{\perp_1} (\Fac \mbT) = \{ X \in {}^{\perp_1}\mbT \mid \pdim X \leq 1 \}$, first note that the inclusion `$\subseteq$' holds by \cref{lem.pd=fac}. For the converse inclusion, let $X \in {}^{\perp_1} \mbT$ be of projective dimension at most one.  Let $F \in \Fac \mbT$, and choose $\varphi \colon T \onto F$ be any epimorphism with $T\in\mbT$.  Then the exact sequence
        \[
            \Ext^1_\mcA(X,T) \to \Ext^1_\mcA(X,F) \to \Ext^2_\mcA(X,\ker\varphi)
        \]
        shows that $\Ext^1_\mcA(X,F) = 0$ since the two outer terms are zero. It follows that $X \in {}^{\perp_1}( \Fac \mbT )$.

    \medskip
    We now assume that $\mbT$ is tilting and let $X \in {}^{\perp_1}( \Fac \mbT ) \cap \Fac \mbT$. Since $X \in \Fac \mbT$, i.e., there exists an epimorphism from an object in $\mbT$, and since $\mbT$ is contravariantly finite, there is a right $\mbT$-approximation $T \onto X$, which necessarily is an epimorphism.  By definition, the kernel $K$ of this approximation satisfies $\Ext^1_\mcA(-,K)|_{\mbT} = 0$, or, in other words, $K \in \mbT^{\perp_1} = \Fac \mbT$. But then the assumption $X \in {}^{\perp_1}( \Fac \mbT )$ forces this short exact sequence to split, so $K \oplus X \isom T$, whence $X \in \mbT$.
\end{proof}

\begin{ex}
    \label{ex.conc_cotorsion_from_tilting}
    Consider $\mcA = \Rep_\field^\mathrm{fp} \mbR_{\geq 0}$, and
    \[
        \mbT_{\geq 1} = \add \big(\{ \field_{\hopen x,\infty} \mid x \geq 1 \} \cup \{ \field_{\hopen 0,x} \mid x < 1 \}  \cup \{ \field_{\hopen 0,\infty} \}\big)
    \]
    as in \cref{ex.tilting}. Then the associated cotorsion pair is given by
    \begin{align*}
        \mcC & = {}^{\perp_1}(\Fac \mbT_{\geq 1}) = \add \big(\{ \field_{\hopen x,\infty} \mid 0\leq x < \infty \} \cup \{ \field_{\hopen x,y} \mid 0 \leq x < y < 1 \}\big) \\
        \mcD & = \Fac \mbT_{\geq 1} = \add \big(\{ \field_{\hopen 0,x} \mid 0 < x \leq \infty \} \cup \{ \field_{\hopen x,y} \mid 1 \leq x < y \leq \infty \}\big),
    \end{align*}
    and so we recover the cotorsion pair of \cref{ex.first_ctp}.

    As $\mbT_{>1}$ of \cref{ex.tilting} is weakly tilting, the previous proposition applies and one obtains a very similar example of a cotorsion pair.
\end{ex}

So far, we have seen that tilting subcategories give rise to both a torsion and a cotorsion pair. The next proposition (also obtained in \cite{Beligiannis2010Tilting}, see \cref{rem:related_results}) gives a converse, at least in the presence of both.

\begin{prop}
    \label{prop.cct_to_tilt}
    Let $\mcA$ be an abelian category with enough projectives, and $(\mcC, \mcT, \mcF)$ be a cotorsion torsion triple in $\mcA$.

    Then $\mcC \cap \mcT$ is a tilting subcategory.
\end{prop}

\begin{proof}
    First we observe that, by definition of cotorsion pairs, we have
    \[
        \Ext_{\mcA}^1( \mcC \cap \mcT, \mcC \cap \mcT) \subseteq  \Ext_{\mcA}^1( \mcC, \mcT) = 0.
    \]
    Next, by \cref{lem.pd=fac}, the projective dimension of any object in $\mcC$ is at most one, so in particular the same holds for any object in $\mcC \cap \mcT$.

    To verify the third point of the definition of tilting, let $P$ be projective, and consider a short exact sequence
    \[
        P \into \widetilde \ctf P \onto \widetilde \ctt P
    \]
    with $\widetilde \ctf P \in \mcT$ and $\widetilde \ctt P \in \mcC$ as in the definition of a cotorsion pair. Since $\mcT$ is closed under factor modules, we observe that $\widetilde \ctt P \in \mcT$. Since $P \in \mcC$ and $\mcC$ is closed under extensions it follows that also $\widetilde \ctf P \in \mcC$. Thus both $\widetilde \ctf P$ and $\widetilde \ctt P$ are in $\mcC \cap \mcT$.

    Finally we show that $\mcC \cap \mcT$ is contravariantly finite in $\mcA$. Let $A \in \mcA$. Any map from $\mcT$ to $A$ factors through $\tors A$. Consider the short exact sequence
    \[ \ctf \tors A \into \ctt \tors A \onto \tors A \]
    with  $\ctf \tors A \in \mcT$ and $\ctt \tors A \in \mcC$ from the definition of a cotorsion pair. Observe that both outer terms lie in $\mcT$, whence so does the middle, i.e., $\ctt \tors A \in \mcC \cap \mcT$.

    We complete the proof by showing that the composition $\ctt\tors A \onto \tors A \into A$ is a right $\mcC\cap\mcT$-approximation.  Let $X \to A$ be any morphism, where $X \in \mcC\cap\mcT$.  Since $X \in \mcT$ we first observe that $X \to A$ factors through $\tors A \into A$.  Next, since $X \in \mcC$ we have $\Ext_{\mcA}^1( X, \ctf \tors A) = 0$, whence the map further factors through $\ctt \tors A \onto \tors A$.
\end{proof}

Combining our previous results, we obtain the following correspondence (also obtained in \cite{Beligiannis2010Tilting}, see \cref{rem:related_results}).

\begin{thm}
    \label{thm.ctt=tilting}
    Let $\mcA$ be an abelian category with enough projectives. Then the two constructions
    \begin{align*}
        \{ \text{tilting subcategories} \} &\leftrightarrow \{ \text{cotorsion torsion triples} \} \\
        \mbT &\mapsto \big( \{ X \in {}^{\perp_1}\mbT \mid \pdim X \leq 1 \} ,\Fac \mbT, \mbT^{\perp} \big) \\
        \mcC\cap\mcT &\mathrel{\reflectbox{$\mapsto$}} (\mcC,\mcT,\mcF)
    \end{align*}
    give mutually inverse bijections between the collection of cotorsion torsion triples and the collection of tilting subcategories.
\end{thm}

\begin{proof}
    We have seen in \cref{prop:torsionpair,prop.cotorsion_from_tilting} that the map from left to right is well-defined, and in \cref{prop.cct_to_tilt} that the map from right to left is. Thus it only remains to check that they are mutually inverse.

    Starting with a cotorsion torsion triple $(\mcC, \mcT, \mcF)$ it suffices to verify that $\mcT = \Fac ( \mcC \cap \mcT )$, since cotorsion torsion triples are determined by any of their parts. The inclusion `$\supseteq$' is immediate since $\mcT$ is closed under quotients. For `$\subseteq$', let $T \in \mcT$ and consider the short exact sequence $\ctf T \into \ctt T \onto T$, with $\ctf T \in \mcT$ and $\ctt T \in \mcC$. Since $\mcT$ is closed under extensions, we infer that $\ctt T \in \mcC \cap \mcT$, whence $T \in \Fac ( \mcC \cap \mcT )$.

    Conversely, starting with a tilting subcategory $\mbT$, \cref{prop.cotorsion_from_tilting} tells us that $\mbT$ is recovered as the intersection of cotorsion and cotorsion-free part of the associated cotorsion pair.
\end{proof}

Combining this with \cref{thm.ctt}, we get:

\begin{cor}
    \label{cor:equiv_ctt}
    Let $\mbT$ be a tilting subcategory in an abelian category $\mcA$ with enough projective objects.  Then there is an equivalence
    \[
        \frac{\{X \in {}^{\perp_1}\mbT \mid \pdim X \leq 1 \}}{\mbT} \simeq \mbT^\perp.
    \]
\end{cor}

\begin{ex}
    Let $\mcA = \Rep_\field^\mathrm{fp} \mbR_{\geq 0}$, and $\mbT_{\geq 1}$ as in \cref{ex.tilting}. We have already described $\mcC$, $\mcT$, and $\mcF$ in \cref{ex.conc_torsion_from_tilting,ex.conc_cotorsion_from_tilting} above. We obtain the indecomposables for $\mcC / \mbT_{\geq 1}$ by removing those in $\mbT_{\geq1}$, so
    \[
        \mathrm{Ob} (\mcC / \mbT_{\geq 1}) = \add \{ \field_{\hopen x,\infty} \mid 0 < x < 1 \} \cup \{ \field_{\hopen x,y} \mid 0 < x < y < 1 \}.
    \]
    Recall that
    \[
        \mcF = \add \{ \field_{\hopen x,y} \mid 0 < x < y \leq 1 \}.
    \]
    One may verify that the bijection on objects given by the equivalence $\mcC / \mbT_{\geq 1} \simeq \mcF$ of \cref{thm.ctt} is
    \begin{align*}
        \mathrm{Ob} (\mcC / \mbT_{\geq 1}) & \longleftrightarrow \mathrm{Ob}(\mcF) \\
        \field_{\hopen x,y} & \longleftrightarrow \field_{\hopen x,y} \\
        \field_{\hopen x,\infty} & \longleftrightarrow \field_{\hopen x,1}
    \end{align*}
    for all $0<x<1$.
\end{ex}

\subsection{Summary of dual statements}
Each result discussed so far in this section has dual results (with dual proofs, also obtained in \cite{Beligiannis2010Tilting}, see \cref{rem:related_results}).  We summarize the results below, keeping in mind that torsion pairs and cotorsion pairs are self-dual structures.

Suppose $\mcA$ is an abelian category with enough injectives.  An additively closed full subcategory $\mbC$ of $\mcA$ is a \emph{weak cotilting} subcategory if and only if it is weakly tilting in $\mcA^\op$.  In other words, it needs to satisfy the dual requirements
\begin{enumerate}
    \item $\Ext_{\mcA}^1(C_1, C_2) = 0$ for all $C_1, C_2 \in \mbC$.
    \item Any object $C \in \mbC$ has injective dimension at most $1$.
    \item For any $I$ injective in $\mcA$, there is a short exact sequence
        \[
            C^1 \into C^0 \onto I
        \]
        with $C^i \in \mbC$.
\end{enumerate}
A weak cotilting subcategory $\mbC \subseteq \mcA$ is \emph{cotilting} if it is additionally covariantly finite in $\mcA$.

The \emph{subobject category} $\Sub \mcX \subseteq \mcA$ is the full subcategory determined by subobjects of finite direct sums of objects in $\mcX$.

\begin{prop}[Dual of \cref{prop:torsionpair,prop.cotorsion_from_tilting}]
    \label{prop:dualtorsionpair}
    Let $\mcA$ be an abelian category with enough injectives.  If $\mbC \subseteq \mcA$ is weakly cotilting, then $\Sub\mbC = {}^{\perp_1} \mbC$,  $( \Sub \mbC )^{\perp_1} = \{X \in \mbC^{\perp_1} \mid \idim X \leq 1 \}$, and
    \[
        \big(\Sub \mbC, (\Sub \mbC)^{\perp_1}\big)
    \]
    is a cotorsion pair.
    If $\mbC$ is additionally cotilting, then
    \[
        ({}^\perp\mbC, \Sub\mbC)
    \]
    is a torsion pair and
    \[
        ( \Sub \mbC )^{\perp_1} \cap \Sub \mbC = \mbC.
    \]
\end{prop}

As expected, by \emph{torsion cotorsion triple} in an abelian category $\mcA$ we mean a triple of subcategories $(\mcT, \mcF, \mcD)$ where $(\mcT, \mcF)$ is a torsion pair and $(\mcF, \mcD)$ is a cotorsion pair.

\begin{thm}[Dual of \cref{thm.ctt}]
    \label{thm:tct}
    Let $(\mcT, \mcF, \mcD)$ be a torsion cotorsion triple in an abelian category.  Then $\tors\colon \mcA \to \mcT$ and $\widetilde\ctf\colon \mcA \to \frac{\mcD}{\mcD\cap\mcF}$ induce mutually inverse equivalences
    \[ \mcT \simeq \frac{\mcD}{\mcD \cap \mcF}.\]
\end{thm}

\begin{thm}[Dual of \cref{thm.ctt=tilting,cor:equiv_ctt}]
    \label{thm.tct_from_tilting}
    Let $\mcA$ be an abelian category with enough injectives.  Cotilting subcategories $\mbC$ correspond bijectively to torsion cotorsion triples $(\mcT,\mcF,\mcD)$ by
    \begin{align*}
        \mbC &\mapsto \big( {}^\perp\mbC, \Sub \mbC, \{ X \in \mbC^{\perp_1} \mid \idim X \leq 1 \} \big) \\
        \mcF\cap\mcD &\mathrel{\reflectbox{$\mapsto$}} (\mcT,\mcF,\mcD),
    \end{align*}
    and for any cotilting subcategory $\mbC$, there is an equivalence
    \[
        \frac{\{X \in \mbC^{\perp_1} \mid \idim X \leq 1 \}}{\mbC} \simeq {}^\perp\mbC.
    \]
\end{thm}

\subsection{Artin algebras}

In this subsection, we specialize to the case that our abelian category $\mcA$ is the category $\mod \Lambda$ of finitely generated modules over some Artin $R$-algebra $\Lambda$.
The main advantage over the more general setup discussed previously is the presence of Auslander--Reiten theory.  None of the results of this section will be used in the remaining parts of the paper, but they give a nice complement to \cref{thm.ctt} and explain why the AR-quivers in \cref{fig:compare-ar-quivers} agree.  Throughout this section we assume familiarity with the basic concepts of Auslander--Reiten theory; see \cite{MR1476671} for an introduction.

Recall that the duality $D\colon \mod\Lambda \to \mod\Lambda^\op$ is defined as $D = \Hom_R(-, E)$, where $E$ is an injective envelope of $R/\operatorname{rad} R$.

Using Auslander--Reiten theory in $\mod \Lambda$ we obtain a version of \cref{thm.ctt}:

\begin{thm} \label{thm.tau_equivalence}
    Let $\Lambda$ be an Artin algebra, and  $(\mcC, \mcT, \mcF)$ a cotorsion torsion triple in $\mod \Lambda$. Then the Auslander--Reiten translation $\tau$ defines an equivalence
    \[
        \frac{\mcC}{\proj \Lambda} \overset{\simeq}{\to} \mcF.
    \]
\end{thm}

\begin{proof}
    By general Auslander--Reiten theory it is known that $\tau$ defines an equivalence
    \[
        \frac{\mod \Lambda}{\proj \Lambda} \overset{\simeq}{\to} \frac{\mod \Lambda}{\inj \Lambda}.
    \]
    Clearly $\frac{\mcC}{\proj \Lambda}$ is a full subcategory of the left hand side, and we observe that $\mcF$ is a full subcategory of the right hand side: Since $\inj \Lambda \subseteq \mcT$ we have $\Hom_{\Lambda}(\inj \Lambda, \mcF) = 0$, and thus no non-zero maps between objects in $\mcF$ factor through injective modules.

    Therefore it suffices to check that a module $M$ lies in $\mcC$ if and only if $\tau M$ lies in~$\mcF$. We observe that
    \[
        M \in \mcC \Longleftrightarrow \Ext_{\Lambda}^1(M, \mcT) = 0 \Longleftrightarrow \Hom_\Lambda(\mcT, \tau M) = 0 \Longleftrightarrow \tau M \in \mcF,
    \]
    where the middle equivalence follows from Auslander--Reiten duality: A priori Auslander--Reiten duality says that $\Ext_{\Lambda}^1(M, \mcT) = D \Hom_{\frac{\mod \Lambda}{\inj \Lambda}}(\mcT, \tau M)$, but since $\mcT$ is closed under quotients one sees that if there is a non-zero map from $\mcT$ to $\tau M$ then the inclusion of its image is also a map from $\mcT$ to $\tau M$ and is non-zero even in the quotient category modulo injectives.
\end{proof}

\begin{thm}
    \label{thm.equ_tau_commute}
    Let $\Lambda$ be an Artin algebra, and $(\mcC, \mcT, \mcF)$ a cotorsion torsion triple in $\mod \Lambda$. Then $\mcC$ has almost split sequences, and the triangle of functors
    \[
        \begin{tikzcd}[row sep=1mm,column sep=2cm]
            \frac{\mcC}{\proj \Lambda} \ar[rd,"\tau_{\Lambda}"] \ar[dd,"\tau_{\mcC}" swap] & \\
            & \mcF \\
            \frac{\mcC}{\mcC \cap \mcT} \ar[ru,"\tfree" swap]
        \end{tikzcd}
    \]
    commutes, where the functors to the right are the equivalences of \cref{thm.tau_equivalence,thm.ctt}, and the vertical functor is induced by the internal Auslander--Reiten translation in $\mcC$.

    In particular, $\tau_\mcC\colon \frac{\mcC}{\proj \Lambda} \to \frac{\mcC}{\mcC \cap \mcT}$ is an equivalence.
\end{thm}

\begin{proof}
    The category $\mcC$ has almost split sequences by \cite[Corollary~2.9]{MR2009441}. Note that the of cotorsion pairs in \cite{MR2009441} is slightly stronger than ours: all extensions from the first to the second category are required to vanish, not just first extensions. However, given a cotorsion torsion triple, the cotorsion pair is actually a cotorsion pair in the sense of \cite{MR2009441}, since $\mcT$ is closed under cosyzygies.

    The following construction of almost split sequences in $\mcC$ is very similar to \cite[Corollary~3.5]{MR617088}:
    Let $C \in \mcC$ indecomposable and not projective. Consider the almost split sequence in $\mod \Lambda$ ending in $C$, as in the top row of the following diagram.
    \[
        \begin{tikzcd}
            \tau C \ar[r,right hook->] & E \ar[r,->>] & C \ar[d,equal] \\
            \ctt (\tau C) \ar[u,->>] \ar[r,right hook->,densely dashed] & F \ar[r,->>,densely dashed] \ar[u,->>,densely dashed] & C \\
            \ctf (\tau C) \ar[u,right hook->] \ar[r,equal] & \ctf(\tau C) \ar[u,right hook->,densely dashed]
        \end{tikzcd}
    \]
    The left vertical short exact sequence comes from the definition of cotorsion pair, see \cref{def.cotorsion}. Since $\Ext_{\Lambda}^2\big(C, \ctf(\tau C)\big) \subseteq \Ext_{\Lambda}^2( \mcC, \mcT ) = 0$, we can complete the diagram as indicated with the dashed arrows above.

    Since $E \to C$ is right almost split in $\mod \Lambda$, and $\Ext_\Lambda^1\big(\mcC, \ctf(\tau C)\big) = 0$, it follows that $F \to C$ is right almost split in $\mcC$. Moreover, note that $\ctt( \tau C)$ is indecomposable up to summands in $\mcC \cap \mcT$ by the two equivalences to $\mcF$ that we already have established. Thus the middle horizontal sequence is the almost split sequence in $\mcC$ ending in $C$, up to possible superfluous summands in $\mcC \cap \mcT$. In particular $\tau_{\mcC} = \ctt \tau_{\Lambda}$ as functors to $\frac{\mcC}{\mcC \cap \mcT}$.  Since $\ctt\tfree = \id_\mcF$, we have established the commutativity of the triangle.
\end{proof}

\begin{cor}
    \label{cor.commutes_with_tau}
    The equivalences of \cref{thm.equ_tau_commute} commute with the Auslander--Reiten translation in the respective subcategories of $\mod \Lambda$.
\end{cor}

\begin{proof}
    Firstly, $\mcF$ has almost split sequences by \cite[Corollary~3.8]{MR617088}.  Since the diagram in \cref{thm.equ_tau_commute} commutes and $\tau_\mcC$ clearly commutes with itself, it is sufficient to show that $\tfree$ commutes with the Auslander--Reiten translation, that is, $\tfree\tau_\mcC = \tau_\mcF\tfree$.  We have already established that $\tau_\Lambda^- = \tau_\mcC^-\ctt$.  Fix an indecomposable $F\in\mcF$.
    By \cite[Proposition~3.3]{MR617088}, $\tau_\mcF^- F = 0$ if and only if $\tfree (\tau_{\Lambda}^- F) = 0$, that is, if and only if $\tfree\big(\tau_\mcC^-(\ctt F)\big) = 0$.  If it is not relative injective, $\tau_{\mcF}^- F$ is a direct summand of $\tfree\big(\tau_\mcC^-(\ctt F)\big) = \tfree (\tau_{\Lambda}^- F)$, by \cite[Corollary~3.5]{MR617088}.
    Thus, if $\tau_\mcF^- F \neq 0$, we have $\tau_\mcF^- F = \tfree\big(\tau_\mcC^-(\ctt F)\big)$, since the latter is indecomposable.  It follows that $\tau_\mcF^- = \tfree \tau_\mcC^- \ctt$.
\end{proof}

\section{Applications of cotorsion torsion triples to quiver representations}
\label{sec:applicationquivers}

Throughout this section, we consider the following special case, where $\mcA$ is an abelian category with enough projectives $\Proj\mcA\subseteq\mcA$, and $Q$ a finite acyclic quiver.  (For the dual results we instead assume that $\mcA$ has enough injectives $\Inj\mcA\subseteq\mcA$.)  We denote by $[Q, \mcA]$ the category of (covariant) representations of $Q$ in $\mcA$, or, equivalently, the functor category from the free category generated by $Q$ to $\mcA$.  This category inherits a lot of structure from $\mcA$; in particular it is abelian.  In this context, we can produce particular families of objects of $[Q,\mcA]$:

\begin{defn}
    For each vertex $i$ of $Q$ we have the evaluation functor
    \[
        -_i \colon [Q, \mcA] \to \mcA, \quad X \mapsto X_i
    \]
    sending each representation to its value on this vertex.  The functor $-_i$ is clearly exact, and moreover has an exact left adjoint
    \[
        \sfP_i \colon \mcA \to [Q, \mcA]
    \]
    where, for each vertex $j$ we have
    \[
        \sfP_i(A)_j = \bigoplus_{\rho \colon i \leadsto j} A^{(\rho)}.
    \]
Here the superscript indices are just used to distinguish the otherwise identical summands $A^{(\rho)}$, in order to refer to the individual summands by their index later.
The structure maps of $\sfP_i(A)$ are as follows: For any arrow $\alpha \colon j \to k$ in the quiver,
    \[
        \sfP_i(A)_j \into \sfP_i(A)_k
    \]
is uniquely determined by sending any summand $A^{(\rho)}$  corresponding to a path $\rho \colon i \leadsto j$ identically to the summand $A^{(\alpha \rho)}$.

Any arrow $\alpha \colon i \shortto j$ in $Q$ gives rise to a natural monomorphism of functors
\[ \sfP_{\alpha} \colon \sfP_j \into \sfP_i, \]
where $\sfP_{\alpha}(A)_k\colon \sfP_j(A)_k \into \sfP_i(A)_k$ is given by sending any summand $A^{(\rho)}$  corresponding to a path $\rho \colon j \leadsto k$ identically to $A^{(\rho \alpha)}$.

Dually to the construction of $\sfP_i$, the evaluation functor $-_i$ has an exact right adjoint denoted by $\sfI_i$, which is explicitly given on vertices by
    \[
        \sfI_i(A)_j = \bigoplus_{\rho \colon j \leadsto i} A^{(\rho)}.
    \]
\end{defn}

For a more detailed discussion of these functors, along with a proof that~$\sfP_i$ and~$\sfI_i$ are adjoint to $-_i$, see \cite[Section~3]{MR3990178}.

\begin{remark}
    \label{rmk:presproj}
    As $\sfI_i$ and $\sfP_i$ are adjoint to the exact functor $-_i$, it is a standard fact that $\sfI_i$ preserves injectives and $\sfP_i$ preserves projectives; that is, for $I$ injective and $P$ projective in $\mcA$, $\sfI_i(I)$ is injective and $\sfP_i(P)$ is projective in $[Q, \mcA]$.

    This can be seen by considering $\sfP_i(P)$, where $P$ is projective in $\mcA$, and letting $f\colon X \onto Y$ be an arbitrary epimorphism in $[Q,\mcA]$.  Via the adjunction, $ \Hom_{[Q,\mcA]}\big(\sfP_i(P), f\big)$ equals the composition
    \[
        \Hom_{[Q,\mcA]}\big(\sfP_i(P),X\big) \isom \Hom_\mcA(P,X_i) \stackrel{f_i\circ -}\onto \Hom_\mcA(P,Y_i) \isom \Hom_{[Q,\mcA]}\big(\sfP_i(P),Y\big),
    \]
    and is therefore an epimorphism.  Thus $\sfP_i(P)$ is a projective.  Dually, $\sfI_i(I)$ is injective.
\end{remark}

\begin{remark}
    \label{rmk:enoughproj}
    As soon as $\mcA$ has enough projectives, the same is true for $[Q,\mcA]$.  This is shown in \cite[Corollary~3.10]{MR3990178} for an arbitrary quiver $Q$ and $\mcA$ admitting small coproducts.  We give a short argument for our more restrictive setting.

    Let $X\in[Q,\mcA]$ be any object.  For any vertex $i$ in~$Q$, the identity on $X_i$ is adjunct to some morphism $\sfP_i(X_i) \to X$ in $[Q,\mcA]$.  This morphism is moreover the identity at the given vertex $i$, since the unit $\eta\colon \id_\mcA \to -_i\circ\sfP_i$ is the identity, as there is a unique path from $i$ to $i$.
Thus we have a canonical epimorphism
    \[
        p \colon \bigoplus_{i\in Q_0} \sfP_i(X_i) \onto X.
    \]

    For each $X_i$, choose an epimorphism from a projective $P^i \onto X_i$.  Since each $\sfP_i$ preserves projectives and is (right) exact, we get
    \[
        \bigoplus_{i\in Q_0} \sfP_i(P^i) \onto \bigoplus_{i\in Q_0} \sfP_i(X_i) \onto X,
    \]
    which is an epimorphism from a projective object onto $X$.  Thus $[Q,\mcA]$ has enough projectives.
\end{remark}

We use this notation in the proof of the following lemma, which shows that the ``standard'' projectives give rise to all projectives.

\begin{lemma} \label{lemma:allproj}
    The subcategory of projective objects in $[Q,\mcA]$ is
    \[
        \Proj\, [Q,\mcA] = \add \{ \sfP_i(P) \mid i \in Q_0, P \in \Proj \mcA \}.
    \]
\end{lemma}
\begin{proof}
    Each object $\sfP_i(P)$ is projective in $[Q,\mcA]$ by \cref{rmk:presproj}.  This establishes `$\supseteq$'.  Conversely, for a projective object $X$ in $[Q,\mcA]$, the epimorphism
    \[
        \bigoplus_{i\in Q_0} \sfP_i(P^i) \onto X
    \]
    of \cref{rmk:enoughproj} splits.  Thus $X\in\add \{ \sfP_i(P) \mid i \in Q_0, P \in \Proj A \}$, establishing `$\subseteq$'.
\end{proof}

\begin{lemma}
    \label{lemma:projres}
    Let $X\in[Q,\mcA]$ be any representation.  Then there is a short exact sequence
    \[
        \bigoplus_{\alpha\colon j \shortto k} \sfP_k(X_j) \stackrel m\into \bigoplus_{j\in Q_0} \sfP_j(X_j) \stackrel p\onto X,
    \]
    which is natural in $X$.

    In particular, if each component $X_i$ is projective, then $X$ has projective dimension at most $1$.
\end{lemma}
\begin{proof}
    The morphism $p$ is the canonical epimorphism of \cref{rmk:enoughproj}.  Given an arrow $\alpha\colon j\shortto k$ in the quiver, we define $m$ at the corresponding summand $\sfP_k(X_j)$ as the morphism
    \[
        \left(\begin{matrix} \sfP_{\alpha}(X_j) \\ -\sfP_k(X_\alpha) \end{matrix}\right)\colon \sfP_k(X_j) \to \sfP_j(X_j) \oplus \sfP_k(X_k).
    \]

    We show that this sequence is component-wise split exact, establishing the first claim.  The restriction of the two morphisms above to an arbitrary vertex $i$ is
    \begin{equation} \label{seq.funct_proj_degree}
        \bigoplus_{\substack{\alpha\colon j \shortto k \\ \rho\colon k \leadsto i}} X^{(\alpha,\rho)}_j
        \stackrel m\to
        \bigoplus_{\pi\colon j \leadsto i} X^{(\pi)}_j
        \stackrel p\to
        X_i.
    \end{equation}
    Here, again, the superscript indices are just used to distinguish the otherwise identical summands $X^{(\alpha,\rho)}_j = X_j$ and $X^{(\pi)}_{j} = X_{j}$.   At the summand $X^{(\alpha,\rho)}_j$
    corresponding to  an arrow $\alpha \colon j \shortto k$  and a path $\rho \colon k \leadsto i$, the morphism $m$ takes the form
    \[
        \left(\begin{matrix} \id_{X_j} \\ -X_\alpha \end{matrix}\right)\colon X_j^{(\alpha,\rho)} \to X_j^{(\rho\alpha)}\oplus X_k^{(\rho)}.
    \]
    Similarly, at the summand $X_j^{(\pi)}$, the morphism $p$ takes the form
    \[
        X_{\pi}\colon X_j^{(\pi)} \to X_i.
    \]
    In particular one immediately verifies that $p \circ m = 0$.

    Note that the only maps from objects of the form $X_j$ to $X_{j'}$ are induced by the structure maps of $X$, and in particular in the direction of the quiver $Q$ (recall that $Q$ has no oriented cycles).
    We can thus extend the partial order on the vertices of $Q$ to a total order $\preceq$, which induces a filtration of the complex \eqref{seq.funct_proj_degree} by restricting only to summands of the form $X_j$ with $j \succeq \ell$ for a given $\ell$.
    The consecutive steps in this filtration yield subquotient complexes involving only terms $X_j$ for a given $j$. These are of the form
    \[
        \bigoplus_{\substack{\alpha\colon j \shortto k \\ \rho\colon k \leadsto i}} X^{(\alpha,\rho)}_j
        \stackrel{m_j}{\to}
        \bigoplus_{\pi\colon j \leadsto i} X^{(\pi)}_j
        \stackrel{p_j}{\to}
        \begin{cases} X_i & i = j \\ 0 & i \neq j \end{cases}.
    \]
    Note that now at the summand $X^{(\alpha,\rho)}_j$, the morphism $m_j$ takes the form
    \[
        \id_{X_j} \colon X^{(\alpha,\rho)}_j \to X_j^{(\rho\alpha)},
    \]
    and for $j = i$ the morphism $p_j$ projects to the summand corresponding to the trivial path. In particular, all these subquotients are trivially split exact sequences, whence the complexes \eqref{seq.funct_proj_degree} also are split exact sequences.
\end{proof}

We recall the following useful trick for constructing adjoint functors:
\begin{lemma}
    \label{lem:adjointtrick}
    Let
    \[
        F \stackrel f\to G \onto H
    \]
    be an exact sequence of functors between abelian categories $\mcA$ and $\mcB$.  If both $F,G \colon \mcA \to \mcB$ admit right adjoints, then so does $H$.
\end{lemma}

\begin{proof}
    Letting $F',G' \colon \mcB \to \mcA$ be the adjoint functors to $F,G$,
    we obtain a commutative diagram
    \[
        \begin{tikzpicture}[scale=.75]
            \node (1) at (0,2) {$\Hom_{\mcB}\big(H(-),-\big)$};
            \node (2) at (5,2) {$\Hom_{\mcB}\big(G(-),-\big)$};
            \node (3) at (10,2) {$\Hom_{\mcB}\big(F(-),-\big)$};
            \node (4) at (0,0) {$\Hom_\mcA\big(-,(\ker f')(-)\big)$};
            \node (5) at (5,0) {$\Hom_\mcA\big(-,G'(-)\big)$};
            \node (6) at (10,0) {$\Hom_\mcA\big(-,F'(-)\big)$};
            \path[->,font=\scriptsize]
                (1) edge[right hook->] node[auto] {} (2)
                (2) edge node[auto] {$f^*$} (3)
                (4) edge[right hook->] node[auto] {} (5)
                (5) edge node[auto] {$f'_*$} (6)
                (1) edge[densely dashed] node[auto] {$\wr$} (4)
                (2) edge node[auto] {$\wr$} (5)
                (3) edge node[auto] {$\wr$} (6);
        \end{tikzpicture}
    \]
    where $f'_*$ is the unique morphism making the rightmost square commute, yielding a natural transformation $f'\colon G' \to F'$ by the Yoneda lemma.  Thus, by universality of the kernels in the left column, we obtain a natural isomorphism
    \[
        \Hom_{\mcB}\big(H(-),-\big) \isom \Hom_\mcA\big(-,(\ker f')(-)\big),
    \]
    making $H' := \ker f'$ a right adjoint to $H$.
\end{proof}

\begin{lemma}
    \label{lem:injrighadjoint}
    The functor $\sfI_i \colon \mcA \to [Q, \mcA]$ has a right adjoint $\sfRI_i \colon [Q, \mcA] \to \mcA$.
\end{lemma}

\begin{proof}
    Applying \cref{lemma:projres}, and recalling that $\sfI_i(A)_j = A^{|Q(j,i)|}$, where $|Q(j,i)|$ is the number of paths $j\leadsto i$ in $Q$, we have the following short exact sequence of functors for all $i$,
    \[
        \bigoplus_{\alpha\colon j\rightarrow k} \sfP_k^{|Q(j,i)|} \stackrel m\into \bigoplus_{j\in Q_0} \sfP_j^{|Q(j,i)|} \onto \sfI_i.
    \]
    The first two terms have right adjoints as they are direct sums of the functors $\sfP_j$.  Thus, by \cref{lem:adjointtrick}, $\sfRI_i$ exists and is a right adjoint to $\sfI_i$, for all $i$.
\end{proof}

With these lemmas at hand, we produce tilting subcategories of $[Q,\mcA]$.
\begin{prop}
    \label{prop:tiltingquiver}
    Let $\mcA$ be an abelian category with enough projectives and put
    \[
        \mbT = \add \{ \sfI_i(P) \mid i \in Q_0 \text{ and } P \in\Proj\mcA \}.
    \]
    Then $\mbT$ is tilting in $[Q, \mcA]$.
\end{prop}

\begin{proof}
    We need to check the definition of tilting; see \cref{def.tilting}. The first point follows from the fact that adjunctions of \emph{exact} functors of abelian categories extend to $\Ext$-groups.  In particular, the adjunction between $-_j$ and $\sfI_j$ extends to $\Ext^1$, namely
    \[
        \Ext_{[Q, \mcA]}^1\big(\sfI_i(P), \sfI_j(P')\big) \isom \Ext_{\mcA}^1\big(\sfI_i(P)_j, P'\big).
    \]
    This particular case is straightforward to verify.  Indeed, let $X\in[Q,\mcA]$ have projective resolution $\sfP^\bullet \to X$.  Then $\sfP^\bullet_j \to X_j$ is a projective resolution in $\mcA$, as $-_j$ both is exact and preserves projectives.  The adjunction gives an isomorphism of complexes
    \[
        \Hom_{[Q,\mcA]}\big(\sfP^\bullet, \sfI_j(A)\big) \isom \Hom_\mcA(\sfP^\bullet_j, A),
    \]
    which compute $\Ext^1_{[Q,\mcA]}\big(X, \sfI_j(A)\big)$ and $\Ext^1_\mcA(X_j,A)$, respectively.

    Now $\sfI_i(P)_j$ is a (possibly empty) sum of copies of $P$, hence projective.  It follows that this $\Ext_{[Q, \mcA]}^1\big(\sfI_i(P), \sfI_j(P')\big)$ is zero.

    For the second point, we need to check that every object in $\mbT$ has projective dimension at most $1$.  This follows from \cref{lemma:projres}, since every component of $\sfI_i(P)$ is projective in $\mcA$.

    The third point follows in a similar way: By \cref{lemma:allproj} and \cref{rem.cores_on_sums_enough} it suffices to consider projectives of the form $\sfP_i(P)$. By the dual of \cref{lemma:projres}, applied to the particular representation $\sfP_i(P)$, there is a short exact sequence
    \[
        \sfP_i(P) \into  \bigoplus_{j \in Q_0} \sfI_j\big(\sfP_i(P)_j\big) \onto \bigoplus_{\alpha\colon k \shortto j} \sfI_k\big(\sfP_i(P)_j\big),
    \]
    whose two right terms are in $\mbT$.

    In order to conclude that $\mbT$ is tilting, we need to show that $\mbT$ is contravariantly finite in $[Q,\mcA]$.
    Let $X\in[Q,\mcA]$ be any representation.  Let $\sfRI_i$ be the right adjdoint of $\sfI_i$, as given in \cref{lem:injrighadjoint}.  For each $i$, choose an epimorphism $p^i: P^i \onto \sfRI_i(X)$ from a projective object
    $P^i \in \Proj\mcA$.
    We claim that the sum of the adjoints
    \[
        \bigoplus_{i\in Q_0} \sfI_i(P^i) \stackrel\varphi\to X
    \]
    is a right $\mbT$-approximation, i.e., any morphism $T \to X$ with $T\in \mbT$ factors through~$\varphi$.  It is sufficient to check the claim on the generators of $\mbT$, namely the objects~$\sfI_i(P)$, where $P$ is projective and $i$ is a vertex.  To this end, choose any morphism $f:\sfI_i(P) \to X$.  It is adjunct to a morphism $P \to \sfRI_i(X)$, which factors through $p^i$ by projectivity of $P^i$.  Thus $f$ factors through the $i$th component of $\varphi$, $\sfI_i(P^i) \to X$, and therefore through~$\varphi$ itself.
\end{proof}

\begin{cor}
    \label{cor.ctt_from_tilting}
    Let $\mbT$ be as in \cref{prop:tiltingquiver}.
    There is a cotorsion torsion triple $(\mcC, \mcT, \mcF)$ in $[Q, \mcA]$ given by
    \begin{align*}\SwapAboveDisplaySkip
        \mcC & = [Q,\Proj\mcA] \\
        \mcT & = \Fac \mbT \\
        \mcF & = \mbT^{\perp}.
    \end{align*}
    Moreover, we have $\mcT \cap \mcC = \mbT$.
\end{cor}

\begin{proof}
    Except for the description of $\mcC$, this is one direction of \cref{thm.ctt=tilting}. Thus the only thing to show is that
    \[
        {}^{\perp_1}\Fac\mbT = \{ X \in {}^{\perp_1}\mbT \mid \pdim X\leq 1 \}
    \]
    coincides with $[Q,\Proj\mcA]$.

    Note first that since $\sfI_i$ is exact and $\mcA$ has enough projectives, $\sfI_i(A)$ is in $\Fac \mbT$, for all $A\in\mcA$.  Thus, for $X \in {}^{\perp_1} \Fac \mbT$ one has
    \[
        \Ext_{\mcA}^1(X_i, A) \isom \Ext_{[Q, \mcA]}^1\big(X, \sfI_i(A)\big) = 0.
    \]
    It follows that $X_i$ is projective.

    Going the other way, let $X \in [Q, \mcA]$ such that all $X_i$ are projective.  By \cref{lemma:projres}, $\pdim X \leq 1$, so we need only show that $X \in {}^{\perp_1}\mbT$.  This follows from
    \[
        \Ext^1_{[Q,\mcA]}\big(X,\sfI_i(P)\big) \isom \Ext^1_\mcA(X_i,P) = 0. \qedhere
    \]
\end{proof}

\begin{thm}
    \label{thm.QA}
    Let $\mcA$ be abelian with enough projectives and $Q$ a finite acyclic quiver. With $\mbT$ as in \cref{prop:tiltingquiver}, we have
    \[
        \frac{[Q,\Proj\mcA]}{\mbT} \simeq \mbT^\perp
    \]
    induced by $\tfree$ and $\ctt$ coming from the (co)torsion pairs.
\end{thm}

\begin{proof}
    This is now just a combination of \cref{thm.ctt,cor.ctt_from_tilting}.
\end{proof}

We explicitly state the dual version of this theorem, which additionally summarizes the dual of previous results.
\begin{thm}
    \label{thm.QAop}
    Let $\mcA$ be an abelian category with enough injective objects, $Q$ a finite acyclic quiver, and put
    \[
        \mbC = \add \{ \sfP_i(I) \mid i \in Q_0 \text{ and } I \in\Inj\mcA \}.
    \]
    Then $(\mcT,\mcF)$ is a torsion pair and $(\mcF,\mcD)$ is cotorsion pair in $[Q,\mcA]$, where
    \begin{align*}
        \mcT &= {}^\perp\mbC \\
        \mcF &= \Sub\mbC \\
        \mcD &= [Q,\Inj\mcA] \\
        \mcF \cap \mcD &= \mbC,
    \end{align*}
    and there is an equivalence
    \[
        \frac{[Q,\Inj\mcA]}{\mbC} \simeq {}^\perp\mbC
    \]
    induced by $\tors\colon \mcA \to \mcT$ and $\widetilde\ctf\colon \mcA \to \frac{\mcD}{\mcF\cap\mcD}$.
\end{thm}

\subsection{Representations of $A_n$ quivers}
\label{sec:quiverspecialcases}

In general, the right hand sides of the equivalences are fairly non-explicit.  We therefore further interpret these theorems in the case where $Q$ is a Dynkin quiver.  We start with the easiest case, which is also the most relevant to our applications, namely the linearly oriented Dynkin quiver of type $A$:
\[
    Q=\vec A_n\colon 1\to 2\to \cdots \to n.
\]

\begin{lemma}
    \label{lemma:TFAn}
    Let $\mcA$ be abelian with enough projectives and put
    \[
        \mbT = \add \{ \sfI_i(P) \mid 1\leq i\leq n \text{ and } P\in\Proj\mcA \} \subseteq [\vec A_n, \mcA].
    \]
    Then
    \begin{enumerate}
        \item $X \in \mbT^\perp$ if and only if $X_1 = 0$; and
        \item $X \in \Fac\mbT$ if and only if all structure maps of $X$ are epimorphisms.
    \end{enumerate}
\end{lemma}
\begin{proof}
    Recall first that for this quiver, $\sfI_i(P)$ is just
    \[
        \underbrace{P \stackrel\id\to P \stackrel\id\to \cdots \stackrel\id\to P}_\text{$i$ copies} \to 0 \to \cdots \to 0.
    \]

    For (1) we note that $X \in \mbT^\perp$ if and only if $\Hom_{[\vec A_n,\mcA]}\big(\sfI_i(P),X\big) = 0$ for all $1\leq i \leq n$ and $P \in \Proj\mcA$.  Any morphism $P \to X_1$ lifts uniquely to $\sfI_n(P) \to X$, so if $X\in\mbT^\perp$, then $P \to X_1$ is zero.  Since $\mcA$ has enough projectives, $X_1 = 0$.  On the other hand, if $X_1 = 0$, then $\Hom_{[Q,\mcA]}\big(\sfI_i(P),X\big) = 0$ for all $1\leq i\leq n$ and $P\in\Proj\mcA$.

    For (2), let first $X\in\Fac\mbT$, i.e., there is an epimorphism $T \onto X$ with $T\in\mbT$.  Since the structure maps for each $\sfI_i(P)$ is an epimorphism, the same is true for $T$.  Because $T\onto X$, each structure map of $X$ is an epimorphism.  Conversely, assume that the structure maps of $X$ are all epimorphisms and choose an epimorphism $P \onto X_1$.  It extends to a unique epimorphism $\sfI_n(P) \onto X$, that is to say, $X\in\Fac\mbT$.
\end{proof}

Noting that in this case $\mbT^\perp \simeq [\vec A_{n-1},\mcA]$, we arrive at the following corollary for \cref{thm.QA}.
\begin{cor}
    \label{cor.An}
    Let $\mcA$ be abelian with enough projectives.  Then
    \[
        \frac{[\vec A_n, \Proj\mcA]}{ \mbT } \simeq [\vec A_{n-1}, \mcA],
    \]
    where $\mbT$ is the additive subcategory generated by representations of the form $\sfI_i(P)$ with $P \in \Proj\mcA$.
    In particular, $\mbT$ is the subcategory of epimorphic representations with projectives at every vertex.
\end{cor}
\begin{proof}
    It only remains to establish the last claim.  By \cref{cor.ctt_from_tilting} we have
    \[
        \mbT = \Fac\mbT \cap [Q,\Proj\mcA].
    \]
    The second point of \cref{lemma:TFAn} completes the proof.
\end{proof}

Dually, from \cref{thm.QAop} one immediately obtains the following result.
\begin{cor}
    \label{cor.Andual}
    Let $\mcA$ be abelian with enough injectives.  Then
    \[
        \frac{[\vec A_n, \Inj\mcA]}{ \mbC } \simeq [\vec A_{n-1}, \mcA],
    \]
    where $\mbC$ is the additive subcategory generated by representations of the form $\sfP_i(I)$ with $I \in \Inj\mcA$.
    In particular, $\mbC$ is the subcategory of monomorphic representations with injectives at every vertex.
\end{cor}

\begin{remark}
    For $n=2$, these results just state that passing between modules and their projective (or injective) presentations preserves most information.
\end{remark}

We now restrict to the case where $\mcA = \rep_\field \vec A_m$ to recover \cref{thm:concrete} of the introduction.
\begin{cor}
    \label{cor:epireps}
    The functor $\tors$ induces an equivalence
    \[
         \frac{\rep_{\field}^{\e, *}(\vec A_m \otimes \vec A_n)}{\rep_{\field}^{\e, \m}(\vec A_m \otimes \vec A_n)} \to \rep_{\field}(\vec A_m \otimes \vec A_{n-1}).
    \]
    Moreover, $\rep_{\field}^{\e, \m}(\vec A_m \otimes \vec A_n)$ consists precisely of all direct sums of thin modules of the form $\field_{\{1, \ldots, i \} \times \{ j, \ldots, m \}}$.
\end{cor}
\begin{proof}
    Letting $\mcA = \rep_\field \vec A_m$ in \cref{cor.Andual}, we have
    \[
        [\vec A_n, \mcA] \simeq \rep_\field (\vec A_m \tensor \vec A_n).
    \]
    Since the injective representations of $\vec A_m$ are precisely the representations where every structure map is an epimorphism, we we have
    \[
        [\vec A_n, \Inj\mcA] \simeq \rep_\field^{\e,*} (\vec A_m \tensor \vec A_n).
    \]
    Moreover, $\mbC \subseteq \rep_\field (\vec A_m \tensor \vec A_n)$ is the subcategory additively generated by the representations
    \[
        \sfP_j(I_i) = I_i \tensor P_j = \field_{\{1,\ldots,i\} \times \{j,\ldots,m\}},
    \]
    and $\mbC = \rep_\field^{\e,\m} (\vec A_m \tensor \vec A_n)$, since it is the subcategory of representations with monomorphisms in the $\vec A_n$-direction between injective $\vec A_m$-representations.
\end{proof}

\begin{construction}
    \label{const.explicit_An}
    The equivalence of \cref{thm.QA} has only been given rather abstractly, but in the special case of \cref{cor.An} we can explicitly describe the functors $\tfree\colon [\vec A_n,\mcA] \to \mcF$ and $\ctt\colon [\vec A_n,\mcA] \to \mcC/\mbT$ giving rise to the equivalence.

    Recall that the cotorsion torsion triple is given by
    \begin{align*}
        \mcC &= {}^{\perp_1}\Fac\mbT = [\vec A_n, \Proj\mcA] \\
        \mcT &= \Fac\mbT = \{ X \in [\vec A_n,\mcA] \mid \text{all structure maps of $X$ are epimorphisms} \} \\
        \mcF &= \mbT^\perp = \{ X \in [\vec A_n,\mcA] \mid X_1 = 0 \} \simeq [\vec A_{n-1},\mcA],
    \end{align*}
    by \cref{cor.ctt_from_tilting,lemma:TFAn}.

    The functor $\tfree\colon [\vec A_n, \mcA] \to \mcF$ is explicitly given by sending
    \[
        X := [ X_1 \stackrel{f_1}\to X_2 \stackrel{f_2}\to \cdots \stackrel{f_{n-1}}\to X_n ]
    \]
    to
    \[
        \tfree X := \bigg[ 0 \to \coker f_1 \to \coker f_2f_1 \to \cdots \to \coker f_{n-1}\cdots f_1 \bigg];
    \]
    the reason being that the canonical epimorphism $X \onto \tfree X$ has kernel
    \[
        \tors X := [ X_1 \onto \im f_1 \onto \im f_2f_1 \onto \cdots \onto \im f_{n-1}\cdots f_2f_1 ]
    \]
    which lies in $\mcT$, as every structure map is an epimorphism.

    Conversely, the construction from $[\vec A_n, \mcA]$ to $\mcC$, which gives rise to the functor $[\vec A_{n-1}, \mcA] \to \frac{ [\vec A_n, \Proj\mcA] }{ \mbT }$, by \cref{lemma:cotors_functors_cd}, is given by constructing a short exact sequence
    \[
        \ctf X \into \ctt X \onto X,
    \]
    where $\ctf X \in \mcT$ and $\ctt X \in \mcC$.  This is done in the following manner:

    Starting with a representation $X \in [\vec A_n,\mcA]$, one constructs the following diagram from right to left.
    \[
        \begin{tikzpicture}[scale=.75]
            \node (A6) at (13,0) {$X_n$};
            \node (A5) at (10,0) {$X_{n-1}$};
            \node (A4) at (7,0) {$X_{n-2}$};
            \node (A3) at (5,0) {$\cdots$};
            \node (A2) at (3,0) {$X_2$};
            \node (A1) at (0,0) {$X_1$};
            \node (P1) at (-.75, -2) {$P_1$};
            \node (W1) at (-1.5,-4) {$\Omega X_1$};
            \node (P2) at (2.25, -2) {$P_2$};
            \node (PB1) at (.75,-2) {$R_2$};
            \node (W2) at (1.5,-4) {$\Omega X_2$};
            \node (P4) at (6.25, -2) {$P_{n-2}$};
            \node (W4) at (5.5,-4) {$\Omega X_{n-2}$};
            \node (PB4) at (7.75,-2) {$R_3$};
            \node (P5) at (9.25, -2) {$P_{n-1}$};
            \node (W5) at (8.5,-4) {$\Omega X_{n-1}$};
            \node (PB5) at (10.75,-2) {$R_{n}$};
            \node (P6) at (12.25, -2) {$P_n$};
            \node (W6) at (11.5,-4) {$\Omega X_n$};
            \node (PB2) at (3.5,-2) {};
            \node (PB3) at (5,-2) {};
            \node (D3) at (4.25,-2) {$\cdots$};
            \node (W3) at (3.5,-4) {$\cdots$};
            \tikzpbadhoc{PB5}{A6}
            \tikzpbadhoc{PB4}{A5}
            \tikzpbadhoc{PB1}{A2}
            \path[->,font=\scriptsize]
                (A1) edge (A2)
                (A2) edge (A3)
                (A3) edge (A4)
                (A4) edge (A5)
                (A5) edge (A6)
                (P1) edge [->>] (A1)
                (PB1) edge [->>] (A1)
                (W1) edge [right hook->] (P1)
                (P2) edge [->>] (A2)
                (W2) edge [right hook->] (PB1)
                (W2) edge [right hook->] (P2)
                (P4) edge [->>] (A4)
                (PB4) edge [->>] (A4)
                (W4) edge [right hook->] (P4)
                (P5) edge [->>] (A5)
                (PB5) edge [->>] (A5)
                (W5) edge [right hook->] (PB4)
                (W5) edge [right hook->] (P5)
                (P6) edge [->>] (A6)
                (W6) edge [right hook->] (PB5)
                (W6) edge [right hook->] (P6)
                (P1) edge [->>] (PB1)
                (PB1) edge (P2)
                (P4) edge [->>] (PB4)
                (PB4) edge (P5)
                (P5) edge [->>] (PB5)
                (PB5) edge (P6)
                (W1) edge [->>] (W2)
                (W4) edge [->>] (W5)
                (W5) edge [->>] (W6)
                (P2) edge [->>] (PB2)
                (PB3) edge (P4)
                (W2) edge[->>] (W3)
                (W3)edge[->>] (W4);
        \end{tikzpicture}
    \]
    That is, start with an epimorphism from a projective $P_n$ to $X_n$ and take the pullback along $X_{n-1} \to X_n$, denoted by $R_n$.  These two epimorphisms have the same kernel, which we denote by $\Omega X_n$.  Next, take an epimorphism from a projective $P_{n-1}$ to the pullback, which also gives an epimorphism $P_{n-1} \onto X_{n-1}$.  Note that the map of kernels $\Omega X_{n-1} \onto \Omega X_n$ is an epimorphism, as can be seen by an easy diagram chase.
    Take the pullback along $X_{n-2} \to X_{n-1}$ and iterate this procedure down to $X_1$.

    Denoting the middle row of projective objects by $\ctt X$ and the lower row by $\ctf X$, we have constructed the desired short exact sequence, as $\ctt X \in \mcC$ and $\ctf X \in \mcT$.  Thus the image of $X$ under $\ctt\colon [\vec A_n,\mcA] \to \frac{[\vec A_n,\mcA]}{\mbT}$ is given by $[P_1 \to P_2 \to \cdots \to P_n]$.

    The constructions of $\tfree$ and $\ctt$ can be formally dualized to produce the equivalence of \cref{cor.Andual}.  In particular, we recover \cref{const:intro}.
\end{construction}

\begin{ex}
    We take $\mcA = \Rep_\field^\mathrm{fp} \mbR_{\geq 0}$, and denote by $\proj_\field \mbR_{\geq0} = \add \{ \field_{\hopen x,\infty} \}$ its subcategory of projectives. Equivalently, $\proj_\field \mbR_{\geq0}$ is the subcategory given by requesting the property that all structure maps are monomorphic.

    We are interested in determining all indecomposables in $[\vec A_2, \proj_\field \mbR_{\geq0}]$.  By \cref{cor.An}, we have
    \[
        \frac{ [\vec A_2, \proj_\field \mbR_{\geq0}] }{\mbT} \simeq \mcA.
    \]
    Thus the indecomposables in $[\vec A_2, \proj_\field \mbR_{\geq0}]$ are those in $\mbT$, along the ones coming from the indecomposables in $\mcA$.  In this setup we know both collections:
    \begin{align*}
        \mbT &= \add \big(\{ \field_{\hopen x,\infty} \to \field_{\hopen x,\infty} \} \cup \{ \field_{\hopen x,\infty} \to 0 \}\big), \quad\text{and} \\
        \mcA &= \add \{ \field_{\hopen x,y} \}, \quad\text{where } 0\leq x < y\leq\infty.
    \end{align*}
    Under the construction above, the representations of $\mbT$ are sent to $0$ and the indecomposable $\field_{\hopen x,y}$ in $\mcA$ corresponds to
    \[
        [ \field_{\hopen y,\infty} \to \field_{\hopen x,\infty} ] \in [\vec A_2, \proj_\field \mbR_{\geq0}].
    \]

    Thus our complete list of indecomposables in $[\vec A_2, \proj_\field \mbR_{\geq0}]$ is
    \begin{align*}
        [\field_{\hopen x,\infty} & \to \field_{\hopen x,\infty}], &&\text{for } 0 \leq x < \infty, \\
        [\field_{\hopen y,\infty} & \to 0], &&\text{for } 0 \leq y < \infty, \text{ and} \\
        [\field_{\hopen y,\infty} & \to \field_{\hopen x,\infty}], &&\text{for } 0 \leq x < y \leq \infty.
    \end{align*}
    Respectively, we may picture these as in \cref{fig:indecsA2R}.

    \begin{figure}
        \centering
        \begin{subfigure}[t]{.25\textwidth}
            \begin{tikzpicture}[scale=.55,baseline=0]
                \node (1) at (1,2) {$1$};
                \node (2) at (1,0) {$2$};
                \fill[color=lightergray, path fading=east] (2,2) rectangle (6,0);
                \draw[->] (1) edge (2);
                \draw[thick] (5,0) -- (2,0) -- (2,2) -- (5,2);
                \draw[densely dashed, thick,->] (5,0) -- (6,0);
                \draw[densely dashed, thick,->] (5,2) -- (6,2);
                \draw (2,2pt) -- (2,-2pt) node[anchor=north] {$x$};
            \end{tikzpicture}
        \end{subfigure}
        \hspace{.01\textwidth}
        \begin{subfigure}[t]{.20\textwidth}
            \begin{tikzpicture}[scale=.55,baseline=0]
                \node (1) at (2,2) {$\phantom1$};
                \draw[thick] (2,2) -- (5,2);
                \draw[densely dashed, thick,->] (5,2) -- (6,2);
                \draw (2,2pt) -- (2,-2pt) node[anchor=north] {$y$};
            \end{tikzpicture}
        \end{subfigure}
        \hspace{.01\textwidth}
        \begin{subfigure}[t]{.25\textwidth}
            \begin{tikzpicture}[scale=.55,baseline=0]
                \node (1) at (2,2) {$\phantom1$};
                \fill[color=lightergray, path fading=east] (3,2) rectangle (6,0);
                \draw[thick] (1,0) -- (5,0);
                \draw[thick] (3,0) -- (3,2) -- (5,2);
                \draw[densely dashed, thick,->] (5,0) -- (6,0);
                \draw[densely dashed, thick,->] (5,2) -- (6,2);
                \draw (1,2pt) -- (1,-2pt) node[anchor=north] {$x$};
                \draw (3,2pt) -- (3,-2pt) node[anchor=north] {$y$};
            \end{tikzpicture}
        \end{subfigure}
        \hspace{.01\textwidth}
        \begin{subfigure}[t]{.20\textwidth}
            \begin{tikzpicture}[scale=.55,baseline=0]
                \node (1) at (2,2) {$\phantom1$};
                \draw[thick] (2,0) -- (5,0);
                \draw[densely dashed, thick,->] (5,0) -- (6,0);
                \draw (2,2pt) -- (2,-2pt) node[anchor=north] {$x$};
            \end{tikzpicture}
        \end{subfigure}
        \caption{The four types of indecomposable representations of $[\vec A_2,\proj_\field \mbR_{\geq0}]$.  From left to right: $[\field_{\hopen x,\infty} {\rightarrow} \field_{\hopen x,\infty}]$, $[\field_{\hopen y,\infty} {\rightarrow} 0]$, $[\field_{\hopen y,\infty} {\rightarrow} \field_{\hopen x,\infty}]$, and $[0 {\rightarrow} \field_{\hopen x,\infty}]$.  The first two objects are contained in $\mbT$, and thus sent to $0$ in $\Rep_\field^\mathrm{fp} \mbR_{\geq0}$, whereas the latter two are sent to $\field_{\hopen x,y}$ and $\field_{\hopen x,\infty}$.}
        \label{fig:indecsA2R}
    \end{figure}

    Dually, and more relevant to clustering, we may study the subcategory $\inj \mbR_{\geq0}$ of $\Rep_\field^\mathrm{fp} \mbR_{\geq 0}$ where all structure maps are surjections.
    These are the injectives in $\Rep_\field^\mathrm{fp} \mbR_{\geq 0}$.  We obtain a complete list of indecomposables in $[\vec A_2, \inj \mbR_{\geq0}]$ as
    \begin{align*}
        \field_{\hopen 0,x} & \to \field_{\hopen 0,x} && \text{with } 0 < x \leq \infty \\
        0 & \to \field_{\hopen 0,x} &&  \text{with } 0 < x \leq \infty \\
        \field_{\hopen 0,y} & \to \field_{\hopen 0,x} && 0 \leq x < y \leq \infty.
    \end{align*}
\end{ex}

\begin{ex}
    Now we consider filtered hierarchical clustering for $3$-step filtration.  That is, we would like to find all indecomposables in $[\vec A_3, \inj_\field \mbR_{\geq0}]$.  This turns out to be a lot harder.  Let us illustrate this by restricting to $5$ fixed critical values, that is, we replace $\rep_\field^\mathrm{fp} \mbR_{\geq0}$ with
$\rep_\field A_5$.  By \cref{cor.Andual} we have
    \[
        \frac{ [\vec A_3, \inj_\field \vec A_5] }{\mbC} \simeq [\vec A_2, \rep_\field \vec A_5],
    \]
    where $\mbC$ is generated by the representations $\sfP_i(I)$.  On the right hand side, one easily finds a one-parameter family of indecomposables. (The one below comes from embedding an infinite family for $\widetilde D_4$.)
    \[
        \begin{tikzpicture}[scale=.75]
            \node (11) at (0,0) {$0$};
            \node (12) at (2,0) {$\field$};
            \node (13) at (4,0) {$\field^2$};
            \node (14) at (6,0) {$\field^2$};
            \node (15) at (8,0) {$\field$};
            \node (21) at (0,-2) {$\field$};
            \node (22) at (2,-2) {$\field^2$};
            \node (23) at (4,-2) {$\field^2$};
            \node (24) at (6,-2) {$\field$};
            \node (25) at (8,-2) {$0$};
            \path[->,font=\scriptsize]
                (11) edge (12)
                (12) edge node [above] {$\left( \begin{smallmatrix} 1 \\ 0 \end{smallmatrix} \right)$} (13)
                (13) edge [-,double distance=.5mm] (14)
                (14) edge node [above] {$\left( \begin{smallmatrix} 1 & 1 \end{smallmatrix} \right)$} (15)
                (11) edge (21)
                (12) edge node [left] {$\left( \begin{smallmatrix} 1 \\ 0 \end{smallmatrix} \right)$} (22)
                (13) edge [-,double distance=.5mm] (23)
                (14) edge node [left] {$\left( \begin{smallmatrix} 1 & 0 \end{smallmatrix} \right)$} (24)
                (15) edge (25)
                (21) edge node [above] {$\left( \begin{smallmatrix} \alpha \\ 1 \end{smallmatrix} \right)$} (22)
                (22) edge [-,double distance=.5mm] (23)
                (23) edge node [above] {$\left( \begin{smallmatrix} 1 & 0 \end{smallmatrix} \right)$} (24)
                (24) edge (25);
        \end{tikzpicture}
    \]
    Abstractly, it follows immediately from the equivalence that also $[\vec A_3, \inj \field \vec A_5]$ has a one-parameter family of indecomposables.

    Concretely, one may follow the (dual) instructions in \cref{const.explicit_An} above and obtain the following $1$-parameter family of indecomposables in $[\vec A_3, \inj \field \vec A_5]$:
    \[
        \begin{tikzpicture}[scale=.75]
            \node (11) at (0,0) {$\field^2$};
            \node (12) at (2,0) {$\field^2$};
            \node (13) at (4,0) {$\field^2$};
            \node (14) at (6,0) {$\field^2$};
            \node (15) at (8,0) {$\field$};
            \node (21) at (0,-2) {$\field^3$};
            \node (22) at (2,-2) {$\field^3$};
            \node (23) at (4,-2) {$\field^2$};
            \node (24) at (6,-2) {$\field$};
            \node (25) at (8,-2) {$0$};
            \node (31) at (0,-4) {$\field^2$};
            \node (32) at (2,-4) {$\field$};
            \node (33) at (4,-4) {$0$};
            \node (34) at (6,-4) {$0$};
            \node (35) at (8,-4) {$0$};
            \path[->,font=\scriptsize]
                (11) edge [-,double distance=.5mm] (12)
                (12) edge [-,double distance=.5mm] (13)
                (13) edge [-,double distance=.5mm] (14)
                (14) edge node [above] {$\left( \begin{smallmatrix} 1 & 1 \end{smallmatrix} \right)$} (15)
                (11) edge node [left] {$\left( \begin{smallmatrix} 1 & 0 \\ 0 & 1 \\ 0 & 0 \end{smallmatrix} \right)$} (21)
                (12) edge node [left] {$\left( \begin{smallmatrix} 1 & 0 \\ 0 & 1 \\ 0 & 0  \end{smallmatrix} \right)$} (22)
                (13) edge [-,double distance=.5mm] (23)
                (14) edge node [right] {$\left( \begin{smallmatrix} 1 & 0 \end{smallmatrix} \right)$} (24)
                (15) edge (25)
                (21) edge node [below] {$\left( \begin{smallmatrix} 1 & 0 & \alpha \\ 0 & 1 & 0 \\ 0 & 0 & 1 \end{smallmatrix} \right)$} (22)
                (22) edge node [above] {$\left( \begin{smallmatrix} 1 & 0 & 0 \\ 0 & 1 & 1 \end{smallmatrix} \right)$} (23)
                (23) edge node [above] {$\left( \begin{smallmatrix} 1 & 0 \end{smallmatrix} \right)$} (24)
                (24) edge (25)
                (21) edge node [left] {$\left( \begin{smallmatrix} 1 & 0 & 0 \\ 0 & 1 & 0 \end{smallmatrix} \right)$} (31)
                (22) edge node [right] {$\left( \begin{smallmatrix} 0 & 1 & 0 \end{smallmatrix} \right)$} (32)
                (23) edge (33)
                (24) edge (34)
                (25) edge (35)
                (31) edge node [above] {$\left( \begin{smallmatrix} 0 & 1 \end{smallmatrix} \right)$} (32)
                (32) edge (33)
                (33) edge (34)
                (34) edge (35);
        \end{tikzpicture}
    \]
\end{ex}

\subsection{Representations in a module category of a finite dimensional algebra}

As a next special case of \cref{thm.QA}, we consider the case that $Q$ is an arbitrary finite acyclic quiver arbitrary while the abelian category $\mcA$ is the category of finite dimensional modules over some finite dimensional $\field$-algebra $\Lambda$. In particular in this case we have $[Q, \mcA] = \mod \Lambda Q$, and we will present the following results in this more familiar language.

Before we can rephrase \cref{thm.QA} in this setup, note that there are natural forgetful functors $\mod \Lambda Q \to \mod \Lambda$ and $\mod \Lambda Q \to \mod \field Q$. The image of a module $M$ under these functors will be denoted by $M_{\Lambda}$ and $M_{\field Q}$, respectively.

With this notation, we obtain the following version of \cref{thm.QA}.

\begin{cor}
    \label{cor.no_inj_summands}
    Let $\Lambda$ be a finite dimensional $\field$-algebra, and $Q$ a finite acyclic quiver.  Then
    \[
        \frac{ \{ M \in \mod \Lambda Q \mid M_{\Lambda} \in \proj \Lambda \} }{\mathbb{T}} \simeq \{ M \in \mod \Lambda Q \mid M_{\field Q} \in \modni \field Q \} ,
    \]
    where
    \[
        \mathbb{T} = \add \{ \sfI_i(P) \mid i \in Q_0 \text{ and } P \in \proj \Lambda \}
    \]
    and $\modni \field Q$ denotes the subcategory of $\mod \field Q$ containing all modules without non-zero injective summands.
\end{cor}

\begin{proof}
    We apply \cref{thm.QA}. The left hand side of the equivalence is literally the same as in that theorem, so it only remains to simplify the right hand side. Here we have
    \begin{align*}
        \mbT^{\perp} & =  \big\{ M \in \mod \Lambda Q \mid \Hom_{\mod \Lambda Q}\big( \sfI_i(\Lambda), M\big) = 0 \text{ for all } i\in Q_0 \big\} \\
            & = \big\{ M \in \mod \Lambda Q \mid \Hom_{\mod \field Q} \big( \sfI_i(\field), M \big) = 0 \text{ for all } i\in Q_0 \big\} \\
            & = \{ M \in \mod \Lambda Q \mid M_{\field Q} \in \modni \field Q \}.
    \end{align*}
    Here, the first equality follows from the fact that $\sfI(\Lambda) = \Lambda \otimes \sfI(\field)$, together with the fact that tensoring with $\Lambda$ is left adjoint to the forgetful functor $\mod \Lambda Q \to \mod \field Q$. The second equality holds since for $\mod \field Q$, only injective modules permit non-zero homomorphisms from the indecomposable injectives $\sfI_i(\field)$.
\end{proof}

\subsection{Categories of monomorphisms of quiver representations}

In the next application, we combine \cref{cor.An,cor.no_inj_summands} to study monomorphisms of representation of a quiver $Q$, as well as finite compositions thereof. To set the stage, recall that we denote by
$\rep_\field^{*, \m}(Q \otimes A_n)$
the subcategory of
$\rep_\field(Q \otimes A_n)$
where all structure maps in the $A_n$-direction are monomorphisms.

\begin{defn}
    Let $Q$ be a Dynkin quiver. The \emph{costable Auslander algebra} of $\field Q$ is given as
    \[
        \Gamma = \End_{\field Q} \left( \bigoplus_M  M \right),
    \]
    where the sum runs over all isomorphism classes of non-injective indecomposables. (Note that this is indeed a finite sum, since we assumed $Q$ to be Dynkin.)
\end{defn}

\begin{cor}
    \label{cor.auslander_alg}
    Let $Q$ be a Dynkin quiver, and $\Gamma$ its costable Auslander algebra. Then
    \[
        \frac{  \rep_\field^{*, \m}  (Q \otimes \vec A_n) }{\mbX} \simeq \mod \Gamma \vec A_{n-1},
    \]
    where $\mbX$ be the subcategory whose indecomposables are the representations
    \begin{itemize}
        \item $X = \cdots = X \into I = \cdots = I$,
            \par for $X \into I$ an injective envelope of an indecomposable $\field Q$-module $X$,
        \item $0 = \cdots = 0 \to I = \cdots = I$,
            \par where $I$ is an indecomposable injective $\field Q$-module.
    \end{itemize}

    \medskip
    In particular, for $n = 2$ above, we have
    \[
        \frac{  \rep_\field^{*, \m}  (Q \otimes \vec A_2) }{\mbX} \simeq \mod \Gamma.
    \]
\end{cor}

\begin{proof}
    Since $\field \vec A_n$-modules are projective if and only if all structure maps are monomorphisms, we may reformulate the numerator of the left hand side as
    \[
        \rep_\field^{*, \m}  (Q \otimes \vec A_n) = \{ M \in \mod \field Q \otimes \field \vec A_n \mid M_{\field \vec A_n} \in \proj \field \vec A_n \}.
    \]
    Thus, by \cref{cor.no_inj_summands}, we have
    \[
        \frac{\rep_\field^{*, \m}  (Q \otimes \vec A_n)}{\mbT} \simeq \{ M \in \mod \field Q \otimes \field \vec A_n \mid M_{\field Q} \in \modni \field Q \},
    \]
    where $\mbT$ consists precisely of (sums of) representations of the form $I \otimes P$, where $I$~is injective over $\field Q$ and $P$ is projective over $\field \vec A_n$. For indecomposable $P$ we may depict this representation as
    \[
        0 = \cdots = 0 \to I = \cdots = I.
    \]

    The next step is completely formal: Since $\modni \field Q \simeq \proj \Gamma$, we obtain an equivalence
    \[
        \{ M \in \mod \field Q \otimes \field \vec A_n \mid M_{\field Q} \in \modni \field Q \} \simeq \{ M \in \mod \Gamma \vec A_n \mid M_{\Gamma} \in \proj \Gamma \}.
    \]
    Thus we are in the setup of \cref{cor.An}, and obtain
    \[
        \frac{ \{ M \in \mod \Gamma \vec A_n \mid M_{\Gamma} \in \proj \Gamma \} }{ \mbT' } \simeq \mod \Gamma \vec A_{n-1},
    \]
    where $\mbT'$ consists of direct sums of representations
    \[
        P = \cdots = P \to 0 = \cdots = 0
    \]
    with $P \in \proj \Gamma$.

    Thus we complete the proof by observing that the representations in $\mbT'$ correspond to representations of the form
    \[
        [X = \cdots = X \to 0 = \cdots = 0] \in \rep_\field  (Q \otimes \vec A_n)
    \]
    with $X \in \modni \field Q$ under the second equivalence in this proof, and finally to
    \[
        [ X = \cdots = X \into I = \cdots = I] \in \rep_\field^{*, \m}  (Q \otimes \vec A_n)
    \]
    under the first equivalence above.
\end{proof}

\begin{ex}
    Let $Q = [1 \to 2 \longleftarrow 3]$.
    We are interested in the subcategory $\rep_\field^{*, \m} ( Q \otimes \vec A_n )$ of $\rep_\field ( Q \otimes \vec A_n )$. By the above corollary, up to finitely many indecomposables, $\rep_\field^{*, \m} ( Q \otimes \vec A_n )$ is equivalent to $\rep_\field ( \vec A_{n-1} \otimes Q )$. (Here we use that for this specific quiver $Q$, the costable Auslander algebra is equivalent to $\field Q$ again.)

    In particular, $\rep_\field^{*, \m} ( Q \otimes \vec A_n )$ has finitely many indecomposables if and only if $\field Q \otimes \field \vec A_{n-1}$ is representation finite. This is known to be the case if and only if $n \leq 3$ by \cite[Theorems~2.4 and~2.5]{MR1273693}.
\end{ex}

\begin{thm}
    \label{thm.funfact}
    Let $Q$ be a finite acyclic quiver.
    The category $\rep_\field^{*, \m} ( Q \otimes \vec A_2 )$ has finitely many indecomposables if and only if
    \begin{itemize}
        \item $Q$ is of Dynkin type $A_1$, $A_2$, $A_3$, or $A_4$, or
        \item $Q$ is of Dynkin type $A_5$ and the Loewy length of $\field Q$ is at least $4$.
    \end{itemize}
\end{thm}

\begin{proof}
    First consider the case that $Q$ is not Dynkin. Then $\modcat \field Q$ already contains infinitely many indecomposables, and we obtain infinitely many indecomposables in $\rep_\field^{*, \m} ( Q \otimes \field \vec A_2 )$ by considering isomorphisms in the $\vec A_2$-direction.

    Thus we may assume that $Q$ is Dynkin, and hence the ``In particular"-part of \cref{cor.auslander_alg} applies. Recall that this asserts that $\rep_\field^{*, \m} ( Q \otimes \vec A_2 ) / \mbX \simeq \modcat \Gamma$, where $\Gamma$ is the costable Auslander algebra of $\field Q$, that is the algebra given by the Auslander--Reiten quiver of $\field Q$ without the vertices corresponding to injective modules, and subject to mesh relations.

    Since the category $\mbX$ only contains finitely many indecomposables, it follows that $\rep_\field^{*, \m} ( Q \otimes \vec A_2 )$ has finitely many indecomposables if and only if $\mod \Gamma$ does, i.e., if $\Gamma$ is representation finite. So we only need to investigate for which Dynkin quivers $Q$ the costable Auslander algebra is representation finite.

    Note that for a full subquiver $Q'$ of $Q$, the corresponding costable Auslander algebra $\Gamma'$ is an idempotent subalgebra of $\Gamma$. In particular, if the costable Auslander algebra of $Q'$ is representation infinite then so is the costable Auslander algebra of $Q$. Therefore, it suffices to prove the ``only if'' direction for the minimal Dynkin quivers not in the list of finite cases given in the theorem.

    Using a double arrow to indicate all possible orientations, these are:
    \begin{itemize}
        \item type $D_4$, with any orientation; this is a full subquiver of any Dynkin quiver not of type $A$.
        \item $1 \longleftrightarrow 2 \to 3 \longleftarrow 4 \longleftrightarrow 5,\quad 1 \longleftarrow 2 \to 3 \to 4 \longleftarrow 5$, and their opposites; these are all orientations of $A_5$ of Loewy length at most $3$.
        \item $1 \to 2 \to 3 \to 4 \to 5 \longleftrightarrow 6$, and its opposite; these are the only orientations of $A_6$ not already covered by the previous point.
    \end{itemize}

    We investigate them one by one, but using the same strategy in all cases. The Auslander--Reiten quivers of $\modcat \field Q$ and $\modni \field Q$ are depicted in \cref{table.funfact_infinite} for all 4 cases.
    Note that the costable Auslander algebra $\Gamma$ is given as the path algebra of the Auslander--Reiten quiver of $\modni \field Q$, modulo mesh relations.
    \begin{table}
        \[ \begin{array}[t]{ccc}
            Q & \text{AR quiver of }\modcat \field Q & \text{AR quiver of }\modni \field Q \\
            \begin{tikzpicture}[scale=.7]
                \node (1) at (0,2) {$1$};
                \node (2) at (0,1) {$2$};
                \node (3) at (0,0) {$3$};
                \node (4) at (1,1) {$4$};
                \path[->,font=\scriptsize]
                    (1) edge [<->] (4)
                    (2) edge [<->] (4)
                    (3) edge [<->] (4);
            \end{tikzpicture}
            &
            \begin{tikzpicture}[scale=.7]
                \node (11) at (0,0) {$\circ$};
                \node (21) at (0,1) {$\circ$};
                \node (31) at (0,2) {$\circ$};
                \node (41) at (1,1) {$\circ$};
                \node (12) at (2,0) {$\circ$};
                \node (22) at (2,1) {$\circ$};
                \node (32) at (2,2) {$\circ$};
                \node (42) at (3,1) {$\circ$};
                \node (13) at (4,0) {$\circ$};
                \node (23) at (4,1) {$\circ$};
                \node (33) at (4,2) {$\circ$};
                \node (43) at (5,1) {$\circ$};
                \node (14) at (6,0) {$\circ$};
                \node (24) at (6,1) {$\circ$};
                \node (34) at (6,2) {$\circ$};
            \path[->,font=\scriptsize]
                (11) edge [dashed] (41)
                (21) edge [dashed] (41)
                (31) edge [dashed] (41)
                (41) edge (12)
                (41) edge (22)
                (41) edge (32)
                (12) edge (42)
                (22) edge (42)
                (32) edge (42)
                (42) edge (13)
                (42) edge (23)
                (42) edge (33)
                (13) edge (43)
                (23) edge (43)
                (33) edge (43)
                (43) edge [dashed] (14)
                (43) edge [dashed] (24)
                (43) edge [dashed] (34);
            \end{tikzpicture}
            &
            \begin{tikzpicture}[scale=.7]
                \node (11) at (0,0) {$\circ$};
                \node (21) at (0,1) {$\circ$};
                \node (31) at (0,2) {$\circ$};
                \node (41) at (1,1) {$\bullet$};
                \node (12) at (2,0) {$\circ$};
                \node (22) at (2,1) {$\circ$};
                \node (32) at (2,2) {$\circ$};
                \node (42) at (3,1) {$\bullet$};
                \node (13) at (4,0) {$\circ$};
                \node (23) at (4,1) {$\circ$};
                \node (33) at (4,2) {$\circ$};
                \path[->,font=\scriptsize]
                    (11) edge [dashed] (41)
                    (21) edge [dashed] (41)
                    (31) edge [dashed] (41)
                    (41) edge (12)
                    (41) edge (22)
                    (41) edge (32)
                    (12) edge (42)
                    (22) edge (42)
                    (32) edge (42)
                    (42) edge [dashed] (13)
                    (42) edge [dashed] (23)
                    (42) edge [dashed] (33);
            \end{tikzpicture}
            \\
            \begin{tikzpicture}[scale=.7]
                \node (1) at (0,4) {$1$};
                \node (2) at (0,3) {$2$};
                \node (3) at (0,2) {$3$};
                \node (4) at (0,1) {$4$};
                \node (5) at (0,0) {$5$};
                \path[->,font=\scriptsize]
                    (1) edge [<->] (2)
                    (2) edge (3)
                    (4) edge (3)
                    (4) edge [<->] (5);
            \end{tikzpicture}
            &
            \begin{tikzpicture}[scale=.7]
                \node (31) at (2,2) {$\circ$};
                \node (21) at (3,3) {$\circ$};
                \node (11) at (2,4) {$\circ$};
                \node (52) at (2,0) {$\circ$};
                \node (42) at (3,1) {$\circ$};
                \node (32) at (4,2) {$\circ$};
                \node (22) at (5,3) {$\circ$};
                \node (12) at (4,4) {$\circ$};
                \node (53) at (4,0) {$\circ$};
                \node (43) at (5,1) {$\circ$};
                \node (33) at (6,2) {$\circ$};
                \node (23) at (7,3) {$\circ$};
                \node (13) at (6,4) {$\circ$};
                \node (54) at (6,0) {$\circ$};
                \node (44) at (7,1) {$\circ$};
                \node (14) at (8,4) {$\circ$};
                \node (55) at (8,0) {$\circ$};
                \path[->,font=\scriptsize]
                    (31) edge (21)
                    (11) edge [dashed] (21)
                    (31) edge (42)
                    (21) edge (32)
                    (21) edge (12)
                    (52) edge [dashed] (42)
                    (42) edge (32)
                    (32) edge (22)
                    (12) edge (22)
                    (42) edge (53)
                    (32) edge (43)
                    (22) edge (33)
                    (22) edge (13)
                    (53) edge (43)
                    (43) edge (33)
                    (33) edge (23)
                    (13) edge (23)
                    (43) edge (54)
                    (33) edge (44)
                    (23) edge [dashed] (14)
                    (54) edge (44)
                    (44) edge [dashed] (55);
            \end{tikzpicture}
            &
            \begin{tikzpicture}[scale=.7]
                \node (31) at (2,2) {$\circ$};
                \node (21) at (3,3) {$\bullet$};
                \node (11) at (2,4) {$\circ$};
                \node (52) at (2,0) {$\circ$};
                \node (42) at (3,1) {$\bullet$};
                \node (32) at (4,2) {$\bullet$};
                \node (22) at (5,3) {$\bullet$};
                \node (12) at (4,4) {$\circ$};
                \node (53) at (4,0) {$\circ$};
                \node (43) at (5,1) {$\bullet$};
                \node (13) at (6,4) {$\circ$};
                \node (54) at (6,0) {$\circ$};
                \path[->,font=\scriptsize]
                    (31) edge (21)
                    (11) edge [dashed] (21)
                    (31) edge (42)
                    (21) edge (32)
                    (21) edge (12)
                    (52) edge [dashed] (42)
                    (42) edge (32)
                    (32) edge (22)
                    (12) edge (22)
                    (42) edge (53)
                    (32) edge (43)
                    (22) edge [dashed] (13)
                    (53) edge (43)
                    (43) edge [dashed] (54);
            \end{tikzpicture}
            \\
            \begin{tikzpicture}[scale=.7]
                \node (1) at (0,4) {$1$};
                \node (2) at (0,3) {$2$};
                \node (3) at (0,2) {$3$};
                \node (4) at (0,1) {$4$};
                \node (5) at (0,0) {$5$};
                \path[->,font=\scriptsize]
                    (2) edge (1)
                    (2) edge (3)
                    (3) edge (4)
                    (5) edge (4);
            \end{tikzpicture}
            &
            \begin{tikzpicture}[scale=.7]
                \node (41) at (1,1) {$\circ$};
                \node (31) at (2,2) {$\circ$};
                \node (21) at (3,3) {$\circ$};
                \node (11) at (2,4) {$\circ$};
                \node (52) at (2,0) {$\circ$};
                \node (42) at (3,1) {$\circ$};
                \node (32) at (4,2) {$\circ$};
                \node (22) at (5,3) {$\circ$};
                \node (12) at (4,4) {$\circ$};
                \node (53) at (4,0) {$\circ$};
                \node (43) at (5,1) {$\circ$};
                \node (33) at (6,2) {$\circ$};
                \node (13) at (6,4) {$\circ$};
                \node (54) at (6,0) {$\circ$};
                \node (44) at (7,1) {$\circ$};
                \path[->,font=\scriptsize]
                    (41) edge (31)
                    (31) edge (21)
                    (11) edge (21)
                    (41) edge (52)
                    (31) edge (42)
                    (21) edge (32)
                    (21) edge (12)
                    (52) edge (42)
                    (42) edge (32)
                    (32) edge (22)
                    (12) edge (22)
                    (42) edge (53)
                    (32) edge (43)
                    (22) edge (33)
                    (22) edge (13)
                    (53) edge (43)
                    (43) edge (33)
                    (43) edge (54)
                    (33) edge (44)
                    (54) edge (44);
            \end{tikzpicture}
            &
            \begin{tikzpicture}[scale=.7]
                \node (41) at (1,1) {$\circ$};
                \node (31) at (2,2) {$\bullet$};
                \node (21) at (3,3) {$\bullet$};
                \node (11) at (2,4) {$\bullet$};
                \node (52) at (2,0) {$\bullet$};
                \node (42) at (3,1) {$\bullet$};
                \node (32) at (4,2) {$\bullet$};
                \node (12) at (4,4) {$\bullet$};
                \node (53) at (4,0) {$\bullet$};
                \node (43) at (5,1) {$\circ$};
                \path[->,font=\scriptsize]
                    (41) edge (31)
                    (31) edge (21)
                    (11) edge (21)
                    (41) edge (52)
                    (31) edge (42)
                    (21) edge (32)
                    (21) edge (12)
                    (52) edge (42)
                    (42) edge (32)
                    (42) edge (53)
                    (32) edge (43)
                    (53) edge (43);
            \end{tikzpicture}
            \\
            \scalebox{.7}{
                \begin{tikzpicture}[scale=.7]
                    \node (1) at (0,4) {$1$};
                    \node (2) at (0,3) {$2$};
                    \node (3) at (0,2) {$3$};
                    \node (4) at (0,1) {$4$};
                    \node (5) at (0,0) {$5$};
                    \node (6) at (0,-1) {$6$};
                    \path[->,font=\scriptsize]
                        (1) edge (2)
                        (2) edge (3)
                        (3) edge (4)
                        (4) edge (5)
                        (5) edge [<->] (6);
                \end{tikzpicture}
            }
            &
            \scalebox{.7}{
                \begin{tikzpicture}[scale=.7]
                    \node (61) at (-1,-1) {$\circ$};
                    \node (51) at (0,0) {$\circ$};
                    \node (41) at (1,1) {$\circ$};
                    \node (31) at (2,2) {$\circ$};
                    \node (21) at (3,3) {$\circ$};
                    \node (11) at (4,4) {$\circ$};
                    \node (62) at (1,-1) {$\circ$};
                    \node (52) at (2,0) {$\circ$};
                    \node (42) at (3,1) {$\circ$};
                    \node (32) at (4,2) {$\circ$};
                    \node (22) at (5,3) {$\circ$};
                    \node (12) at (6,4) {$\circ$};
                    \node (63) at (3,-1) {$\circ$};
                    \node (53) at (4,0) {$\circ$};
                    \node (43) at (5,1) {$\circ$};
                    \node (33) at (6,2) {$\circ$};
                    \node (54) at (6,0) {$\circ$};
                    \node (64) at (5,-1) {$\circ$};
                    \node (44) at (7,1) {$\circ$};
                    \node (65) at (7,-1) {$\circ$};
                    \node (55) at (8,0) {$\circ$};
                    \node (66) at (9,-1) {$\circ$};
                    \path[->,font=\scriptsize]
                        (61) edge [dashed] (51)
                        (51) edge (41)
                        (41) edge (31)
                        (31) edge (21)
                        (21) edge (11)
                        (51) edge (62)
                        (41) edge (52)
                        (31) edge (42)
                        (21) edge (32)
                        (11) edge (22)
                        (62) edge (52)
                        (52) edge (42)
                        (42) edge (32)
                        (32) edge (22)
                        (22) edge [dashed] (12)
                        (52) edge (63)
                        (42) edge (53)
                        (32) edge (43)
                        (22) edge (33)
                        (63) edge (53)
                        (53) edge (43)
                        (43) edge (33)
                        (53) edge (64)
                        (43) edge (54)
                        (33) edge (44)
                        (64) edge (54)
                        (54) edge (44)
                        (54) edge (65)
                        (44) edge (55)
                        (65) edge (55)
                        (55) edge (66);
                \end{tikzpicture}
            }
            &
            \scalebox{.7}{
                \begin{tikzpicture}[scale=.7]
                    \node (61) at (-1,-1) {$\circ$};
                    \node (51) at (0,0) {$\circ$};
                    \node (41) at (1,1) {$\circ$};
                    \node (31) at (2,2) {$\bullet$};
                    \node (21) at (3,3) {$\circ$};
                    \node (11) at (4,4) {$\circ$};
                    \node (62) at (1,-1) {$\circ$};
                    \node (52) at (2,0) {$\bullet$};
                    \node (42) at (3,1) {$\bullet$};
                    \node (32) at (4,2) {$\bullet$};
                    \node (63) at (3,-1) {$\circ$};
                    \node (53) at (4,0) {$\bullet$};
                    \node (43) at (5,1) {$\circ$};
                    \node (54) at (6,0) {$\circ$};
                    \node (64) at (5,-1) {$\circ$};
                    \node (65) at (7,-1) {$\circ$};
                    \path[->,font=\scriptsize]
                        (61) edge [dashed] (51)
                        (51) edge (41)
                        (41) edge (31)
                        (31) edge (21)
                        (21) edge [dashed] (11)
                        (51) edge (62)
                        (41) edge (52)
                        (31) edge (42)
                        (21) edge (32)
                        (62) edge (52)
                        (52) edge (42)
                        (42) edge (32)
                        (52) edge (63)
                        (42) edge (53)
                        (32) edge (43)
                        (63) edge (53)
                        (53) edge (43)
                        (53) edge (64)
                        (43) edge (54)
                        (64) edge (54)
                        (54) edge (65);
                \end{tikzpicture}
            }
        \end{array} \]
        \caption{The minimal cases of Dynkin quivers $Q$ with representation infinite costable Auslander algebra, as in the proof of \cref{thm.funfact}. (Dashed parts might or might not be present, depending on the precise choice of orientation of the arrows.)
        }
        \label{table.funfact_infinite}
    \end{table}
    In all cases, we marked a subset of the vertices in the costable Auslander--Reiten quiver by solid dots. The idempotent subalgebras of the costable Auslander algebras given by this subset of vertices are, in the four cases:
    \begin{itemize}
        \item The Kronecker algebra
            $
                \begin{tikzpicture}[scale=.75]
                    \node (0) at (0,0.2) { $\circ$ };
                    \node (1) at (1,0.2) { $\circ$ };
                    \path[->,font=\scriptsize]
                        (0) edge [bend left] (1)
                        (0) edge [bend right] (1);
                \end{tikzpicture}
            $.
        \item The algebra of type $\widetilde{D}_4$
            $
                \begin{tikzpicture}[scale=.75,baseline=0]
                    \node (0) at (0,0.7) { $\circ$ };
                    \node (1) at (0,-0.3) { $\circ$ };
                    \node (2) at (1,0.2) { $\circ$ };
                    \node (3) at (2,0.7) { $\circ$ };
                    \node (4) at (2,-0.3) { $\circ$ };
                    \path[->,font=\scriptsize]
                        (0) edge (2)
                        (1) edge (2)
                        (2) edge (3)
                        (2) edge (4);
                \end{tikzpicture}
            $.
        \item The algebra given by the quiver with relations as on the left below (zero and commutativity relations indicated by dotted lines). This algebra has the ``frame'' as depicted on the right  in the terminology of Happel and Vossieck's \cite{MR701205}. In particular it appears in their list of tame concealed algebras of type $\widetilde{E}_7$.
            \[
                \begin{tikzpicture}[scale=.75,baseline=0]
                    \node (0) at (0,0.7) { $\circ$ };
                    \node (1) at (0,-0.3) { $\circ$ };
                    \node (2) at (1,0.2) { $\circ$ };
                    \node (3) at (2,0.7) { $\circ$ };
                    \node (4) at (2,-0.3) { $\circ$ };
                    \node (5) at (3,0.2) { $\circ$ };
                    \node (6) at (4,0.7) { $\circ$ };
                    \node (7) at (4,-0.3) { $\circ$ };
                    \path[->,font=\scriptsize]
                        (0) edge (2)
                        (2) edge (1)
                        (3) edge (2)
                        (2) edge (4)
                        (3) edge (5)
                        (5) edge (4)
                        (6) edge (5)
                        (5) edge (7)
                        (0) edge [-, dotted,thick] (1)
                        (3) edge [-, dotted,thick] (4)
                        (6) edge [-, dotted,thick] (7);
                \end{tikzpicture}
                \text{ has frame }
                \begin{tikzpicture}[scale=.75,baseline=0]
                    \node (0) at (-0.5,0.2) { $\circ$ };
                    \node (1) at (0.25,0.2) { $\circ$ };
                    \node (2) at (1,0.2) { $\circ$ };
                    \node (3) at (2,0.7) { $\circ$ };
                    \node (4) at (2,-0.3) { $\circ$ };
                    \node (5) at (3,0.2) { $\circ$ };
                    \node (6) at (3.75,0.2) { $\circ$ };
                    \node (7) at (4.5,0.2) { $\circ$ };
                    \path[->,font=\scriptsize]
                        (0) edge [-] (1)
                        (2) edge [-] (1)
                        (3) edge (2)
                        (2) edge (4)
                        (3) edge (5)
                        (5) edge (4)
                        (6) edge [-] (5)
                        (6) edge [-] (7)
                        (3) edge [-, dotted,thick] (4);
                \end{tikzpicture}
            \]
        \item The algebra of type $\widetilde{D}_4$, as above.
    \end{itemize}
    In particular these algebras are all representation infinite, and hence so are the costable Auslander algebras from \cref{table.funfact_infinite}.

    For the ``if'' part of the statement, it suffices to show that all maximal quivers in the list actually give rise to finite categories $\rep_\field^{*, \m} ( Q \otimes \vec A_2 )$. These are $1 \to 2 \longleftarrow 3 \to 4$, $1 \to 2 \to 3 \to 4 \longleftarrow 5$, and $1 \to 2 \to 3 \to 4 \to 5$. The finiteness of  $\rep_\field^{*, \m} ( Q \otimes \vec A_2 )$ for these three cases can be checked by calculating explicitly the Auslander--Reiten quivers.
\end{proof}

\subsection{Rectangular grid representations with monomorphic structure maps}

Our final explicit quiver application concerns the case that we have monomorphisms in all directions. It differs slightly from all the above, in that it does not build on \cref{thm.QA}, but rather on the underlying \cref{cor:equiv_ctt}.

\begin{cor}
    \label{cor.mono_mono}
    In the subcategory $\rep_\field^{\m, \m} ( \vec A_m \otimes \vec A_n )$ of $\rep_\field ( \vec A_m \otimes \vec A_n )$, let $\mbT$ the subcategory formed by finite direct sums of modules of the form
    \[
        T_{i,j} = \field_{\{(x, y) \mid x > i \text{ or } y > j\}} .
    \]
    Then
    \[
        \frac{\rep_\field^{\m, \m} ( \vec A_m \otimes \vec A_n ) }{\mbT} \simeq \rep_\field ( \vec A_{m-1} \otimes  \vec A_{n-1} ) .
    \]
    In particular, $\rep_\field^{\m, \m} ( \vec A_m \otimes  \vec A_n )$ contains finitely many indecomposables precisely in the cases
    \begin{itemize}
        \item $m \leq 2$ or $n \leq 2$,
        \item $(m,n) \in \{ (3,3), (3,4), (3,5), (4,3), (5,3)\}$.
    \end{itemize}
    It is of tame representation type precisely in the cases
    \begin{itemize}
        \item $(m,n) \in \{ (3,6), (4,4), (6,3)\}$.
    \end{itemize}
    In all other cases it is of wild representation type.
\end{cor}

\begin{figure}
    \centering
    \begin{tikzpicture}
        \node at (-.3, 2) {1};
        \node at (-.3, 1) {$j$};
        \node at (-.3, 0) {$n$};
        \node at (0, -.3) {1};
        \node at (2, -.3) {$i$};
        \node at (3,-.3) {$m$};
        \draw (0,0) rectangle (3,2);
        \draw [fill=lightergray] (0,0) -- (0,1) -- (2,1) -- (2,2) -- (3,2) -- (3,0) -- cycle;
    \end{tikzpicture}
    \caption{Diagramatic depiction of the module $T_{i,j} = \field_{\{(x, y) \mid x > i \text{ or } y > j\}}$.}
\end{figure}

\begin{proof}
    The category $\mcA = \rep_\field ( \vec A_m \otimes \vec A_n )$ has projective and injective modules
    \begin{align*}
        P_{i,j} &= \field_{\{i,\ldots,m\}\times\{j,\ldots,n\}} \\
        I_{i,j} &= \field_{\{1,\ldots,i\}\times\{1,\ldots,j\}}.
    \end{align*}
    Thus $T_{i,j}$ has projective resolution
    \[
        P_{i+1,j+1} \into \underbrace{P_{i+1,1} \oplus P_{1,j+1}}_{T_{i,n} \oplus T_{n,j}} \onto T_{i,j},
    \]
    so $\mathbb T$ satisfies the final two properties of being weakly tilting, \cref{def.tilting}.

    Moreover, $T_{i,j}$ has injective resolution
    \[
        T_{i,j} \into I_{m,n} \onto I_{i,j}.
    \]
    If follows that
    \begin{align*}
        \Ext^1_\mcA( M, T_{i,j} ) & = \coker [ \Hom_\mcA(M, I_{m,n}) \to \Hom_\mcA(M, I_{i,j}) ] \\
            & = D \ker [M_{i,j} \to M_{m,n}],
    \end{align*}
    where $D$ denotes the vector space dual.
    Thus ${}^{\perp_1} \mbT = \rep_\field^{\m, \m} ( \vec A_m \otimes \vec A_n )$.  In particular $\mathbb T \subseteq \rep_\field^{\m, \m} ( \vec A_m \otimes \vec A_n )$, so $\Ext^1_\mcA(\mathbb T, \mathbb T) = 0$.  It thus follows that $\mathbb T$ is weakly tilting, and so by \cref{prop.noetherian_tilting}, tilting.

    Applying \cref{cor:equiv_ctt}, we get
    \[
        \frac{ \{ M \in {}^{\perp_1} \mbT \mid \pdim M \leq 1 \} }{\mbT} \simeq \mbT^{\perp}.
    \]

    Since any module in $\rep_\field^{\m, \m} ( \vec A_m \otimes \vec A_n )$ is a submodule of (copies of) the projective module $P_{1,1}$,
    and the global dimension of $\field \vec A_m \otimes \field \vec A_n$ is $2$, these modules automatically have projective dimension at most $1$. Thus we have shown that
    \[
        \{ M \in {}^{\perp_1} \mbT \mid \pdim M \leq 1 \} = \rep_\field^{\m, \m} ( \vec A_m \otimes \vec A_n ).
    \]
    Now we explicitly describe $\mbT^{\perp}$. Since $P_{i,1} = T_{i-1,n}$ and $P_{1,j} = T_{m,j-1}$, all modules in $\mbT^{\perp}$ vanish in the vertices of the forms $(i,1)$ and $(1,j)$. On the other hand, all modules in $\mbT$ are generated in these vertices, so there cannot be any nonzero maps from $\mbT$ to modules vanishing in these vertices. In other words, we have shown that
    \[
        \mbT^{\perp} = \{ M \in \rep_\field ( \vec A_m \otimes \vec A_n ) \mid \forall i \colon M_{i,1} = 0 \text{ and } \forall j \colon M_{1,j} = 0 \}.
    \]
    Clearly this latter subcategory may be identified with $\rep_\field ( \vec A_{m-1} \otimes \vec A_{n-1} )$.  Thus we have established the equivalence of the categories.

    In particular $\rep_\field^{\m, \m} ( \vec A_m \otimes \vec A_n )$ and $\rep_\field ( \vec A_{m-1} \otimes \vec A_{n-1} )$ have the same representation type.  An application of \cref{thm.types_for_grids} finishes the proof.
\end{proof}

\section*{Acknowledgements}
UB and MB have been supported by the DFG Collaborative Research Center TRR109 \emph{Discretization in Geometry and Dynamics}.  SO has been supported by Norwegian Research Council project 250056, ``Representation theory via subcategories''.  JS has been partially supported by Norwegian Research Council project 231000, ``Clusters, combinatorics and computations in algebra''.
We thank the anonymous referee for carefully reading the paper and providing valuable comments.

\bibliography{cotorsion-torsion}
\bibliographystyle{amsplain}

\end{document}